\documentclass[11pt]{article}
\usepackage{amssymb} %for Blackboard bold etc
\usepackage{graphicx} %for including eps graphics
\usepackage{amsthm} %for defining theorems
\usepackage{color}
\usepackage{amsmath, stackrel}
\usepackage{hyperref}
\usepackage{dsfont}
\usepackage[parfill]{parskip}
\usepackage[a4paper, total={6in, 8in}]{geometry}
\usepackage{bbm}
\usepackage{cleveref} % Must be last referencing package loaded
\usepackage{mathtools}
\usepackage[utf8]{inputenc}

\usepackage{subfloat}

\usepackage{setspace}
\usepackage{background}

\usepackage{enumerate}

\definecolor{textcolor}{HTML}{0A75A8}
\newcommand\Text{}

\SetBgColor{textcolor}
\SetBgOpacity{1}
\SetBgAngle{90}
\SetBgPosition{current page.center}
\SetBgVshift{-0.36\textwidth}
\SetBgScale{1.8}
\SetBgContents{\sffamily\Text}
\newcommand{\mathd}{\mathrm{d}}

\newcommand{\ssup}[1] {{\scriptscriptstyle{({#1}})}}

\newcommand{\R}{\mathbb R}

\newcommand{\E}{\mathbb E}
\renewcommand{\P}{\mathbb P}

\newcommand{\x}{\mathbf{x}}
\newcommand{\y}{\mathbf{y}}
\newcommand{\z}{\mathbf{z}}
\newcommand{\0}{\mathbf{0}}
\renewcommand{\iota}{\varkappa}

\renewcommand{\phi}{\varphi}

\renewcommand{\t}{{\tau}}

\newcommand{\Z}{\mathbb Z}

\renewcommand{\P}{\mathbb P}
\newcommand{\X}{\mathcal X}

\newcommand{\abs}[1]{\left| #1 \right|}

\newcommand{\gaus}[1]{\lfloor #1 \rfloor}
\newcommand{\suag}[1]{\lceil #1 \rceil}
\newcommand{\conn}{\leftrightarrow}
\DeclarePairedDelimiter\set{\lbrace}{\rbrace}
\DeclareMathOperator{\Cat}{Cat}
\DeclareMathOperator{\diam}{diam}

\title{Chemical distance in geometric random graphs with long edges and scale-free degree distribution\\[-2mm]} 
\author{Peter Gracar, Arne Grauer, and Peter M\"orters\footnote{Corresponding author, E-Mail: moerters@math.uni-koeln.de}\\
{\sl\small Mathematisches Institut, Universit\"at zu K\"oln, Germany}}
\date{}

\usepackage{todonotes}
\usepackage{mathrsfs}

\theoremstyle{plain} %theorems
\newtheorem{prop}{Proposition}[section]
\newtheorem{lemma}{Lemma}[section]
\newtheorem{theorem}{Theorem}[section]

\newtheorem{assum}{Assumption}[section]

\theoremstyle{definition} %definitions

\theoremstyle{remark} %notes and remarks
\newtheorem{remark}{Remark}[section]

\begin{document}
\maketitle

\vspace{-8mm}
\abstract{\noindent We study geometric random graphs defined on the points of a Poisson process in \mbox{$d$-dimensional space}, which additionally carry independent random marks. Edges are established at random using the marks of the endpoints and the distance between points in a flexible way. Our framework includes the soft Boolean model (where marks play the role of radii of balls centred in the vertices),  
a version of spatial preferential attachment (where marks play the role of birth times), and a whole range of other graph models 
with scale-free degree distributions  and edges spanning large distances. 
In this versatile framework we  give sharp criteria for absence of ultrasmallness of the graphs and in the ultrasmall regime establish a limit theorem for the chemical distance of two points.  Other than in the mean-field scale-free network models the boundary of the ultrasmall regime depends not only on the power-law exponent of the degree distribution  but also on the spatial embedding of the graph, quantified by the rate of decay of the probability of an edge connecting typical points in terms of their spatial distance.}

{\footnotesize \tableofcontents}
\newpage
\section{Introduction}
\subsection{Background}

An important topic in percolation theory and, more generally, the theory of geometrically embedded random graphs, is the comparison of Euclidean distances of two points with their graph distance, often called chemical distance. Starting with the work of Grimmett and Marstrand~\cite{GriM1990}, this problem has been studied for Bernoulli percolation, for example by Antal and Pisztora~\cite{AntP1996} and Garet and Marchand \cite{GarM2004, GarM2007}, but also for models with long range interactions, such as random interlacements, see \v{C}ern\'y and Popov~\cite{CerP2012}, its vacant set and the Gaussian free field, see~Drewitz et al.~\cite{DreRB2014}. In the supercritical phase of these models Euclidean and chemical  distance of points on the unbounded connected component are typically of comparable order when the points are distant, see~\cite{DreRB2014} for general conditions for percolation models on $\mathbb{Z}^d$ to share this behaviour. The introduction of additional long edges can change this behaviour and the graph distance can be a power of the logarithm  or even an iterated logarithm of the 
Euclidean distance. In the latter case the graph is called \emph{ultrasmall}. The focus of this paper is to characterise ultrasmallness in geometric random graphs and provide 
a universal limit theorem for typical distances in such graphs. \bigskip

We briefly review what is known on this problem. A classical scenario is \emph{long-range percolation}. 
Here points $x,y$ of a Poisson process in $\R^d$ or of the lattice $\Z^d$ are connected independently with probability
$$p(x,y)=|x-y|^{-\delta d + o(1)},$$
for some $\delta>1$. Biskup~\cite{Bis2004, BisL2019} has shown that if $1<\delta<2$ then the chemical distance~is 
$$d(x,y)= (\log |x-y|)^{\Delta+o(1)},$$ with high probability as 
$x, y$ are fixed points on the infinite component with $|x-y|\to\infty$, where $\Delta=\frac{\log 2}{\log(2/\delta)}$.  
If $\delta>2$ it was shown by Berger~\cite{BenB2001} that the chemical distance is at least linear  in the Euclidean distance 
and for $\delta=2$ there is recent progress by Ding and Sly~\cite{DinS2015}, but in both cases 
the precise asymptotics is still an open problem.  In general, ultrasmallness cannot occur in long-range percolation models. \bigskip

Ultrasmallness is however a well established phenomenon in \emph{scale-free networks}. These networks are typically not modelled as spatial graphs, so to compare the results to 
our scenario we restrict the graph to
the vertices inside a ball of radius $R$, which now contains $N$ lattice or Poisson points, with $N$  of order $R^d$. The mean-field nature of these models is reflected in the fact that connection probabilities do not depend on the spatial position of these points. Instead, points carry independent uniform marks and connections between points are established independently given the marks, with a probability 
$1 \wedge \frac1N g(s,t)$ depending on the marks $s,t$ of the vertices at the ends of a potential edge. Dependencies of interest are, for example, \medskip
\begin{itemize}
\item[(i)]  $g(s,t)=s^{-\gamma }t^{-\gamma},$\smallskip
 \item[(ii)] $g(s,t)=(s\vee t)^{-\gamma }(s \wedge t)^{\gamma-1},$\smallskip
 \item[(iii)] $g(s,t)= (s^{-\gamma/d}+ t^{-\gamma/d})^{d}.$
 \end{itemize}\bigskip
For all these examples, the graphs have scale-free degree distributions with power-law exponent $\tau=1+\frac1\gamma$. When $\gamma<\frac12$ (or, equivalently, $\tau>3$) 
the chemical distance of two randomly chosen points $x,y$ in the largest component is of order $\log N$ or, equivalently, $\log |x-y|$, see Bollobas et al.~\cite{BolJR2007}.
If however $\gamma>\frac12$ (or, equivalently, $2<\tau<3$), then the graph is ultrasmall and there is a universal limit theorem for the chemical distance of two randomly chosen potints $x,y$, namely 
\begin{equation}\label{limitclassical}
\frac{d(x,y)}{\log\log(|x-y|)}\longrightarrow  \frac{c}{\log \tfrac\gamma{1-\gamma}}, \mbox{ with high probability as $R\to\infty$,}
\end{equation}
where $c=2$ for (i) and $c=4$ for (ii), (iii), see Dommers et al.~\cite{DomHH2010}, van der Hofstad et al.~\cite{vdHHZ2007} and Norros and Reittu~\cite{NorR2006} for the existence of an ultrasmall phase and Dereich et al.~\cite{DerMM2012} for general lower bounds that match the upper bounds in the ultrasmall phase in all those examples.
\bigskip
\pagebreak[3]

Looking at spatially embedded graphs with a scale-free degree distribution, Deijfen et al.~\cite{DeivHH2013}, Deprez et al.~\cite{DepW2019} and Bringmann et al.~\cite{BriKL2018} investigated a range of spatial models where points are endowed with weights, which are heavy-tailed random variables corresponding loosely to negative powers $t^{-\gamma}$ of uniformly chosen marks~$t$. The connection probability of two marked points depends on the \emph{product} of the weights and the spatial distance of the points, which is the case in models like scale-free percolation and hyperbolic random graphs. %Bringmann et al.~\cite{BriKL2018} 
Behaviour analogous to kernel (i) in the non-spatial case is identified in~\cite{BriKL2018} for these models, namely that the transition between ultrasmall and small world behaviour occurs at $\gamma=\frac12$ (equivalently, $\tau=3$) and in the former case a limit theorem as in~\eqref{limitclassical} with $c=2$~holds.
\bigskip 

We shall see in the present paper that not only the proof techniques but also the results of \cite{DeivHH2013}, \cite{DepW2019} and~\cite{BriKL2018} depend crucially on the fact that connections are considered that depend on the weights of points by taking the product.  In fact, the situation changes radically when other, equally natural, ways of connecting vertices are considered, and we shall see that the novel behaviour that we unlock in this paper is also of a universal nature. We now discuss two natural examples, which constitute our main motivation. In both cases the vertices of the graph are the points of a standard Poisson process in~$\mathbb R^d$ and every point is endowed with an independent mark, which is uniformly distributed on the unit interval~$(0,1)$.
\bigskip

In the \emph{Boolean (graph) model} on $\R^d$  the points carry random radii, which can be derived from the uniform marks~$t$, for example  as $t^{-\gamma/d}$. In the hard version of the model two points are connected by an edge if the balls around them with the associated random radii intersect. In the more powerful soft version of the Boolean model independent, identically distributed positive random variables $X=X(x,y)$ are associated with every unordered pair of vertices $\{x,y\}$ and a connection is made iff
$$\frac{|x-y|}{s^{-\gamma/d}+t^{-\gamma/d}} \leq X,$$
where $s,t$ are the marks of the vertices. The choice $X=1$ corresponds to the hard Boolean model, while the choice of $\gamma=0$ and a heavy-tailed random variable $X$ with decay
$$\mathbb P(X>r) \asymp r^{-\delta d} \mbox{ as } r\to \infty,$$
for some $\delta>1$, replicates the long-range percolation model. While neither of these boundary cases is ultrasmall, we show that a choice of $\gamma\in(0,1)$ and $\delta>1$ gives \smallskip
\begin{itemize}
\item \emph{ultrasmallness} if \smash{$\gamma>\frac{\delta}{\delta+1}$ } but, %surprisingly, 
\smallskip
\item \emph{no ultrasmallness}  if  \smash{$\gamma<\frac{\delta}{\delta+1}$}. \medskip
\end{itemize}
Note that this boundary depends not only on the power-law exponent
of the degree distribution, which is $\tau=1+\frac1\gamma$, but also on $\delta$, which is a geometric quantity related to the decay in the presence of long edges between typical vertices.
In particular ultrasmallness does \emph{not} occur when the variance of the degree distribution becomes infinite, but at a threshold that depends on spatial correlations influencing the graph topology beyond the degree distribution, a feature that is not present in the scale-free percolation or hyperbolic random graph models.
%\smallskip
In the ultrasmall  case we also get a different form of the limit theorem for the chemical distance, namely 
\begin{equation}\label{limitnew}
\frac{d(x,y)}{\log\log(|x-y|)}\longrightarrow  \frac{4}{\log \tfrac\gamma{\delta(1-\gamma)}}\, \mbox{ with high probability as $|x-y|\to\infty$,}
\end{equation}
where the dependence of the limiting constant on $\delta$ is another novel feature.
\medskip

\pagebreak[3]
In our second example we look at the \emph{age-based random connection model}, which was introduced in Gracar et al.~\cite{GraGLM2019}. 
Here the mark of a vertex is considered to be its birth time so that the model is intrinsically dynamical. At its birth time~$t$ a vertex is connected 
to all vertices born previously with a probability
$$\varphi\bigg( \frac{t \, |x-y|^d}{(t/s)^{\gamma}} \bigg), $$
where $s<t$ is the birth-time of the older vertex and
$\varphi\colon(0,\infty)\to[0,1]$ is a non-increasing profile function.  As $(t/s)^{\gamma}$ is the asymptotic order of the expected degree at time $t$  of a vertex born at time  $s\downarrow 0$ this  infinite graph model mimics the behaviour of spatial preferential attachment networks~\cite{BarA99, JacM2015}. An upper bound for the chemical distance for spatial preferential attachment is given by Hirsch and M\"onch in~\cite{HirM2020}, but lower bounds are not known.
Our results show that, as in the soft Boolean model, we have in the age-dependent random connection model that ultrasmallness fails if
\smash{$\gamma<\frac{\delta}{\delta+1}$}. If $\gamma>\frac{\delta}{\delta+1}$ we get a lower bound matching that of~\cite{HirM2020}
and we get the precise asymptotics for the chemical distance as stated in~(\ref{limitnew}).% 
\bigskip%

The similarity in the behaviour of our examples is a strong hint that there is a large class of spatial graph models which displays universal behaviour markedly different from both the class of spatial scale-free graphs investigated in~\cite{BriKL2018} and the non-spatial scale-free models studied, for example, in~\cite{vdH}.  This idea is further supported by the recent paper by Gracar et al.~\cite{GraLM2021} which investigates the existence of a subcritical percolation phase and reveals the same regime boundary depending on the parameters $\gamma$ and~$\delta$.
In the present paper we explore this universality class of spatial scale-free random graphs by  providing general bounds for the chemical distance based only on upper and lower bounds on the connection probabilities between finitely many pairs of points. This approach is sufficiently flexible to yield the fine results described above for the entire range of models in this class,  including of course  both of the  examples described above. The main difficulty here is to produce lower bounds larger than those obtainable for the non-spatial scale-free models by making substantial use of the restrictions coming from the underlying Euclidean geometry. 
\medskip

\subsection{Framework}

Suppose $\mathscr G$ is a graph with vertex set given by the points of a Poisson process $\mathcal{X}$ of unit intensity on $\mathbb{R}^d \times (0,1)$.  We write the points 
of this process as $\x=(x,t)$ and refer to $x$ as {the} location and $t$ as the mark of the vertex~$\x$. Small marks indicate powerful vertices. We write $\x\sim \y$ if the vertices $\x$, $\y$ are connected by an edge in $\mathscr{G}$. \medskip

We denote by $\mathbb{P}_\mathcal{X}$ the law of $\mathscr{G}$ conditioned on the Poisson process $\mathcal{X}$ and by $\mathbb{P}_{\mathbf{x_1},\ldots, \mathbf{x_n}}$ the law of $\mathscr{G}$ conditioned on the event that $\mathbf{x_1},\ldots, \mathbf{x_n}$ are points of the Poisson process~$\mathcal{X}$.
%\medskip
%
The following assumption depends on parameters  $\delta>1$ and $0\leq \gamma<1$, it leads to \emph{lower bounds} on chemical distances in the graph.
\bigskip

\begin{assum}%[Assumption $U(\delta,\gamma, \kappa)$]
\label{ass:main1}
There exists $\kappa>0$ %$\beta>0$ 
such that, for every {finite} set of pairs of vertices $I\subset \mathcal{X}^2$ {in which each vertex appears at most twice}, we have
\begin{align*}
\P_{\mathcal{X}} \bigg(\bigcap_{(\x_i,\y_i)\in I} \set{\x_i\sim \y_i}\bigg)
\leq 
\prod_{(\x_i,\y_i)\in I}  \kappa \, (t_i\wedge s_i)^{-\delta\gamma} (t_i\vee s_i)^{\delta(\gamma-1)}\abs{x_i-y_i}^{-\delta d}
\end{align*}
where $\x_i = (x_i,t_i)$, $\y_i=(y_i,s_i)$.
\end{assum}
\bigskip

In Section~1.4 we shall see several natural examples of geometric random graphs {which satisfy}
Assumption~1.1. Note that the assumption does not include conditional independence of the 
events $\set{\x_i\sim \y_i}$, { which makes several classical tools, such as the BK-inequality,  unavailable in our proofs.} {Without the conditional independence one cannot give a precise description for the degree distribution. However, it is worth noting that Assumption~1.1 is formed in such a way that it implies the existence of a constant $C>0$ for which the \emph{expected} degree of a vertex with mark $t$ is smaller than~$Ct^{-\gamma}$.} The next assumption, 
which we use to give matching \emph{upper bounds} on chemical distances in the ultrasmall regime, however, does contain a conditional independence assumption.\bigskip

\begin{assum}\label{ass:main2}
Given $\mathcal{X}$ edges are drawn independently of each other and there exists $\alpha,\kappa>0$ such that, for every pair of vertices $\x=(x,t), \y=(y,s)\in \mathcal{X}$,

\[
\P_{\x,\y}\set{\x \sim \y} \geq \alpha\, \big(1\wedge \kappa \, (t \wedge s)^{-\delta\gamma} \abs{x-y}^{-\delta d}\big).
\]
\end{assum}
\bigskip

The weight dependent random connection model is a class of graphs introduced in~\cite{GraHMM2019, GraLM2021} as a general framework, which incorporates many (but not all) of our examples of spatial random graphs. In that context our assumptions roughly mean that the random graphs are stochastically dominated by the random connection model with preferential attachment kernel (Assumption 1.1) and dominate the  random connection model with min kernel (Assumption 1.2). {Note, that these models have a scale-free degree distribution with power-law exponent $\tau = 1 + \frac{1}{\gamma}$. Hence, as previously mentioned these examples deviate from the behaviour of non-spatial models and scale-free percolation in that the emergence of ultrasmallness does not depend only on the power-law exponent.}

\subsection{Statement of the main results}

We write $\x\stackrel{n} {\leftrightarrow}\y$ if there exists a path of length $n$
from $\x$ to $\y$ in $\mathscr{G}$, i.e. there exist $\mathbf{x}_1,\ldots,\mathbf{x}_{n-1} \in \mathscr{G}$  such that 
$$\mathbf{x} \sim \mathbf{x}_1\sim \ldots \sim\mathbf{x}_{n-1} \sim \mathbf{y}.$$
We denote by  $\x \leftrightarrow \y$ if $\x\stackrel{n} {\leftrightarrow} \y$ holds for some $n$, i.e. if $\x$ and $\y$ are in the same connected component in $\mathscr{G}$.
The graph distance, or chemical distance, is given by 
\[
\mathd(\mathbf{x},\mathbf{y}) = \min \set{n\in \mathbb{N} \colon \mathbf{x} \stackrel{n} {\leftrightarrow} \mathbf{y}}.
\]
Our main results identify the regime where $\mathscr G$ is ultrasmall, i.e. where the graph distance behaves like an iterated logarithm of the 
Euclidean distance. Moreover in this regime we provide a precise limit theorem for the behaviour of the graph distance of remote points. The first and foremost result in this context  are lower bounds for the chemical distance of two points at large Euclidean distance using only Assumption~1.1.
\bigskip

\pagebreak[3]

\begin{theorem}\label{mainthm}
Let $\mathscr{G}$ be a general geometric random graph which satisfies Assumption~\ref{ass:main1} for some $\gamma \in [0,1)$ and $\delta>1$.\medskip
\begin{enumerate}[(a)]
\item If $\gamma< \frac{\delta}{\delta+1}$, then $\mathscr{G}$ is {\bf not ultrasmall}, i.e. for $\x,\y\in \mathbb{R}^d\times (0,1)$, under $\P_{\x,\y}$, the distance $\mathd (\x,\y)$ is of larger order than $\log\log\abs{x-y}$ with high probability as $\abs{x-y}\to \infty$.\medskip
\item If $\gamma > \frac{\delta}{\delta+1}$, then 
for $\x,\y\in \mathbb{R}^d\times (0,1)$ %, under $\P_{\x,\y}$, 
{we have
\[
\mathd (\x,\y) \geq \frac{4\log\log\abs{x-y}}{{\log\big(\frac{\gamma}{\delta(1-\gamma)}\big)}}
\]
under $\P_{\x,\y}$ with high probability as $\abs{x-y} \to \infty$.}
%we have
%$$\liminf_{\abs{x-y}\to \infty} \, \frac{\mathd (\x,\y)}{\log\log\abs{x-y}}
%\geq \frac{4}{{\log\big(\frac{\gamma}{\delta(1-\gamma)}\big)}}$$ in  probability.\bigskip
\end{enumerate}
\end{theorem}

The second result provides a matching upper bound for the chemical distance in the ultrasmall regime under Assumption~1.2. Put together we get the following limit theorem for the chemical distance under Assumptions~1.1 and~1.2 in the ultrasmall regime.
\bigskip

\begin{theorem}\label{limitthm}
Let $\mathscr{G}$ be a general geometric random graph which satisfies Assumption~\ref{ass:main1} 
and Assumption~\ref{ass:main2} for some  $\gamma>\frac{\delta}{\delta+1}$. Then $\mathscr{G}$ is {\bf ultrasmall} and, for $\x,\y \in \mathbb{R}^d\times (0,1)$,  we have 
\begin{equation}\label{gold}
\mathd(\mathbf{x},\mathbf{y}) = (4 + o(1)) \frac{\log\log \abs{x-y}}{\log\big(\frac{\gamma}{\delta(1-\gamma)}\big)}
\end{equation}
under $\P_{\x,\y}(\phantom{i}\cdot\phantom{i} \mid \x \leftrightarrow \y )$ with high probability as $\abs{x-y} \to \infty$.
\end{theorem}\bigskip

\pagebreak[3]
{\bf Remarks:}%\smallskip 
\begin{itemize}
\item For the convergence in %Theorem~\ref{mainthm}(b) and 
Theorem~\ref{limitthm} we fix marks $s,t\in(0,1)$ and add points $\x=(x,s)$ and $\y=(y,t)$ to the Poisson process. Then we show that
$$\P_{\x,\y}\Big( \Big|\frac{\mathd(\mathbf{x},\mathbf{y})}{\log\log \abs{x-y}} -  \frac{4}{\log\big(\frac{\gamma}{\delta(1-\gamma)}\big)}\Big|  > \epsilon \, 
\Big| \, \x \leftrightarrow \y\Big)$$
converges to zero if $\abs{x-y} \to \infty$.
\item Stronger results, like explicit lower bounds on $\mathd (\x,\y)$ under Assumption~\ref{ass:main1}  and upper bounds under Assumption~\ref{ass:main2} only will be formulated in Propositions~\ref{prop:smlthm} and Proposition~\ref{prop:uppthm}
below.
\item The results continue to hold \emph{mutatis mutandis} when the underlying Poisson process is replaced by the points of the lattice $\Z^d$ endowed with independent uniformly distributed marks.
\end{itemize}
\bigskip

\subsection{Examples}

\subsubsection{The soft Boolean model}\label{ex:softBool}

As explained in the introduction in the (soft) Boolean model
on $\R^d$  the points $x$  carry independent identically distributed random radii $R_x$
and unordered pairs of points $\{x,y\}$ carry independent identically distributed nonnegative random variables
$X(x, y)$. Given these variables two points $x$ and $y$ are connected iff
$$\frac{|x-y|}{R_x+R_y} < X(x,y).$$
For a lower bound we assume that there are constants $C_1, C_2>0$ such that 
$$\mathbb P(R_x > r)\leq C_1 r^{-d/\gamma}, \qquad 
\mathbb P(X(x,y) > r)\leq C_2 r^{-d \delta}.$$
We can put this model into our framework by constructing the radius $R_x$ of a point 
$\x=(x,t)$ as $R_x=F^{-1}(1-t)$ where $F$ is the distribution function of the radius distribution and $F^{-1}(t)=\inf\{ u \colon F(u)\geq t\}$ its generalised inverse.
Given $\mathcal X$, the probability of a connection of $\x$ and $\y$ is
$$\mathbb P_{\mathcal X}\big(   X(x,y) > \tfrac{|x-y|}{R_x+R_y}\big)
\leq C_2   \tfrac{(F^{-1}(1-t)+F^{-1}(1-s))^{d\delta}}{|x-y|^{d\delta}}.$$
As $F^{-1}(1-t)=\inf\{ u \colon 1-F(u)\le t\}\leq C_1^{\gamma/d} t^{-\gamma/d}$ we infer
that the probability of a connection of $\x$ and $\y$ is bounded by
\smash{$\kappa (t\wedge s)^{-\delta\gamma} \abs{x-y}^{-\delta d}$} and hence, using conditional independence of edges, 
Assumption~1.1 holds. The assumption then implies no ultrasmallness if  \smash{$\gamma<\frac{\delta}{\delta+1}$},
which holds in particular in the hard model for arbitrary $0<\gamma<1$, as $X(x,y)$ is constant and hence $\delta$ can be chosen arbitrarily large. 
%\medskip
Similarly, if \smash{$\gamma>\frac{\delta}{\delta+1}$} and for every small $\epsilon>0$
there are constants $c,C>0$ such that, for all $r\geq 1$,
$$c r^{-d/\gamma-\epsilon} \leq \mathbb P(R_x > r)\leq C r^{-d/\gamma+\epsilon}, \qquad 
c r^{-d\delta-\epsilon} \leq \mathbb P(X(x,y) > r)\leq C r^{-d \delta+\epsilon},$$
then Assumptions 1.1 and~1.2 hold for values arbitrarily close to $\gamma$ and $\delta$
and hence the full limit theorem in probability \eqref{gold} holds. 

\subsubsection{Hirsch's scale-free Gilbert graph}\label{ex:hirsch}

Hirsch~\cite{Hir2017} discusses a model which in its soft version connects every unordered pair of vertices $\{x,y\}$  iff
$$\frac{|x-y|}{R_x\vee R_y} \leq X(x,y),$$
where $R_x, R_y$ and $X(x,y)$ are as in Example \ref{ex:softBool}. He gives a lower bound for the chemical distance of the hard model, which is of the from $ |x-y| /\log |x-y|$. Our result
also shows that the hard model is not ultrasmall albeit with a much smaller lower bound
of an order slightly below $\log|x-y|$. However, this bound extends uniformly to the soft model if \smash{$\delta>\frac{\gamma}{1-\gamma}$}.
This includes long-range percolation, which corresponds to the case $\gamma=0$, in which we know from~\cite{Bis2004} that if $\delta<2$ the chemical distance is indeed of the order of a power of a logarithm. Our results become best possible looking at the soft model with $X$ heavy-tailed with \smash{$\delta<\frac{\gamma}{1-\gamma}$}. In that case  we show that distances can be drastically smaller and satisfy the limit theorem 
in Theorem~1.2.

\subsubsection{The age-dependent random connection model}\label{ex:agedep}

This dynamical model was introduced in~\cite{GraGLM2019} as a simplification of the spatial preferential attachment model of Jacob and M\"orters~\cite{JacM2015,JacM2017}.
A vertex $\x=(x,t)$ is born at time $t$ and at birth connects to all vertices $\y=(y,s)$
born previously with probability
$$\varphi\bigg( \frac{t \, |x-y|^d}{\beta (t/s)^{\gamma}} \bigg), $$
where $\beta>0$ is a density parameter and $\varphi\colon(0,\infty)\to[0,1]$ is a non-increasing profile function standardized to satisfy
$\int \varphi(|x|^d)\, dx=1$. It is easy to see that for $t\gg s$ the expected
degree at time $t$ of a vertex born at time $s$ is of asymptotic order $(t/s)^{\gamma}$, 
so that the model combines preferences of attachment to vertices of high degree and to nearby vertices in a balanced way. If $\varphi(r) \leq C r^{-\delta}$ we see that Assumption~1.1 holds so that ultrasmallness fails if
\smash{$\gamma<\frac{\delta}{\delta+1}$}. But if \smash{$\gamma>\frac{\delta}{\delta+1}$} and also, for every $\epsilon>0$, there is $c>0$
with $\varphi(r) \geq c r^{-\delta+\epsilon}$ for all $r\geq1$, then ultrasmallness holds and we get the asymptotic chemical distance as stated in~(\ref{limitnew}).

\subsubsection{Scale-free percolation}

As explained in Section~1.1 for the model of Deijfen et al.~\cite{DeivHH2013} and other models constructed by taking products of vertex weights and distances we do not expect our results to be relevant or even sharp. {In fact, the dependence on the weights in these models is so strong that the geometry does not play a significant role and the techniques developed in this paper are not needed to understand the behaviour of the chemical distance. For these models} Assumption~1.1 only holds for 
$\gamma<\frac{1}{2}$ and in this case we recover from Theorem~\ref{mainthm} the well-known result that the graph is not ultrasmall when the power-law exponent is $\tau>3$. {For recent results for the chemical distance when $\gamma<\frac{1}{2}$, see~\cite{HaoH2021}.}

\subsubsection{The reinforced age-dependent random connection model}

{We consider a reinforced version of the age-dependent random connection model described above, where the connection probability between vertices is reinforced by additional weights of the nodes. Interestingly, although edges do not occur independently of each other due to the additional weights, our results still apply in full generality. Let the vertex set be a Poisson point process $\mathcal{X}$ on $\mathbb{R}^d\times (0,1)$ as before. We assign in addition to each point $\x\in \mathcal{X}$ an independent identically distributed \emph{reinforcement weight} $W = W_\x$, for which we assume the second moment exists that it is almost surely bounded away from zero, i.e. there exists $\alpha>0$ such that $\P(W\geq \alpha) = 1$. Given $\mathcal{X}$ and the reinforcement weights, edges are then formed independently between $\x = (x,t)$ and $\y = (y,s)$ with probability 
\[
\varphi\left(\tfrac{(W_\x W_\y)^{-1/\delta}}{\beta}(t\wedge s)^\gamma (t\vee s)^{1-\gamma}\abs{x-y}^d\right),
\]
where $\varphi$ is as in Example \ref{ex:agedep}. Let $I \subset \mathcal{X}^2$ be a set of pairs of vertices where each vertex appears at most twice. If there is $C>0$ such that $\varphi(r) \leq C r^{-\delta}$ for all $r>0$, 
\begin{align*}
\E_\mathcal{X} \bigg[\prod_{(\x_i,\y_i)\in I} \!\! & 1\{\x_i \sim \y_i\}\bigg] \\&\leq 
\E_\mathcal{X} \bigg[\prod_{(\x_i,\y_i)\in I} C W_{\x_i} W_{\y_i} \beta^\delta (t_i\wedge s_i)^{-\gamma\delta}(t_i\vee s_i)^{(\gamma-1)\delta}\abs{x_i-y_i}^{-\delta d} \bigg]
\end{align*}
\begin{align*}
&\leq \!\!\prod_{(\x_i,\y_i)\in I} \!\! C \beta^\delta (t_i\wedge s_i)^{-\gamma\delta}(t_i\vee s_i)^{(\gamma-1)\delta}\abs{x_i-y_i}^{-\delta d} \E[W_{\x_i}^2]\E[W_{\y_i}^2],
\end{align*}
where the second inequality holds since each reinforcement weight appears at most twice in the product and they are independent of $\mathcal{X}$. As the second moment of the weights exists, Assumption 1.1 holds for an appropriately chosen $\kappa$. Hence, ultrasmallness fails if $\gamma<\frac{\delta}{\delta+1}$.
On the other hand, we can easily couple the reinforced age-dependent random connection model to an age-dependent random connection model with a modified density parameter, such that the later is a subgraph of the former.
Indeed, for each pair of vertices we draw an independent uniform random variable $U(\x,\y)$. Given the Poisson process~$\mathcal{X}$, the reinforcement weights and the family \smash{$(U(\x,\y))_{\x,\y\in \mathcal{X}}$}, we can construct the age-dependent random connection model and the reinforced model in the following way. First, add an edge between any pair of vertices when
\[
U(\x,\y) \leq \varphi\left(\frac{\alpha^{-2/\delta}}{\beta}(t\wedge s)^\gamma (t\vee s)^{1-\gamma}\abs{x-y}^d\right).
\]
This leads to the age-dependent random connection model with new density parameter $\tilde{\beta} = \beta\alpha^{-2/\delta}$. Since $W\geq \alpha$ almost surely, each such edge is also added in the reinforced model. To get the full reinforced model, we add additional edges to hitherto unconnected pairs of vertices if
\[
U(\x,\y) \leq \varphi\left(\frac{(W_\x W_\y)^{-1/\delta}}{\beta}(t\wedge s)^\gamma (t\vee s)^{1-\gamma}\abs{x-y}^d\right).
\]
As the age-dependent random connection model is ultrasmall when 
\smash{$\gamma>\frac{\delta}{\delta+1}$} and if for every $\epsilon>0$, there exists $c>0$
with $\varphi(r) \geq c r^{-\delta+\epsilon}$ for all $r\geq1$, the reinforced model is ultrasmall as well and we  get the asymptotic chemical distance as stated in~(\ref{limitnew}) under both tail assumptions stated for $\varphi$ in this section. Note that Examples \ref{ex:softBool} and \ref{ex:hirsch} can similarly be reinforced, and similar conclusions can consequently be drawn.
}

\subsubsection{Ellipses percolation}
In \cite{TeiU2017} Teixeira and Ungaretti introduce a model on $\mathbb{R}^2$ as a collection of random ellipses centred on points of a Poisson process $\mathcal{X}$ on $\mathbb{R}^2\times(0,1)$ with uniform marks $t$, from which the size of the major {half-}axis is derived as $t^{-\gamma/2}$ while its direction  is sampled uniformly. The size of the minor {half-}axis is one. The random graph is then constructed by taking the Poisson process as the vertex set and forming edges given the collection of random ellipses between pairs of points of the point process if their ellipses intersect. Hil\'ario and Ungaretti \cite{HilU2021} show that, for $\gamma \in (1,2)$, the model is ultrasmall. \bigskip 
 
We introduce a soft version of this model, where for each pair of vertices $\x,\y$ we consider copies of their ellipses where the size of the major axes are multiplied with independent, identically distributed positive heavy-tailed random variables $X=X(\x,\y)$ with $\P(X>r)\sim r^{-2\delta}$ for some $\delta>1$. An edge between $\x$ and $\y$ is then formed if the new ellipses intersect. Note that given $\mathcal{X}$ edges are not drawn independently of each other, as the neighbourhood of each vertex depends on the orientation of the ellipses. Our results show that, for $\gamma \in [0,1)$, the original model is never ultrasmall and the soft model is not ultrasmall if $\gamma <\frac{\delta}{\delta+1}$. 
%\bigskip 
%
%By comparing the ellipses with rectangles of width one and length of the major axis, 
{We see that if an edge is formed between $\x = (x,t)$ and $\y = (y,s)$, this implies that balls around $x$ and $y$ with radii \smash{$X(\x,\y)t^{-\gamma /2}$} and \smash{$X(\x,\y)s^{-\gamma /2}$} intersect.} %we have
%\[
%\frac{\abs{x-y}}{\cos(\theta)t^{-\gamma/2}+\sin(\phi)s^{-\gamma/2}+2}\leq X,
%\]
%where $\theta$ and $\phi$ are the uniformly distributed angles between the major axes and the line between $\x$ and $\y$. 
Thus, there exists $\kappa > 0$ such that
\[
\P_{\mathcal{X}}\set{\x\sim \y} \leq \P\big(X\geq \tfrac{\abs{x-y}}{t^{-\gamma /2} + s^{-\gamma /2}}\big) \leq \kappa (t\wedge s)^{-\gamma \delta} (t\vee s)^{(\gamma-1)\delta} \abs{x-y}^{-2\delta}.
\]
Since the random variables $X(\x,\y)$ are independent, Assumption 1.1 holds for $\gamma \in [0,1)$ and $\delta>1$ and the claimed result follows.

\section{Proof of the lower bounds for the chemical distance}

\paragraph{Truncated first moment method}

To prove the lower bounds of Theorem \ref{mainthm} we find an upper bound for $\P_{\x,\y}\set{\mathd (\x, \y) \leq 2\Delta}$ and choose $\Delta$ as large as possible while keeping 
the probability sufficiently small. Note that the definition of the graph distance~$\mathd$ can be reduced to the existence of self-avoiding paths, since if there exists a path of length $n$ between two given vertices there also exists a self-avoiding path with shorter or equal length between those two. Hence, the paths considered throughout this section are assumed to be self-avoiding. The event $\set{\mathd (\x, \y) \leq 2\Delta}$ is equivalent to the existence of at least one path between $\x$ and $\y$ of length smaller than $2\Delta$. Hence,
\begin{align*}
\P_{\x,\y}\set{\mathd (\x, \y) \leq 2\Delta} &= \P_{\x,\y}\bigg(\bigcup_{n=1}^{2\Delta} \bigcup_{\x_1,\ldots,\x_{n-1}\in \mathscr{G}}^{\neq} \set{\x_0 \sim \x_1 \sim \ldots \sim \x_{n-1} \sim \x_n}\bigg)\\
&\leq \sum_{n=1}^{2\Delta} \E \bigg[\sum_{\x_1,\ldots,\x_{n-1} \in \mathscr{G}}^{\neq} \mathbb{P}_{\mathcal{X}\cup\set{\mathbf{x}, \mathbf{y}}}\set{\x_0\sim \x_1\sim \ldots\sim \x_{n-1} \sim \x_n}\bigg],
\end{align*}
where $\x = \x_0$, $\y = \x_n$, $\bigcup^{\neq}$ (resp. $\sum^{\neq}$) denotes the union (resp. sum) over all possible sets of pairwise distinct vertices $\x_0,\ldots,\x_n$ of the Poisson process and $\E$ is the expectation with respect to the law of a Poisson process with unit intensity on $\mathbb{R}^d\times (0,1)$. {To keep notation throughout the paper short we will abbreviate the previous notation and write $\sum_{\x_1,\ldots,\x_m}$ for the sum over all sets of $m$ distinct vertices of the Poisson process.} We get, by using Mecke's equation~\cite{LasP2017} 
and {Assumption~1.1} that
\begin{align*}
& \P_{\x,\y}\set{\mathd (\x, \y) \leq 2\Delta}\\
& \leq \sum_{n=1}^{2\Delta} \int\limits_{(\mathbb{R}^d\times (0,1])^{n-1}} \bigotimes_{i=1}^{n-1} \mathd \x_i  \, \E \left[\mathbb{P}_{\mathcal{X}\cup\set{\mathbf{x},\x_1, \ldots, \x_{n-1}, \mathbf{y}}}\set{\x_0\sim \x_1\sim \ldots\sim \x_{n-1} \sim \x_n}\right]\\
&\leq \sum_{n=1}^{2\Delta} \int\limits_{(\mathbb{R}^d\times (0,1])^{n-1}} \!\! \bigotimes_{i=1}^{n-1} \mathd (x_i,t_i) \, \prod_{j=0}^{n-1} 
1 \wedge\!   \kappa (t_j\wedge t_{j+1})^{-\gamma \delta} (t_j\vee t_{j+1})^{\delta(\gamma-1)} \abs{x_j-x_{j+1}}^{-\delta d } .
\end{align*}
This {bound is only good enough if $\gamma<\frac12$. If $\gamma \geq\frac12$ the expectation on the right is dominated by} paths which are typically not present in the graph. These are paths which connect $\x$ or $\y$ quickly to vertices with small mark $t$. Our strategy is therefore to truncate the admissible mark of the vertices of a possible path between $\x$ and $\y$. We define a decreasing sequence $(\ell_k)_{k\in \mathbb{N}_0}$ of thresholds and call a tuple of vertices $(\x_0,\ldots,\x_n)$ \emph{good} if their marks 
satisfy $t_k\wedge t_{n-k} \geq \ell_k$ for all $k\in \set{0,\ldots, n}$. A path consisting of a good tuple of vertices is called a good path. We denote by $A^{_{(\x)}}_k$ the event that there exists a path starting in $\x$ which fails this condition after exactly $k$ steps, i.e. a path $((x,t),(x_1,t_1),\ldots (x_k,t_k))$ with $t\geq \ell_0, t_1\geq \ell_1,\ldots, t_{k-1} \geq \ell_{k-1}$, but $t_k<\ell_k$. Furthermore we denote by $B^{_{(\x,\y)}}_n$ the event that there exists a good path of length $n$ between $\x$ and $\y$. Then, for given vertices $\x$ and $\y$
\begin{equation}\label{eq:truncmombound}
\P_{\x,\y}\set{\mathd (\x, \y) \leq 2\Delta} \leq \sum_{n=1}^\Delta \P_{\x}(A^{(\x)}_n) + \sum_{n=1}^\Delta \P_{\y}(A^{(\y)}_n) + \sum_{n=1}^{2\Delta} \P_{\x,\y}(B^{(\x,\y)}_n). \tag{TMB}
\end{equation}
This decomposition is the same as for the mean-field models in~\cite{DerMM2012}. The main feature of our proof is to show that the geometric restrictions and resulting correlations in our spatial random graphs make it much more difficult for a path to connect to a vertex with small mark. Hence a larger sequence $(\ell_k)$ of thresholds can be chosen that still makes the two first sums on the right of~\eqref{eq:truncmombound} small, allowing the third sum to be small for a larger choice of $\Delta$. This requires a much deeper analysis of the graph and its spatial embedding.

\subsection{Outline of the proof}\label{subsec:outline}
The characteristic feature of the shortest path connecting two typical vertices is that, starting from both ends, the path contains a subsequence of  increasingly powerful vertices. The two parts started at the ends meet roughly in the middle in a vertex of exceptionally high power depending on the distance between the starting vertices. In our framework powerful vertices are characterised by small marks. For geometric random graphs fulfilling Assumption~\ref{ass:main1} we show that  
arbitrary strategies connecting increasingly powerful vertices are dominated by an
\emph{optimal strategy} by which paths make connections between vertices of increasingly high power in a way depending on the parameters~$\gamma$ and $\delta$ in our assumption: \medskip

\begin{itemize}
\item If \smash{$\gamma > \frac{\delta}{\delta+1}$} 
%the optimal strategy is to 
we connect two powerful vertices $\x$ and $\y$ 
%typically not connected directly, but 
via a \emph{connector}, a single vertex with a larger mark 
which is connected to both $\x$ and $\y$;\smallskip
\item if 
\smash{$\gamma<\frac{\delta}{\delta + 1}$} we connect them by a single edge.\medskip
\end{itemize}
In both cases, we now sketch how our argument works on paths containing only the optimal type of connection between powerful vertices. The principal challenge of the proof will however be to show how these proposed optimal strategies dominate the entirety of other possible strategies. This is particularly hard in the former case, because a vast number of potential strategies leads to a massive entropic effect that needs to be controlled.
Note also that at this point we need not show that the proposed optimal strategies actually work. This (easier) part of the proof requires Assumption~\ref{ass:main2} and is carried out in Section~3.
\smallskip

\begin{figure}[h]
\begin{center}
\begin{tikzpicture}[scale=0.4, every node/.style={scale=0.7}]
\draw[dashed] (-1,10.5) node[left] {$1$} -- (25, 10.5);
\draw[dashed] (-1,-1) node[left] {$0$} -- (25, -1);
\node (A) at (0,6)[circle, fill = black ,label={left:$\mathbf{x}$}] {};
\node (B) at (2,9.9)[circle, fill = gray, label={}] {};
\node (C) at (4,5) [circle, fill = black, label={left:$\mathbf{x}_2$}] {};
\node (D) at (6,10)[circle,  fill= gray, label = {}] {};
\node (E) at (8, 3)[circle, fill = black, label={}] {};
\node (F) at (10,9.8) [circle, fill= gray, label={}] {};
\node (G) at (12,0)[circle, fill = black, label={left:$\mathbf{x}_{n}$}] {};
\node (H) at (14,9.9)[circle,  fill= gray, label={}] {};
\node (I) at (16,2.8) [circle, fill = black, label={}] {};
\node (J) at (18,10.1)[circle,  fill= gray, label = {}] {};
\node (K) at (20,4.8)[circle, fill = black, label= {left:$\mathbf{x}_{2n-2}$}] {};
\node (L) at (22,9.7) [circle,  fill= gray, label= {}] {};
\node (M) at (24,5.5)[circle, fill = black, label={left:$\mathbf{y}$}]{};
\draw (A) to (B);
\draw (B) to (C);
\draw[dotted] (C) to (D);
\draw[dotted] (D) to (E);
\draw (E) to (F);
\draw (F) to (G);
\draw (G) to (H);
\draw (H) to (I);
\draw[dotted] (I) to (J);
\draw[dotted] (J) to (K);
\draw (K) to (L);
\draw (L) to (M);
\draw[dotted] (0-0.75,5) -- (0+0.75,5) node[right] {$\ell_0$};
\draw[dotted] (4-0.75,4.4) -- (4+0.75,4.4) node[right] {$\ell_2$};
\draw[dotted] (8-0.75,2.3) -- (8+0.75,2.3) node[right] {$\ell_{n-2}$} ;
\draw[dotted] (12-0.75,-0.5) -- (12+0.75,-0.5) node[right] {$\ell_{n}$};
\draw[dotted] (16-0.75,2.3) -- (16+0.75,2.3) node[right] {$\ell_{n-2}$};
\draw[dotted] (20-0.75,4.4) -- (20+0.75,4.4) node[right] {$\ell_2$};
\draw[dotted] (24-0.75,5) -- (24+0.75,5) node[right] {$\ell_0$};
\end{tikzpicture}
\caption{An example of a path with optimal connection type for $\gamma>\frac{\delta}{\delta+1}$. The horizontal axis corresponds to the sequential numbering of vertices on the path, the vertical axis represents the mark space. Powerful vertices (indicated by black dots) alternate with connectors (indicated by grey dots). }
\end{center}
\end{figure}
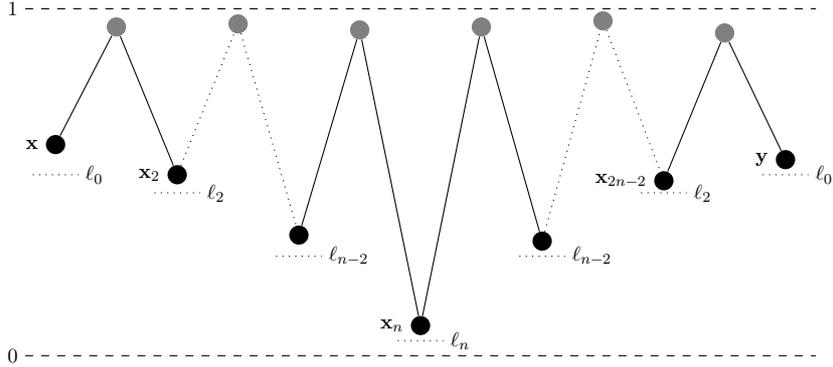

\pagebreak[3]
In the case \smash{$\gamma>\frac{\delta}{\delta+1}$} the optimal connection strategy 
is to follow a path of length~$2n$ between $\x$ and $\y$, where we assume that $n$ is even 
 and that the vertices $\x_1=(x_1,t_1),\ldots,$ $\x_{2n-1}=(x_{2n-1},t_{2n-1})$ of the path satisfy that \smash{$t_{2(k+1)}<t_{2k}<t_{2k+1}$} and \smash{$t_{2n-2(k+1)}<$}
 \smash{$t_{2n-2k}<t_{2n-2k+1}$} for all $k=0,\ldots,n/2$, i.e. the vertices with even index can be seen as powerful vertices, while the ones with odd index represent the connectors between them, see Figure~1. Note that at this point we make no assumptions on the locations of these vertices. \medskip
 
For arbitrary $\varepsilon>0$, we now determine a truncation sequence $(\ell_k)_{k\in \mathbb{N}_0}$, such that paths starting in $\x$, resp. $\y$, which are not good, only exist with a probability smaller than~$\varepsilon$. 
To do so, we establish an upper bound for the probability of the event $A_n^{_{(\x)}}$ that there exists a path starting in $\x$ whose $n$-th vertex is the first vertex which has a mark smaller than the corresponding~$\ell_n$. We denote by $N(\x,\y,n)$ the number of paths of length $n$ from $\x=(x,t)$ to a vertex $\y=(y,s)$ whose vertices $(x_1,t_1),\ldots (x_{n-1},t_{n-1})$ fulfill $t_{2(k+1)}<t_{2k}<t_{2k+1}$ for all  $k=0,\ldots,\gaus{n/2}-1$ and which is one half of a good path, i.e. $t\geq \ell_0, t_1\geq \ell_1,\ldots, t_{n-1} \geq \ell_{n-1}$. The mark of $\y$ is not restricted in this definition and is therefore allowed to be smaller than $\ell_n$.  Hence, in this case the event $A_n^{_{(\x)}}$ can only occur for $n$ even, since by definition a connector is less powerful than the preceding and following vertex and therefore has a mark larger than the corresponding~$\ell_n$. For $n$ even we have by Mecke's equation that
\[
\P_{\x}(A_n^{(\x)}) \leq  \int\limits\limits_{\mathbb{R}^d \times (0,\ell_n]} \mathd \y \, \E_{\x,\y} N(\x,\y,n).
\]
Since the existence of a path counted in $N(\x,\y,n)$ is equivalent to the existence of vertices $\z_1,\ldots, \z_{n/2-1}$ such that the marks are bounded from below by $\ell_2,\ell_4,\ldots, \ell_{n-2}$, with $\z_0=\x, \z_{n/2}=\y$
the marks {$u_0,\ldots, u_{n/2}$} of $\z_0,\ldots, \z_{n/2}$ are decreasing, and $\z_i, \z_{i+1}$ are connected via a single connector, Mecke's equation yields 
\begin{equation}
\E_{\x,\y}N(\x,\y,n) \leq \int\limits_{\mathbb{R}^d\times (\ell_{2},{u_0}]}\!\! \mathd \z_1 \cdots \!\!\!\!\!\!\int\limits_{\mathbb{R}^d\times(\ell_{n-2},{u_{n/2-2}}]}\mathd \z_{n/2-1} \E_{\z_0,\ldots,\z_{n/2}}\bigg[\prod_{i=1}^{n/2}K(\z_i,\z_{i-1},2)\bigg],\label{recbound_full}
\end{equation}
where $K(\z_i,\z_{i+1},2)$ is the number of connectors between $\z_i$ and $\z_{i+1}$. Using  Mecke's equation and Assumption~\ref{ass:main1} we have
\[
\E_{\z_0,\ldots,\z_{n/2}}\bigg[\prod_{i=1}^{n/2}K(\z_i,\z_{i-1},2)\bigg]\leq \prod_{i=1}^{n/2} e_K(\z_i,\z_{i-1},2),
\]
where 
\[
e_K(\z_{i},\z_{i-1},2) = \!\!\!\int\limits_{\mathbb{R}^d\times (u_{i-1}\vee u_i,1)}\!\!\!\mathd \z \, \rho(\kappa^{-1/\delta}u_{i-1}^\gamma u^{1-\gamma}\abs{z_{i-1}-z}^d)\rho(\kappa^{-1/\delta}u_i^\gamma u^{1-\gamma}\abs{z_i-z}^d),
\]
for $\rho(x):=1\wedge x^{-\delta}$ and $\z_i=(z_i,u_i)$, $\z_{i-1}=(z_{i-1},u_{i-1})$. 
We see in Lemma~\ref{twoconnlem} that there exists  $C>0$ such that, for two given vertices $\x = (x,t)$ and $\y = (y,s)$ far enough from each other,
\begin{equation}
%\E_{\x,\y}K(\x,\y,2) \leq 
e_K(\x,\y,2) \leq C\rho \big(\kappa^{-1/\delta} (t\wedge s)^{\gamma}(t\vee s)^{\gamma/\delta}\abs{x-y}^{d}\big).\label{eq:twoconn_ass}
\end{equation}
This inequality holds for the optimal connection type between two powerful vertices of the path and we will see that this type of bound holds also for the case of multiple connectors between two powerful vertices (cf. Lemma \ref{lem:kconn}). 
It also clearly displays the influence of the spatial embedding of the random geometric graph via the parameter $\delta$. Assuming \eqref{eq:twoconn_ass} for the moment, we obtain
\begin{align}
\begin{split}
& \P_{\x}(A_n^{(\x)})\leq \\
&\int\limits_{\mathbb{R}^d\times (\ell_{2},{u_0}]}\mathd \z_1 \cdots \!\!\!\!\!\!\!\!\!\int\limits_{\mathbb{R}^d\times(\ell_{n-2},{u_0}]}\mathd \z_{n/2-1} \int\limits_{\mathbb{R}^d \times (0,\ell_n]} \mathd \z_{n/2}\, \prod_{i=1}^{n/2}C\rho \big(\kappa^{-1/\delta} u_i^{\gamma}u_{i-1}^{\gamma/\delta}\abs{z_i-z_{i-1}}^{d}\big), \label{eq:badpaths_optimal}
\end{split}
\end{align}
where \smash{$\z_i = (z_i,u_i)$ for $i=0,\ldots n/2$} {and where we without loss of generality integrate up to $u_0$ in all but the last integral}. When dealing with a general (rather than the optimal) connection strategy, we will use a classification of the strategies in terms of binary trees. Left-to-right exploration of the tree will reveal the structure of the decomposition that replaces the straightforward decomposition in \eqref{recbound_full} and additional information on the location of the vertices will be encoded in terms of colouring of the leaves.
Figure~\ref{fig:dec_order-tree_optimal} displays the classifying binary tree for the optimal connection type.

\begin{figure}[h]
\begin{center}
\begin{tikzpicture}[scale=0.30, every node/.style={scale=0.7}]
\node (C) at (4,5.5)[circle, draw, fill= black, label={left:$\x$}] {};
\node (D) at (6,10)[circle,  draw, fill=gray, label = {$\x_1$}] {};
\node (E) at (8,4)[circle,  draw, fill=black, label = {left:$\x_2$}] {};
\node (F) at (10, 9.7)[circle, draw,fill = gray, label={$\x_3$}] {};
\node (G) at (12,3) [circle, draw, fill=black, label={left:$\x_4$}] {};
\node (H) at (14,10.3)[circle, draw,fill=gray, label={$\x_5$}] {};
\node (I) at (16,1) [circle, draw,fill= black, label={left:$\x_6$}] {};
\node (J) at (18,9.8) [circle, draw,fill= gray, label={$\x_7$}] {};
\node (K) at (20,0)[circle, fill = black, label={left:$\y$}] {};
\draw (C) to (D) to (E) to (F) to (G) to (H) to (I) to (J) to (K);
\end{tikzpicture}
\qquad
\begin{tikzpicture}[scale=0.40, every node/.style={scale=0.7}]
\node (A) at (9, 1)[circle, draw,fill = black, label={left:$\x_6$}] {};
\node (B) at (6,3)[circle, draw, fill= black, label={left:$\x_4$}] {};
\node (C) at (3,5)[circle, draw, fill= black ,label={left:$\x_2$}] {};
\node (D) at (0,7)[circle, draw, fill=black, label={left:$\x$}] {};
\node (E) at (2,8)[circle,  draw, fill=gray, label = {$\x_1$}] {};
\node (F) at (5,6)[circle,  draw, fill=gray, label = {$\x_3$}] {};
\node (G) at (8,4) [circle, draw,fill= gray, label={$\x_5$}] {};
\node (H) at (11,2) [circle, draw, fill=gray, label={$\x_7$}] {};
\draw (A) to (B) to (C) to (D);
\draw[dashed] (A) to (H);
\draw[dashed] (B) to (G);
\draw[dashed] (C) to (F);
\draw[dashed] (D) to (E);
\end{tikzpicture}
\caption{Representation of a path with optimal connection type 
by a binary tree. {For a less trivial example resulting from a general connection strategy, see Figure \ref{fig:dec_order-tree2}.}
%General connection strategies will be classified by means of associated binary trees.
\label{fig:dec_order-tree_optimal}}
\end{center}
\end{figure}
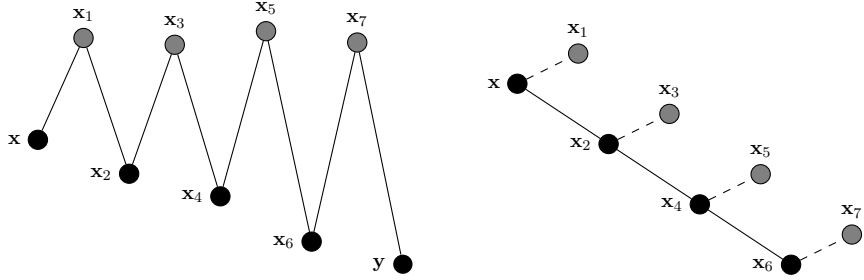

\noindent
For a sufficiently large constant $c>0$ the right-hand side of \eqref{eq:badpaths_optimal} can be bounded by $$
c^{n/2}\ell_{n}^{1-\gamma}\ell_0^{-\gamma/\delta}\prod_{i=1}^{n/2-1}\ell_{2i}^{1-\gamma-\gamma/\delta}$$ as shown in  Lemma~\ref{lem:mu_bound} considering all paths.
With $\ell_0$ smaller than the mark of $\x$ we choose the truncation sequence $(\ell_k)$ for $\varepsilon>0$, such that
\begin{equation}
c^{n/2}\ell_{n}^{1-\gamma}\ell_0^{-\gamma/\delta}\prod_{i=1}^{n/2-1}\ell_{2i}^{1-\gamma-\gamma/\delta} = \frac{\varepsilon}{\pi^2n^2}, \label{eq:truncdef}
\end{equation}
and we have
\[
\sum_{n=1}^\Delta \P_{\x}(A^{(\x)}_n) = \sum_{\substack{n=1\\ n\ even}}^\Delta \P_{\x}(A^{(\x)}_n) \leq \sum_{\substack{n=1\\ n\ even}}^\Delta c^{n/2}\ell_{n}^{1-\gamma}\ell_0^{-\gamma/\delta}\prod_{i=1}^{n/2-1}\ell_{2i}^{1-\gamma-\gamma/\delta} \leq \sum_{n=1}^\infty \frac{\varepsilon}{\pi^2n^2} = \frac{\varepsilon}{6}.
\]
Writing $\eta_n := \ell_n^{-1}$ we can deduce from \eqref{eq:truncdef}  a recursive description of $(\ell_n)_{n\in\mathbb{N}_0}$  such that
$$\eta_{n+2}^{1-\gamma}  = \frac{(n+2)^2}{n^2} c \eta_n^{\gamma/\delta}.$$
Consequently there exist $b>0$ and $B>0$ such that
$
\eta_n \leq b\exp(B(\gamma/(\delta(1-\gamma)))^{n/2}). 
$
We close the argument with heuristics that leads from this truncation sequence to a lower bound for the chemical distance. Let $\x$ and $\y$ be two given vertices. If there exists a path of length $n<\log \abs{x-y}$ between them, there must exist at least one edge in this path which is longer than \smash{$\frac{\abs{x-y}}{\log \abs{x-y}}$}. For $\abs{x-y}$ large, this edge typically must have an endvertex whose mark is, up to a multiplicative constant, smaller than \smash{$\abs{x-y}^{-d}$}. Hence, if we choose $$\Delta < (2 + o(1)) \frac{\log\log \abs{x-y}}{\log\big(\frac{\gamma}{\delta(1-\gamma)}\big)}$$ we ensure $\ell_\Delta$ is of larger order than $\abs{x-y}^{-d}$. Therefore there is no good path whose vertices are powerful enough to be an endvertex of an edge longer than \smash{$\frac{\abs{x-y}}{\log\abs{x-y}}$} and consequently no good path of length shorter than $2\Delta$ can exist between $\x$ and $\y$.
\medskip \pagebreak[3]

Turning to the case $\gamma < \frac{\delta}{\delta +1}$, we consider paths whose powerful vertices are connected directly to each other. 
For a path of length $n$ between two given vertices $\x$ and $\y$ we assume that $n$ is even and for the vertices $\x_1=(x_1,t_1),\ldots, \x_{n-1}=(x_{n-1},t_{n-1})$ of the path we assume that we have $t_0>t_1>\ldots>t_{n/2}$ and $t_n>t_{n-1}>\ldots>t_{n/2}$, where $t_0$ is the mark of $\x$ and $t_n$ the mark of $\y$. We again make no restrictions on the locations of those vertices. Restricting the paths described in $A_n^{_{(\x)}}$ and $B_n^{_{(\x,\y)}}$ to paths with this structure we  follow the same argumentation as above to establish sufficiently small bounds for the event $A_n^{_{(x)}}$ for a given vertex $\x = (x_0,t_0)$,
\[
\P_{\x}(A_n^{(\x)}) \leq \!\!\!\int\limits_{\mathbb{R}^d\times (\ell_{1},t_0]}\!\!\mathd \x_1 \cdots \!\!\!\!\int\limits_{\mathbb{R}^d\times(\ell_{n-1},t_0]}\!\!\!\!\!\mathd \x_{n-1} \int\limits_{\mathbb{R}^d \times (0,\ell_n]} \!\!\mathd \x_n\prod_{i=1}^{n}C\rho \big(\kappa^{-1/\delta} t_i^{\gamma}t_{i-1}^{1-\gamma}\abs{x_i-x_{i-1}}^{d}\big),
\]
where we again without loss of generality integrate over a larger range. For $c>0$ large enough, the right-hand side can be further bounded by $$c^{n}\ell_{n}^{1-\gamma}\ell_0^{\gamma-1}\prod_{i=1}^{n-1}\log\left(\frac{1}{\ell_{i}}\right),$$ see Lemma \ref{lem:mu_bound_non}. Choosing $\ell_0 < t_0$ and $(\ell_n)_{n\in \mathbb{N}_0}$ for $\epsilon>0$, such that the last displayed term equals 
$\frac{\varepsilon}{\pi^2n^2}$
ensures that $\sum_n \P_\x(A_n^{\x})<\frac{\varepsilon}{6}$ and by induction we see that this choice is possible while for any $p>1$ there exists $B>0$ such that
$
\eta_n \leq B^{n\log^p(n+1)}.
$
Following the same heuristics as before leads to the choice $$\Delta \leq \frac{c\log\abs{x-y}}{(\log\log\abs{x-y})^p}$$ for some constant $c>0$ such that paths between $\x$ and $\y$ with length shorter than $2\Delta$ do not exist with high probability.\pagebreak[3]

\subsection{The ultrasmall regime}\label{subsec:ultsml_phase}

We now start the full proof in the case $\gamma>\frac{\delta}{\delta+1}$ considering all possible connection strategies.  We prepare this by first modifying the graph by adding edges between vertices which are sufficiently close to each other. We call
a path \emph{step minimizing} if it connects any pair of vertices on the path by a direct edge, if it is available. Note that the  length of any path connecting two fixed vertices can be bounded from below by the length of a step mimimizing path connecting the two vertices.
Two spatial constraints emerge from this: On the one hand, vertices on a step minimizing path in the modified graph that are not neighbours on the path cannot be near to each other. On the other hand, vertices connected by one of the added edges have to be near to each other. 
To make full use of these constraints we need to distinguish  between original edges and edges added to the graph. This can be done efficiently by endowing every edge with a \emph{conductance}, which is one for original and~two for added edges.\medskip

More precisely, we consider a graph $\tilde{\mathscr{G}}$ where edges are endowed with
{conductances}  as follows: First, create a copy of $\mathscr{G}$ and assign to every edge conductance one. Then, between two vertices $\x = (x,t)$ and $\y = (y,s)$ of $\tilde{\mathscr{G}}$ an edge is added to $\tilde{\mathscr{G}}$ with conductance two whenever 
\[
\abs{x-y}^d\leq \kappa^{1/\delta}(t\wedge s)^{-\gamma} (t\vee s)^{-\gamma/\delta}.
\]
Since all conductances and edges of $\tilde{\mathscr{G}}$ are deterministic functionals of $\mathscr{G}$, there exists an almost sure correspondence between $\mathscr{G}$ and $\tilde{\mathscr{G}}$, under which an edge with conductance one in $\tilde{\mathscr{G}}$ implies the existence of the same edge in $\mathscr{G}$. With conductances assigned to every edge of $\tilde{\mathscr{G}}$, we define the conductance of a path $P=(\x_0,\ldots,\x_n)$ \smash{in~$\tilde{\mathscr{G}}$} as the sum over all conductances of the edges of $P$ and denote it by $w_P$.
\medskip

\noindent We call a self-avoiding path $P=(\x_0,\ldots,\x_n)$ in $\mathscr{G}$ or $\tilde{\mathscr{G}}$ \emph{step minimizing}
\begin{equation}
\text{if there exists no edge between }\x_i\text{ and }\x_j\text{ for all }i,j\text{ with }\abs{i-j}\geq 2. \label{eq:dist_min_definition}
\end{equation}
Note that a step minimizing path in $\mathscr{G}$ 
is not necessarily step minimizing 
\smash{in~$\tilde{\mathscr{G}}$}, since there could exist an edge of conductance two between two vertices of the path that would reduce the number of steps. But by removing the vertices connecting such a pair of vertices from the path we can shorten the path to a step minimizing path
in~$\tilde{\mathscr{G}}$ whose length and conductance is no more than the length of the original path. Hence the chemical distance  $\mathd(\mathbf{x},\mathbf{y})$ between vertices $\x$ and $\y$ in $\mathscr{G}$ is larger or equal than the conductance $\mathd_w(\mathbf{x},\mathbf{y}) := \min \set{w_P: P\ \text{is a path between }\x\ \text{and}\ \y}$ between them in $\tilde{\mathscr{G}}$.\pagebreak[3]\medskip%

To bound the probabilities occurring in \eqref{eq:truncmombound}, we express the events on $\mathscr{G}$ with the help of corresponding events on $\tilde{\mathscr{G}}$ by replacing the role of the length of a path by its conductance. The role of the conductance is crucial, as it allows us to distuingish newly added edges in a path, which is necessary to keep the bounds of the probabilities in \eqref{eq:truncmombound} sufficiently small. We call a path $P = (\x_0,\ldots,\x_n)$ in $\tilde{\mathscr{G}}$ \emph{good} if its marks satisfy $t_k\geq \ell_{w_P(k)}$ and $t_{n-k}\geq \ell_{w_P-w_P(n-k)}$ for all $k = 0,\ldots,n$, {where $w_P(k)$ is the conductance of $P$ between $\x_0$ and $\x_k$.} We denote by \smash{$\tilde{A}_k^{\x}$} the event that there exists a step minimizing path starting in $\x$ in $\tilde{\mathscr{G}}$ with conductance $k$ which fails to be good on its last vertex. \label{defatilde} Notice that if there exists a path described by the event $A_k^{\x}$, i.e. a {path for which the $k$-th vertex is the first one whose mark is smaller than the corresponding truncation value $\ell_k$}, then due to the correspondence between $\mathscr{G}$ and $\tilde{\mathscr{G}}$ there also exists a step minimizing path $P$ in $\tilde{\mathscr{G}}$ with $w_P\leq k$ which also fails the condition on its last vertex. Hence, the first two summands of the right-hand side of \eqref{eq:truncmombound} can be bounded from above by 
\smash{$\sum_{n=1}^\Delta \P_{\x}(\tilde{A}^{(\x)}_n)$ and $\sum_{n=1}^\Delta \P_{\y}(\tilde{A}^{(\y)}_n)$}.
\bigskip\\
To bound  $\P_{\x}(\tilde{A}^{(\x)}_n)$, we count the expected number of paths occurring in the event \smash{$\tilde{A}^{(\x)}_n$}. Note that if $\smash{\abs{x-y}^d\leq \kappa^{1/\delta}(t\wedge s)^{-\gamma} (t\vee s)^{-\gamma/\delta}}$ holds for vertices $\x$ and $\y$, there exist no step minimizing paths between $\x$ and $\y$ with conductance larger or equal three and there exists one step minimizing path with conductance two, since there exists an edge of conductance two between the two vertices. This property also holds for any of the subclasses of step minimizing paths introduced in the following.
\bigskip\pagebreak[3]

For given vertices $\x=(x,t)$ and $\y=(y,s)$ define the random variable $N(\x,\y,n)$ \label{Ndef}
as the number of distinct step minimizing paths $P$ between $\x$ and $\y$ with $w_P=n$, whose connecting vertices $(x_1,t_1),\ldots$ 
$(x_{m-1},t_{m-1})$ all have a larger mark than $\y$ and fulfill 
\smash{$t\geq \ell_0, t_1\geq \ell_{w_P(1)},\ldots, t_{m-1} \geq \ell_{w_P(m-1)}$}. 
As $\tilde{A}^{_{(\x)}}_n$ is the event that there exists a path with conductance $n$, where the final vertex is the first and only one which has a mark smaller than the corresponding $\ell_n$, the final vertex is also the most powerful vertex of the path. Hence, the number of paths described by the event $\tilde{A}^{(\x)}_n$ can be written as the sum of $N(\x,\y,n)$ over all sufficiently powerful vertices $\y$ of the graph and, by Mecke's formula, we have
\begin{equation}
\P_{\x}(\tilde{A}^{(\x)}_n) \leq \int\limits_{\mathbb{R}^d \times (0,\ell_n]} \mathd \y \, \E_{\x,\y} N(\x,\y,n).
\end{equation}
We now decompose $N(\x,\y,n)$. For $k=1,\ldots, n-1$, define $N(\x,\y,n,k)$ as the number of step minimizing paths $P$ between $\x$ and $\y$ with $w_P=n$~and\smallskip
\begin{itemize}
\item whose connecting vertices $(x_1,t_1),\ldots (x_{m-1},t_{m-1})$ have marks larger than the corresponding 
thresholds $\ell_{w_P(1)}, \ldots, \ell_{w_P(m-1)}$ and larger than the mark of $\y$, and\smallskip
\item there exists $r \in \set{1,\ldots,m-1}$ such that we have $w_P(r) = n-k$ and the connecting vertex $\x_{r} = (x_{r},t_{r})$ has the smallest mark among the connecting vertices and $\x$.\smallskip
\end{itemize}
The vertex $\x_{r}$ can be understood as the %preceding
 powerful vertex of the path which connects to~$\y$ via a path of less powerful vertices with conductance $k$. Consequently, we write $N(\x,\y,n,n)$ for the number of step minimizing paths of conductance $n$, which connect $\x$ and $\y$ via less powerful vertices. Then we have, for $n\in\mathbb{N}$,
\begin{equation}
 N(\x,\y,n) \leq \sum_{k=1}^{n} N(\x,\y,n,k). \label{eq:recdecomp}
\end{equation}
For $k=1,\ldots,n-1$, the existence of a path counted in $N(\x,\y,n,k)$ implies the existence of a vertex $\z$ such that a step minimizing path counted by $N(\x,\z,n-k)$ exists which connects to $\y$ via a path of less powerful vertices with conductance $k$. Hence
\begin{equation}
N(\x,\y,n,k) \leq \sum_{{\z=(z,u)}\atop{\ell_{n-k}\vee s<u\leq t}}N(\x,\z,n-k)\, K(\z,\y,k), \label{eq:recdecomp2}
\end{equation}
for $n\in \mathbb{N}$ and $k=1,\ldots, n-1$, where we denote by $K(\z,\y,k)$ the number of step minimizing paths $P$ between $\z$ and $\y$ with $w_P = k$ whose vertices have marks larger than the marks of $\z$ and $\y$. Note that unlike $N(\x,\y,n)$, this random variable is symmetric in its first two arguments and {by definition we have that $N(\x,\y,n,n) = K(\x,\y,n)$.} Observe that {$K(\z,\y,1)$} is the indicator whether {$\z$} and $\y$ are connected by an edge with conductance one. We turn our attention to {$K(\z,\y,k)$} in the case $k\geq 2$, i.e.\ two powerful vertices are connected via one or more connectors or an edge with conductance~two.%

\paragraph{Connecting powerful vertices.}
First consider the random variable $K(\x,\y,2)$. If $\smash{\abs{x-y}^d\leq \kappa^{1/\delta}(t\wedge s)^{-\gamma} (t\vee s)^{-\gamma/\delta}}$, the vertices $\x$ and $\y$ are connected by an edge with conductance two and we infer that $K(\x,\y,2) = 1$. In the other case, $K(\x,\y,2)$ is equal to the number of connectors between $\x$ and $\y$, i.e the number of vertices with mark larger than the marks of $\x$ and $\y$, which form an edge of conductance one to $\x$ and $\y$. The following lemma shows the stated inequality \eqref{eq:twoconn_ass} from Section \ref{subsec:outline} for this case. Recall that we write $\rho(x) := 1\wedge x^{-\delta}$ and define $I_\rho := \int_{\mathbb{R}^d}\mathd x \rho(\kappa^{-1/\delta}\abs{x}^d)$.

\begin{lemma}[Two-connection lemma]\label{twoconnlem}
Let $\x = (x,t), \y = (y,s)$ be two given vertices with $\smash{\abs{x-y}^d > \kappa^{1/\delta}(t\wedge s)^{-\gamma} (t\vee s)^{-\gamma/\delta}}$ . Then
\begin{align*}
\E_{\mathbf{x},\mathbf{y}}K(\x,\y,2) &\leq \int_{\mathbb{R}^d\times (t\vee s,1]} \mathd \z\rho \big(\kappa^{-1/\delta} t^{\gamma} u^{1-\gamma} \abs{x-z}^d \big)\rho \big(\kappa^{-1/\delta} s^{\gamma} u^{1-\gamma} \abs{y-z}^d \big)\\
&\leq C \kappa (t\wedge s)^{-\gamma\delta}(t\vee s)^{-\gamma}\abs{x-y}^{-d \delta}
\end{align*}
where $C = \tfrac{I_\rho 2^{d \delta + 1}}{(\gamma - (1-\gamma)\delta)}$.
\end{lemma}\pagebreak[3]

\begin{proof}
The first inequality follows directly by summing over all possible connectors and applying Assumption \ref{ass:main1} and Mecke's formula. 
Observe that for every vertex $\z=(z,u)$ either $\abs{x-z}\geq \frac{\abs{x-y}}{2}$ or $\abs{y-z}\geq \frac{\abs{x-y}}{2}$, as the open sets $B_{\frac{\abs{x-y}}{2}}(x)$ and $B_{\frac{\abs{x-y}}{2}}(y)$ are disjoint. Hence, we have 
\begin{align*}
&\int_{\mathbb{R}^d\times (t\vee s,1]} \mathd \z\rho \big(\kappa^{-1/\delta} t^{\gamma} u^{1-\gamma} \abs{x-z}^d \big)\rho \big(\kappa^{-1/\delta} s^{\gamma} u^{1-\gamma} \abs{y-z}^d \big)\\
&\leq \,\,\int\limits_{t\vee s}^1 \mathd u \rho \big(2^{-d}\kappa^{-1/\delta} t^{\gamma} u^{1-\gamma} \abs{x-y}^d \big) \int\limits_{\mathbb{R}^d} \mathd z \rho \big(\kappa^{-1/\delta} s^{\gamma} u^{1-\gamma} \abs{y-z}^d \big)\\
& \qquad + \int\limits_{t\vee s}^1 \mathd u \rho \big(2^{-d}\kappa^{-1/\delta} s^{\gamma} u^{1-\gamma} \abs{x-y}^d \big)\int\limits_{\mathbb{R}^d}\mathd z\rho \big(\kappa^{-1/\delta} t^{\gamma} u^{1-\gamma} \abs{x-z}^d \big)\\
&\leq  \,\, I_\rho 2^{d\delta}\kappa \bigg[\int\limits_{t\vee s}^1 \mathd u\ t^{-\gamma\delta} u^{(\gamma-1)\delta} \abs{x-y}^{-d \delta} s^{-\gamma} u^{\gamma-1} \\
& \phantom{lamalamalamadingdong} +s^{-\gamma \delta} u^{(\gamma-1)\delta} \abs{x-y}^{-d \delta} t^{-\gamma} u^{\gamma-1} \bigg]\\
&\leq  \,\,\tfrac{I_\rho 2^{d\delta}}{(\gamma - (1-\gamma)\delta)}\kappa \abs{x-y}^{-d \delta}\left[t^{-\gamma\delta}s^{-\gamma} + s^{-\gamma\delta}t^{-\gamma}\right]\\
&\leq  \,\,\tfrac{I_\rho 2^{d\delta+1}}{\gamma - (1-\gamma)\delta}\, \kappa (t\wedge s)^{-\gamma\delta}(t\vee s)^{-\gamma}\abs{x-y}^{-d \delta}.
\end{align*}\ \\[-8mm]
\end{proof}
\bigskip

We consider the event that  vertices $\x$ and $\y$ are
connected by multiple vertices with larger marks. Recall that $K(\x,\y,k)$ is the number of step minimizing paths $P$ between $\x$ and $\y$ with $w_P = k$ whose vertices have marks larger than the marks of $\x$ and $\y$. As before we call the vertices of such a path {connectors}. To control the number of such paths,
notice that for any possible choice of connectors between $\x$ and $\y$, there exists an almost surely unique connector with smallest mark, i.e the most powerful connector. For $i=1,\ldots,k$, we denote by $K(\x,\y,k,i)$ the number of step minimizing paths between $\x$ and $\y$ where the connectors have a larger mark than $\x$ and~$\y$ and there is a vertex $\x_r$ with $w_P(r) = i$ which is the most powerful connector of those vertices. Then, 
\[
K(\x,\y,k) \leq \sum_{i=1}^{k-1} K(\x,\y,k,i).
\]
Assume now that the connector $\x_r$ is the most powerful of all connectors and $w_P(r)=i$. In this case, the possible connectors $\x_1,\ldots, \x_{r-1}$ and $\x_{r+1},\ldots, \x_{m-1}$ need to have larger mark than $\x_r$. Hence, the paths between $\x_r$ and $\x$, resp. $\y$, considered on their own have the same structure as the initial path and 
this leads to
\begin{equation}
K(\x,\y,k)\leq \sum_{i=1}^{k-1} \sum_{\substack{\z = (z,u)\\ u> t\vee s}} K(\x,\z,i)K(\z,\y,k-i). \label{kconn_recrep_rv}
\end{equation}
\begin{figure}[h]
\begin{center}
\begin{tikzpicture}[scale=0.4, every node/.style={scale=0.7}]
\draw[dashed] (-1,13.5) node[left] {$1$} -- (21, 13.5);
\draw[dashed] (-1,5.5) node[left] {$\ell$} -- (21, 5.5);
\node (A) at (0,6)[circle, fill = black ,label={left:$\mathbf{x}$}] {};
\node (B) at (2,12)[circle, fill = gray, label={}] {};
\node (C) at (4,10) [circle, fill = gray, label={}] {};
\node (D) at (6,6.5)[circle,  fill= red, label = {}] {};
\node (F) at (8,8) [circle, fill= gray, label={}] {};
\node (G) at (10,12)[circle, fill = gray, label={}] {};
\node (H) at (12,10)[circle,  fill= gray, label={}] {};
\node (I) at (14,11.5) [circle, fill = gray, label={}] {};
\node (K) at (16,9)[circle, fill = gray, label= {}] {};
\node (L) at (18,12.5) [circle,  fill= gray, label= {}] {};
\node (M) at (20,3)[circle, fill = black, label={left:$\mathbf{y}$}]{};
\draw[dashed] (A) to (B) to (C) to (D);
\draw (D) to (F) to (G) to (H) to (I) to (K) to (L) to (M);
\end{tikzpicture}
\end{center}
\caption{Decomposing the path at the most powerful connector.}
\end{figure}
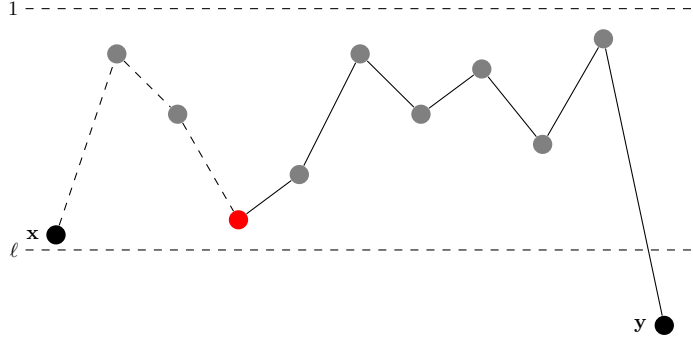\\
We use this decomposition together with Assumption \ref{ass:main1} to find an upper bound for $\E_{\x,\y} K(\x,\y,k)$. Recall that, if $\smash{\abs{x-y}^d\leq \kappa^{1/\delta}(t\wedge s)^{-\gamma} (t\vee s)^{-\gamma/\delta}}$, we have $K(\x,\y,k) = 0$ if $k\geq 3$ and $K(\x,\y,k) = 1$ if $k= 2$ by definition. We now introduce a mapping 
\[
e_K \colon (\mathbb{R}^d\times (0,1])^2 \times \mathbb{N} \to [0,\infty).
\]
by
$e_K(\x,\y,1) = \rho(\kappa^{-1/\delta}(t\wedge s)^{\gamma}(t\vee s)^{1-\gamma}\abs{x-y}^d),$ for $\x,\y \in \mathbb{R}^d\times (0,1],$
and, for $k\geq 2$ under the assumption that  $\smash{\abs{x-y}^d > \kappa^{1/\delta}(t\wedge s)^{-\gamma} (t\vee s)^{-\gamma/\delta}}$,
\begin{equation}
e_K(\x,\y,k) = \sum_{i=1}^{k-1}\int\limits_{\mathbb{R}^d\times (t\vee s,1]} \mathd \z \, e_K(\x,\z,i)e_K(\z,\y,k-i), \quad \text{for}\ \x,\y \in \mathbb{R}^d\times (0,1], \label{eq:e_def_rec}
\end{equation}
and otherwise $e_K(\x,\y,2) = 1$ and $e_K(\x,\y,k) = 0$ for $k\geq 3$. 
\pagebreak[3]

\begin{lemma}\label{lem:kconn_number_determ_bound}
Let $\x, \y \in \mathbb{R}^d\times (0,1]$ be two given vertices. Then, for all $k\in \mathbb{N}$, we have
\begin{equation}
\E_{\x,\y} K(\x,\y,k) \leq e_K(\x,\y,k).
\end{equation}
\end{lemma}
\pagebreak[3]

Note that by Assumption~\ref{ass:main1} and Lemma \ref{twoconnlem}, we have $\E_{\x,\y}K(\x,\y,1)\leq e_K(\x,\y,1)$ and $\E_{\x,\y}K(\x,\y,2)\leq e_K(\x,\y,2)$. We prove the result for general $k$ by induction using \eqref{kconn_recrep_rv}, but to do so we need to 
classify the possible connection strategies according to the way in which powerful vertices are placed. This classification is done by means of coloured binary trees. We write $\mathcal{T}_{k-1}$ for the set of all binary trees with $k-1$ vertices. Here a binary tree is a rooted tree in which every vertex can have either no child, a right child, a left child or both. We colour the vertices of a tree $T\in \mathcal{T}_{k-1}$ in such a way that the leaves of the tree can be either blue or red, and every other vertex is coloured blue. Thus, for each $T\in \mathcal{T}_{k-1}$ there exist $2^{\ell}$ different colourings, where $\ell$ is the number of leaves of $T$. Let $\mathcal{T}_{k-1}^c$ be the set of all coloured trees. \bigskip

Before proceeding we outline the role of the tree and its coloured vertices in regard to the information they capture. We will construct the tree so as to describe the precise order of the connectors' marks. In order to distuingish between connections of vertices that are sufficiently close to form an edge with conductance two and connections between vertices which are further apart, red vertices of the tree will represent the first case and blue the second.\pagebreak[3]\bigskip

\begin{subfigures}
\begin{figure}[!t]
\begin{center}
\begin{tikzpicture}[scale=0.35, every node/.style={scale=0.6}]
\node (A) at (0,5.6)[circle, draw,fill = black ,label={left:$\mathbf{x}$}] {};
\node (B) at (2,12)[circle, draw,fill = gray, label={$\x_1$}] {};
\node (C) at (4,10) [circle, draw,fill = gray, label={$\x_2$}] {};
\node (D) at (6,6.5)[circle,  draw,fill= gray, label = {left:$\x_3$}] {};
\node (F) at (8,8) [circle, draw,fill= gray, label={left:$\x_4$}] {};
\node (G) at (10,12)[circle, draw,fill = gray, label={$\x_5$}] {};
\node (H) at (12,10)[circle,  draw,fill= gray, label={$\x_6$}] {};
\node (I) at (14,11.5) [circle, draw,fill = gray, label={$\x_7$}] {};
\node (K) at (16,9)[circle, draw,fill = gray, label= {left:$\x_8$}] {};
\node (L) at (18,12.5) [circle,  draw,fill= gray, label= {$\x_9$}] {};
\node (M) at (20,3)[circle, draw,fill = black, label={left:$\mathbf{y}$}]{};
\draw (A) to (B) to (C) to (D);
\draw (D) to (F) to (G) to (H) to (I) to (K) to (L) to (M);
\draw[dashed] (0.6,6) rectangle (19.5,14.5);
\draw[dotted] (8.6,8.5) rectangle (19,14);
\end{tikzpicture}
\qquad
\begin{tikzpicture}[scale=0.38, every node/.style={scale=0.65}]
\node (B) at (1,6)[circle, draw,fill = cyan, label={$\x_1$}] {};
\node (C) at (3,4) [circle, draw,fill = cyan, label={$\x_2$}] {};
\node (D) at (6,2)[circle,  draw,fill= cyan, label = {$\x_3$}] {};
\node (F) at (9,4) [circle, draw,fill= cyan, label={$\x_4$}] {};
\node (G) at (9,10)[circle, draw,fill = cyan, label={$\x_5$}] {};
\node (H) at (10,8)[circle,  draw,fill= cyan, label={left:$\x_6$}] {};
\node (I) at (11,10) [circle, draw,fill = cyan, label={$\x_7$}] {};
\node (K) at (11,6)[circle, draw,fill = cyan, label= {left:$\x_8$}] {};
\node (L) at (12,8) [circle, draw, fill= cyan, label= {$\x_9$}] {};
\draw (B) to (C) to (D);
\draw (D) to (F) to (K) to (H) to (G);
\draw (H) to (I);
\draw (K) to (L);
\draw[dashed] (0,1.5) rectangle (13,12.5);
\draw[dotted] (7.8,5.5) rectangle (12.7,12);
\end{tikzpicture}
\end{center}
\caption{Classification of a connection strategy by means of
a binary tree. Local minima of the path correspond to branchpoints and local maxima to blue leaves of the corresponding binary tree $T$. Matching labels in the tree on the right are obtained by left-to-right labelling.}\label{fig:dec_order-tree1}
%\end{figure}%
%\begin{figure}[h]
\begin{center}
\begin{tikzpicture}[scale=0.35, every node/.style={scale=0.6}]
\node (A) at (0,5.6)[circle, draw,fill = black ,label={left:$\mathbf{x}$}] {};
\node (B) at (2,12)[circle, draw,fill = gray, label={$\x_1$}] {};
\node (C) at (4,10) [circle, draw,fill = gray, label={$\x_2$}] {};
\node (D) at (6,6.5)[circle,  draw,fill= gray, label = {left:$\x_3$}] {};
\node (F) at (8,8) [circle, draw,fill= gray, label={left:$\x_4$}] {};
\node (G) at (10,12)[circle, draw,fill = gray, label={$\x_5$}] {};
\node (H) at (12,10)[circle,  draw,fill= gray, label={$\x_6$}] {};
\node (K) at (16,9)[circle, draw,fill = gray, label= {$\x_7$}] {};
\node (L) at (18,12.5) [circle,  draw,fill= gray, label= {$\x_8$}] {};
\node (M) at (20,3)[circle, draw,fill = black, label={left:$\mathbf{y}$}]{};
\draw (A) to (B) to (C) to (D);
\draw (D) to (F) to (G) to (H);
\draw[very thick] (H) to (K);
\draw (K) to (L) to (M);
\draw[dashed] (0.6,6) rectangle (19.5,14.5);
\draw[dotted] (8.6,8.5) rectangle (19,14);
\end{tikzpicture}
\qquad
\begin{tikzpicture}[scale=0.38, every node/.style={scale=0.65}]
\node (B) at (1,6)[circle, draw,fill = cyan, label={$\x_1$}] {};
\node (C) at (3,4) [circle, draw,fill = cyan, label={$\x_2$}] {};
\node (D) at (6,2)[circle,  draw,fill= cyan, label = {$\x_3$}] {};
\node (F) at (9,4) [circle, draw,fill= cyan, label={$\x_4$}] {};
\node (G) at (9,10)[circle, draw,fill = cyan, label={$\x_5$}] {};
\node (H) at (10,8)[circle,  draw,fill= cyan, label={left:$\x_6$}] {};
\node (I) at (11,10) [circle, draw,fill = red] {};
\node (K) at (11,6)[circle, draw,fill = cyan, label= {left:$\x_7$}] {};
\node (L) at (12,8) [circle, draw, fill= cyan, label= {$\x_8$}] {};
\draw (B) to (C) to (D);
\draw (D) to (F) to (K) to (H) to (G);
\draw (H) to (I);
\draw (K) to (L);
\draw[dashed] (0,1.5) rectangle (13,12.5);
\draw[dotted] (7.8,5.5) rectangle (12.7,12);
\end{tikzpicture}
\end{center}
\caption{One connector of the path in Figure \ref{fig:dec_order-tree1} is replaced by an edge of conductance two. This edge corresponds to the red vertex in the tree to which no label and hence no vertex of the path is attached.}
\end{figure}
\end{subfigures}

To each step minimizing path of conductance $k$ between $\x$ and $\y$ we associate a coloured tree $T\in \mathcal{T}_{k-1}^c$ in two steps, 
see Figure~\ref{fig:dec_order-tree1}:\smallskip

\begin{itemize}
\item[(1)]If the connectors of the step minimizing path $P$ of conductance $k$ are $\x_1,\ldots,\x_m$ with $m\leq k-1$, we associate a vector $\mathbf{u} = (u_1,\ldots, u_{k-1})$ to the path defined as follows. We set $u_{w_P(i)} := t_i$ for all $i\in {1,\ldots,m}$ and $u_j = 1$ for all $j\in \set{1,\ldots,k-1}\backslash \set{w_P(1),\ldots,w_P(m)}$. Then 
\[
\mathbf{u} \in\mathcal{U}_{k-1} := \set{\mathbf{u} = (u_1,\ldots,u_{k-1}) \in (0,1]^{k-1} \colon u_i \not= 1\ \text{if}\ u_{i-1}=1}.
\]
\item[(2)] To $\mathbf{u}\in\mathcal{U}_{k-1}$ we associate a coloured tree $T \in \mathcal{T}_{k-1}^c$  
as follows:\smallskip
\begin{itemize}
\item For $k=2$ we have $\mathbf{u} = (u_1)$ and the set $\mathcal{T}_1^c$ contains two trees $T$, each consisting only of the root which may be coloured blue or red. If $\mathbf{u} = (1)$, then $\mathbf{u}$ is associated to the tree $T$ with the red root and otherwise $\mathbf{u}$ is associated to the tree with the blue root.
\smallskip
\item For $k>2$, assume that to every tuple in  $\mathbf{u}\in\mathcal{U}_{j-1}$ with $2\leq j<k$ we have already associated a coloured tree $T\in \mathcal{T}_{j-1}^c$. Let $\mathbf{u} = (u_1,\ldots,u_{k-1})$ and let $u_i$ be the smallest value of $\mathbf{u}$. Then, there exist trees $T_1\in \mathcal{T}_{i-1}^c$ and $T_2\in \mathcal{T}_{k-i-1}^c$ associated to $\mathbf{u}_1 = (u_1,\ldots,u_{i-1})$, resp.\ $\mathbf{u}_2 = (u_{i+1},\ldots, u_{k-1})$. To $\mathbf{u}$ we associate the tree $T\in \mathcal{T}_{k-1}^c$, which has $T_1$ as the left {subtree} of the root and $T_2$ as the right {subtree} and colour the root blue.
\end{itemize}
\end{itemize}
\bigskip

Conversely, given a tree \smash{$T\in \mathcal{T}_{k-1}^c$} {let} $m$ be the number of blue vertices of the tree. We define a labelling
\[
\sigma_T :\set{1,\ldots,m} \to T, i\mapsto \sigma_T(i),
\]
of the blue vertices in $T$ by letting $\sigma_T(i)$ be the $i$th vertex removed in a \emph{left-to-right exploration} of the tree consisting of the blue vertices. This exploration starts with the vertex obtained by starting at the root and going left at any branching until this is no longer possible. Remove this vertex and repeat the procedure unless the removal disconnects a part from the tree or removes the root. If a part is disconnected explore this  part (which is rooted in the right child of the last removed vertex) until it is fully explored and removed, and continue from there
with the remaining tree. If the root is removed while it has a right child, explore the tree rooted in that child until it is fully explored and then stop.
Similarly, define a bijection
\[
\tau_T :\set{1,\ldots, k-1} \to T, i\mapsto \tau_T(i),
\]
by letting $\tau_T(i)$ be the $i$th vertex seen by a left-to-right exploration of \emph{all} vertices on the tree $T$. We also set $\sigma_T^{-1}(\tau_T(0)):= 0$ and $\sigma_T^{-1}(\tau_T(k)) := m+1$. Finally, 
\[
\iota_T:T \to \{0,\ldots,k\}^2, v \mapsto \big(\iota_T^{\ssup{1}}(v),\iota_T^{\ssup{2}}(v)\big)
\]
is defined recursively. For the root $v$ of $T$, we set $\iota_T(v) = (0,k)$. As before, removing $v$ splits $T$ into a left subtree $T_1$ and a right subtree $T_2$. If these trees are nonempty, set $\iota_T(v_1) = \big(\iota_T^{\ssup{1}}(v),\tau_T^{-1}(v)\big)$ for the root $v_1$ of $T_1$, resp. $\iota_T(v_2) = \big(\tau_T^{-1}(v),\iota_T^{\ssup{2}}(v)\big)$ for the root~$v_2$ of $T_2$. Repeat this for the subtrees until $\iota_T(v)$ is defined for all $v\in T$. Thus, for each vertex $v\in T$, its image $\iota_T(v)$ captures\smallskip
\begin{itemize}
\item as its first entry the labelling $\tau_T^{-1}$ of the last vertex seen by a left-to-right exploration before the first vertex of the subtree rooted in $v$ (and set to $0$ if there is no such vertex),
\item as its second entry the labelling $\tau_T^{-1}$ of the first vertex seen by a left-to-right exploration after the last vertex of the subtree rooted {in} $v$ (and set to $k$ if there is no such vertex).\medskip
\end{itemize} 
With these labelings at hand, we now describe 
four restrictions that are satisfied by the marks and locations of the connectors \smash{$\x_1,\ldots, \x_m$} of every step-minimizing path 
connecting {$\x_0= (x_0,t_0)$ and $\x_{m+1}=(x_{m+1},t_{m+1})$} to which the coloured tree $T$ is associated, namely\medskip
\begin{itemize}
\item[(i)]if $\sigma_T(i)$ is the root in $T$, then {$t_i>t_0,t_{m+1}$};\\[-2mm]
\item[(ii)] if $\sigma_T(i)$ is a child of $\sigma_T(j)$ in $T$, then $t_i>t_j$, \\[-2mm]
\item[(iii)] if there is a red leaf $v$ with $i = \sigma_T^{-1}(\tau_T(\iota_T^{\ssup{1}}(v)))$ and $j = \sigma_T^{-1}(\tau_T(\iota_T^{\ssup{2}}(v)))$, then
$$\abs{x_i-x_j}^d \leq  \kappa^{1/\delta}(t_i\wedge t_j)^{-\gamma} (t_i\vee t_j)^{-\gamma/\delta};$$
\item[(iv)] if there is a blue vertex $v$ with $i = \sigma_T^{-1}(\tau_T(\iota_T^{\ssup{1}}(v)))$ and $j = \sigma_T^{-1}(\tau_T(\iota_T^{\ssup{2}}(v)))$, then
$$\abs{x_i-x_j}^d >  \kappa^{1/\delta}(t_i\wedge t_j)^{-\gamma} (t_i\vee t_j)^{-\gamma/\delta}.$$
\smallskip
\end{itemize}
\noindent
Note that whereas (i) and (ii) describe the order of the marks, (iii) and (iv) encode the spatial restrictions on the connectors via the colour of the tree vertices. In (iv), $\x_i$ (resp.~$\x_j$) is the first vertex to the left (resp.\ right) with a smaller mark than \smash{$\x_{\sigma_T^{-1}(v)}$} and the inequality ensures that $\x_i$ and $\x_j$ are far enough apart that no edge with conductance two can exist between them. Conversely, the inequality in (iii) ensures the existence of an edge with conductance two. These conditions motivate the following definitions:
\medskip
\begin{itemize}
\item $M_T$ as the set of vectors $(t_1,\ldots, t_m)\in(0,1)^m$ such that (i), (ii) hold,
\medskip
\item $I^{\mathrm{rl}}_T$ as the set of pairs $(i,j) \in \set{0,\ldots,m+1}^2$ for which a \emph{red leaf} $v$ of {$T$} exists such that $i = \sigma_T^{-1}(\tau_T(\iota_T^{\ssup{1}}(v)))$ and $j = \sigma_T^{-1}(\tau_T(\iota_T^{\ssup{2}}(v)))$,\medskip
\item $I^{\mathrm{b}}_T$ as the set of pairs $(i,j) \in \set{0,\ldots,m+1}^2$ for which a \emph{blue vertex} $v$ of {$T$} exists such that $i = \sigma_T^{-1}(\tau_T(\iota_T^{\ssup{1}}(v)))$ and $j = \sigma_T^{-1}(\tau_T(\iota_T^{\ssup{2}}(v)))$,\medskip
\item and $I^{\mathrm{bc}}_T$ as the set of pairs $(i,i+1) \in \set{0,\ldots,m+1}^2$ for which we have that $\tau_T^{-1}(\sigma_T(i+1)) - \tau_T^{-1}(\sigma_T(i)) = 1$.\medskip
\end{itemize}
{Whereas $M_T$ captures the restrictions on the marks, $I_T^{\mathrm{rl}}$ and $I_T^{b}$ contain the indices to which the the spatial restrictions (iii) and (iv) apply, as for $(i,j)\in I_T^{\mathrm{b}}$ the vertices $\x_i$ and~$\x_j$ cannot be near to each other and for $(i,j)\in I_T^{\mathrm{rl}}$ the vertices $\x_i$ and $\x_j$ have to be that near to each other so that an edge of conductance two exists between them. For each pair $(i,j)\in I_T^{\mathrm{rl}}$ we have $j = i+1$ and $I_T^{\mathrm{rl}}$, $I_T^{\mathrm{bc}}$ form a partition of $\set{(i,i+1):i = 0,\ldots,m}$, because for any $(i,i+1)\in I_T^{\mathrm{bc}}$, there exists an edge of conductance one between the vertices $\x_i$ and $\x_{i+1}$.}\medskip

\begin{proof}[Proof of Lemma \ref{lem:kconn_number_determ_bound}]
For $T\in \mathcal{T}_{k-1}^c$, we define $K_T(\x,\y)$ as the number of step minimizing paths $P$ between $\x$ and $\y$ with $w_P = k$ whose vertices have marks larger than the marks of $\x$ and $\y$ to which $T$ is associated. Then 
\begin{equation*}
\E_{\x,\y} K(\x,\y,k) = \sum_{T\in \mathcal{T}_{k-1}^c} \E_{\x,\y} K_T(\x,\y).
\end{equation*}
If  $k=1$ (or equivalently $T=\emptyset$) we have that $K_T(\x,\y)$ is the indicator of the event that $\x$ and $\y$ are connected by an edge. For $k=2$, if $T$ is the tree consisting of the red root 
 $ K_{T}(\x,\y) = 1\{\abs{x-y}^d \leq \kappa^{1/\delta}(t\wedge s)^{-\gamma} (t\vee s)^{-\gamma/\delta}\}$
 and if $T$ is the tree consisting of the blue root 
\begin{equation*}K_{T}(\x,\y) \leq 1\{\abs{x-y}^d > \kappa^{1/\delta}(t\wedge s)^{-\gamma} (t\vee s)^{-\gamma/\delta}\} \sum_{\substack{\z = (z,u)\\ u> t\vee s}} K_\emptyset(\x,\z)K_\emptyset(\z,\y).\label{finalrootsplit}
\end{equation*}
For $k\geq 3$ we split the tree at the root, i.e.
\begin{equation}
K_T(\x,\y) \leq 1\{\abs{x-y}^d > \kappa^{1/\delta}(t\wedge s)^{-\gamma} (t\vee s)^{-\gamma/\delta}\} \sum_{\substack{\z = (z,u)\\ u> t\vee s}} K_{T_1}(\x,\z)K_{T_2}(\z,\y). \label{rootsplitting}
\end{equation}
where $T_1$ and $T_2$ are the left, resp.\ right,  subtree of $T$ obtained by cutting the root. 
Repeat the step \eqref{rootsplitting} by consecutively splitting the tree at the vertices as seen in the order of a depth first search of the blue vertices in the tree, reducing the product to terms corresponding to empty or single red vertex trees. We get
\begin{align*}
K_T(\x,\y)\leq \sum_{\x_1,\ldots, \x_m} & 1\{(t_1,\ldots, t_m)\in M_T\}\\
& \prod_{(i,j)\in I_T^{\mathrm{b}}} \!\!\! 1\{\abs{x_i-x_j}^d > \kappa^{1/\delta}(t_i\wedge t_j)^{-\gamma}(t_i\vee t_j)^{-\gamma/\delta}\}\\
& \prod_{(i,i+1)\in I_T^{\mathrm{rl}}} \!\!\!\!
K_{v_{(i,i+1)}}(\x_i,\x_{i+1})\!\!\! \prod_{(i,i+1)\in I_T^{\mathrm{bc}}} \!\!\!\! K_{\emptyset}(\x_i,\x_{i+1}) ,
\end{align*}
where $\x_0=\x$, $\x_{m+1}=\y$ and $v_{(i,i+1)}\in T$ is the red leaf associated to $(i,j)$ in the definition of $I_T^{\mathrm{rl}}$. Note that the term $K_{v_{(i,i+1)}}$ contains further spatial restrictions on $\x_i$ and $\x_{i+1}$, ensuring that these vertices are sufficiently close.
Taking expectations {yields}
\begin{align*}
{\E}_{\mathcal{X}}K_T(\x,\y) \leq \sum_{\x_1,\ldots, \x_m}  
& 1\{(t_1,\ldots, t_m)\in M_T\}\\
& \prod_{(i,j)\in I_T^{\mathrm{b}}}
\!\! 1\{\abs{x_i-x_j}^d > \kappa^{1/\delta}(t_i\wedge t_j)^{-\gamma}(t_i\vee t_j)^{-\gamma/\delta}\}\\
&\!\!\! \prod_{(i,i+1)\in I_T^{\mathrm{rl}}} \!\!\!\!
1\{\abs{x_i-x_{i+1}}^d \leq \kappa^{1/\delta}(t_i\wedge t_{i+1})^{-\gamma}(t_i\vee t_{i+1})^{-\gamma/\delta}\}\\
\E_{\mathcal{X}}&\!\!\! \prod_{(i,i+1)\in I_T^{\mathrm{bc}}} \!\!\!\!
1\{\x_i\sim\x_{i+1}\} .
\end{align*}
By Assumption \ref{ass:main1},  we have
\begin{equation}
\E_{\X} \Big[\!\!\! \prod_{(i,i+1)\in I_T^{\mathrm{bc}}} \!\!\!\!
1\{\x_i\sim\x_{i+1}\} \Big] \leq \!\!\! \prod_{(i,i+1)\in I_T^{\mathrm{bc}}}\!\!\!\! e_K(\x_i,\x_{i+1},1). \label{eq:summand_inequality}
\end{equation}
Hence, using the Mecke formula for $m$ points, we get
\begin{align}
{\E}_{\x,\y}K_T(\x,\y)&\leq  \int d\x_1 \cdots\int d\x_m \, 
1\{(t_1,\ldots, t_m)\in M_T\} \notag\\  
&\prod_{(i,j)\in I_T^{\mathrm{b}}}
\!\! 1\{\abs{x_i-x_j}^d > \kappa^{1/\delta}(t_i\wedge t_j)^{-\gamma}(t_i\vee t_j)^{-\gamma/\delta}\}\notag\\
&\!\!\! \prod_{(i,i+1)\in I_T^{\mathrm{rl}}} \!\!\!\!
1\{\abs{x_i-x_{i+1}}^d \leq \kappa^{1/\delta}(t_i\wedge t_{i+1})^{-\gamma}(t_i\vee t_{i+1})^{-\gamma/\delta}\}\notag\\
&\!\!\! \prod_{(i,i+1)\in I_T^{\mathrm{bc}}} \!\!\!\!
e_K(\x_i,\x_{i+1},1). \label{mecke}
\end{align}
What remains to be seen is that when the {right-hand} side in \eqref{mecke} is denoted
$e^T_K(\x, \y)$ and summed over all 
$T\in \mathcal{T}_{k-1}^c$ we obtain $e_K(\x, \y, k)$. This is clearly true when $k=1$
and $k=2$. Otherwise we use \eqref{eq:e_def_rec} to decompose $e_K(\x, \y, k)$. 
By induction,  the factors in this decomposition can be represented as in~\eqref{mecke} and we obtain
\begin{align*}
e_K(\x,\y,k)  =  &1\{\abs{x-y}^d > \kappa^{1/\delta}(t\wedge s)^{-\gamma} (t\vee s)^{-\gamma/\delta}\} \\
&\sum_{i=1}^{k-1} \sum_{T_1\in \mathcal{T}_{i-1}^c} \sum_{T_2\in \mathcal{T}_{k-1-i}^c} \,\,\,
\int\limits_{\mathbb{R}^d\times (t\vee s,1]} \mathd \z  \, e^{T_1}_K(\x, \z)e^{T_2}_K(\z, \y).
\end{align*}
{Writing the terms $e^{T_1}_K(\x, \z)$ and $e^{T_2}_K(\z, \y)$ as in \eqref{mecke}
as integrals over $\x_1, \ldots, \x_{m_1}$
and $\x_{m_1+2}, \ldots, \x_{m}$ we can insert
$\z$ as $\x_{m_1+1}$ and note that the conditions and terms emerging in that integral 
are exactly the same as in \eqref{mecke} for the tree $T$ with $T_1$ and $T_2$ as left and right subtree of the root. Indeed,\medskip
\begin{itemize}
\item the vector $(t_1,\ldots,t_m)$ of the marks of $\x_1,\ldots,\x_m$ is an element of $M_T$ iff \\$(t_1,\ldots,t_{m_1})\in M_{T_1}$, $(t_{m_1+2},\ldots,t_m)\in M_{T_2}$ and $t_{m_1+1} > s\vee t$, \medskip
\item the spatial conditions described by $I_T^{\mathrm{b}}$ are fulfilled iff $x_1,\ldots,x_{m_1}$ fulfills the ones decribed by $I_{T_1}^{\mathrm{b}}$, $x_{m_1+2},\ldots,x_{m}$ the ones by $I_{T_2}^{\mathrm{b}}$ and
$${\abs{x-y}^d > \kappa^{1/\delta}(t\wedge s)^{-\gamma} (t\vee s)^{-\gamma/\delta}},$$ 
\item $I^{\mathrm{rl}}_T$ is the union of $I^{\mathrm{rl}}_{T_1}$ and $I^{\mathrm{rl}}_{T_2}$ where the values of the pairs of $I^{\mathrm{rl}}_{T_2}$ have been increased by $m_1+1$ and in the same way $I^{\mathrm{bc}}_{T}$ directly emerges from $I^{\mathrm{bc}}_{T_1}$ and $I^{\mathrm{bc}}_{T_2}$. \medskip
\end{itemize}
Hence, $e_K(\x,\y,k)$ can be obtained by summing $e_K^T(\x,\y)$ over all $T\in \mathcal{T}_{k-1}^c$.
}\end{proof}
\bigskip

\begin{lemma}[$k$-connection lemma]\label{lem:kconn}
Let $\x = (x,t), \y = (y,s)$ be two given vertices with $\smash{\abs{x-y}^d > \kappa^{1/\delta}(t\wedge s)^{-\gamma} (t\vee s)^{-\gamma/\delta}}$ and {$0<\ell<\frac{1}{e}$} such that $\ell<t\vee s$. Then there exists $C>1$ such that, for $k\geq 3$, we have
\[
e_K(\x,\y,k) \leq C^{k-1}\ell^{(\gaus{\frac{k}{2}}-1)(1-\gamma-\gamma/\delta)}\log(\frac{1}{\ell})^{k_{*}} \kappa(t\wedge s)^{-\gamma \delta}(t\vee s)^{-\gamma}\abs{x-y}^{-d\delta}
\]
where $k_*:= k \mod 2$.
\end{lemma}

\begin{proof}
Choose $C>1$ such that $C$ is larger than the constants appearing in Lemma \ref{twoconnlem} and Lemmas~\ref{cal_lem_1} and \ref{cal_lem_2} of the appendix. We now show by induction that 
\begin{align}
e_K(\x,\y,k) \leq \Cat(k-1) C^{k-1} & \, \ell^{(\gaus{\frac{k}{2}}-1)(1-\gamma-\gamma/\delta)}\notag \\ &\log(\tfrac{1}{\ell})^{k_{*}}\rho\big(\kappa^{-1/\delta} (t\wedge s)^{\gamma}(t\vee s)^{\gamma/\delta}\abs{x-y}^{d}\big) \label{eq:kconn_bound}
\end{align}
holds for all $k\geq 2$, where $\Cat(k-1)$ is the $(k-1)$-th Catalan number. Note that, for $k\geq 2$, it holds $e_K(\x,\y,k) \leq 1$ for $\smash{\abs{x-y}^d \leq \kappa^{1/\delta}(t\wedge s)^{-\gamma} (t\vee s)^{-\gamma/\delta}}$. Thus, it remains to show \eqref{eq:kconn_bound} under the condition $\smash{\abs{x-y}^d > \kappa^{1/\delta}(t\wedge s)^{-\gamma} (t\vee s)^{-\gamma/\delta}}$.
For $k=2$, the bound \eqref{eq:kconn_bound} is already established by Lemma \ref{twoconnlem}. If $k=3$ and 
$\smash{\abs{x-y}^d > \kappa^{1/\delta}(t\wedge s)^{-\gamma} (t\vee s)^{-\gamma/\delta}}$, by~\eqref{eq:e_def_rec} 
we have
\begin{equation*}
e_K(\x,\y,3) \leq \int\limits_{t\vee s}^1 \mathd u \int\limits_{\mathbb{R}^d} \mathd z e(\x,\z,1)e(\z,\y,2) + \int\limits_{t\vee s}^1 \mathd u \int\limits_{\mathbb{R}^d} \mathd z e(\x,\z,2)e(\z,\y,1).
\end{equation*}
{Using the bounds established in Lemma \ref{twoconnlem} together with Lemma \ref{cal_lem_2} leads to}
\begin{equation*}
e_K(\x,\y,3) \leq 2C^2\log(\tfrac{1}{\ell}) \kappa (t\wedge s)^{-\gamma\delta}(t\vee s)^{-\gamma}\abs{x-y}^{-d \delta}.
\end{equation*}
Let $k\geq 4$ and assume that \eqref{eq:kconn_bound} holds for all $j = 2,\ldots, k-1$. For $\x, \y$ such that  \smash{$\abs{x-y}^d > \kappa^{1/\delta} (t\wedge s)^{-\gamma}(t \vee s)^{-\gamma/\delta}$}, by~$\eqref{eq:e_def_rec}$,
\[
e_K(\x,\y,k) = \sum_{i=1}^{k-1}\int\limits_{\mathbb{R}^d\times (t\vee s,1]} \mathd \z e_K(\x,\z,i)e_K(\z,\y,k-i).
\]
With \eqref{eq:kconn_bound} we hence get,
\begin{align*}
e_K& (\x,\y,k) \leq  \,\, \sum_{i=2}^{k-2} \Big[ C^{k-2}\ell^{(\gaus{\frac{i}{2}}+\gaus{\frac{k-i}{2}}-2)(1-\gamma-\gamma/\delta)}\log(\tfrac{1}{\ell})^{i_{*}+(k-i)_{*}}\Cat(i-1)\\ & \Cat(k-i-1) \int\limits_{t\vee s}^1 \mathd u \int\limits_{\mathbb{R}^d} \mathd z\rho\big(\kappa^{-1/\delta} t^{\gamma}u^{\gamma/\delta}\abs{x-z}^{d}\big)\rho\big(\kappa^{-1/\delta} s^{\gamma}u^{\gamma/\delta}\abs{z-y}^{d}\big)\Big]\\
& + C^{k-2}\ell^{(\gaus{\frac{k-1}{2}}-1)(1-\gamma-\gamma/\delta)}\log(\tfrac{1}{\ell})^{(k-1)_*}\Cat(0)\Cat(k-2)\\
&\phantom{pp}\times  \int\limits_{t\vee s}^1 \mathd u \int\limits_{\mathbb{R}^d} \mathd z\rho\big(\kappa^{-1/\delta} t^{\gamma}u^{1-\gamma}\abs{x-z}^{d}\big)\rho\big(\kappa^{-1/\delta} s^{\gamma}u^{\gamma/\delta}\abs{z-y}^{d}\big)
\\
& +  C^{k-2}\ell^{(\gaus{\frac{k-1}{2}}-1)(1-\gamma-\gamma/\delta)}\log(\tfrac{1}{\ell})^{(k-1)_*}\Cat(k-2)\Cat(0) \\
&\phantom{pp}\times \int\limits_{t\vee s}^1 \mathd u \int\limits_{\mathbb{R}^d} \mathd z\rho\big(\kappa^{-1/\delta} t^{\gamma}u^{\gamma/\delta}\abs{x-z}^{d}\big)\rho\big(\kappa^{-1/\delta} s^{\gamma}u^{1-\gamma}\abs{z-y}^{d}\big).
\end{align*}
Using Lemma \ref{cal_lem_1} and Lemma \ref{cal_lem_2} the last expression can be further bounded by
\begin{align*}
&\kappa(t\wedge s)^{-\gamma \delta}(t\vee s)^{-\gamma}\abs{x-y}^{-d\delta}C^{k-1} \\ & \times \bigg[\sum_{i=2}^{k-2} \Cat(i-1)\Cat(k-i-1) \ell^{(\gaus{\frac{i}{2}}+\gaus{\frac{k-i}{2}}-1)(1-\gamma-\gamma/\delta)}
 \log(\tfrac{1}{\ell})^{i_{*}+(k-i)_{*}}\\ &
\phantom{XXX} + 2l^{(\gaus{\frac{k-1}{2}}-1)(1-\gamma-\gamma/\delta)}\log(\tfrac{1}{\ell})^{(k-1)_*+1}\Cat(0)\Cat(k-2)\bigg].
\end{align*}
If $k$ is even, $i$ and $k-i$ need to be either both even or both odd, for $i=1,\ldots,k-1$. Since $\ell>0$ is chosen small enough that $\log(\frac{1}{\ell})^2 < \ell^{1-\gamma-\gamma/\delta}$, we have that in both cases
\[
\ell^{(\gaus{\frac{i}{2}}+\gaus{\frac{k-i}{2}}-1)(1-\gamma-\gamma/\delta)}\log(\tfrac{1}{\ell})^{i_{*}+(k-i)_{*}} < \ell^{(\gaus{\frac{k}{2}}-1)(1-\gamma-\gamma/\delta)}.
\]
If $k$ is odd, an analogous observation leads to 
\[
\ell^{(\gaus{\frac{i}{2}}+\gaus{\frac{k-i}{2}}-1)(1-\gamma-\gamma/\delta)}\log(\tfrac{1}{\ell})^{i_{*}+(k-i)_{*}} < \ell^{(\gaus{\frac{k}{2}}-1)(1-\gamma-\gamma/\delta)}\log(\tfrac{1}{\ell}).
\]
Hence, we have
\begin{align*}
e_K(\x,\y,k) &\leq \kappa(t\wedge s)^{-\gamma \delta}(t\vee s)^{-\gamma}\abs{x-y}^{-d\delta}C^{k-1}\ell^{(\gaus{\frac{k}{2}}-1)(1-\gamma-\gamma/\delta)}\\
& \phantom{loemloeml} \times\log(\tfrac{1}{\ell})^{k_*}\sum_{i=1}^{k-1} \Cat(i-1)\Cat(k-i-1)\\
&\leq \kappa(t\wedge s)^{-\gamma \delta}(t\vee s)^{-\gamma}\abs{x-y}^{-d\delta}\Cat(k-1) C^{k-1}\ell^{(\gaus{\frac{k}{2}}-1)(1-\gamma-\gamma/\delta)}\log(\tfrac{1}{\ell})^{k_*}
\end{align*}
and \eqref{eq:kconn_bound} holds for $k$. The observation that $\Cat(k)\leq 4^k$ concludes the proof.
\end{proof}

\paragraph{Probability bounds for bad paths.}
With Lemma \ref{lem:kconn} we can establish a bound for \smash{$\E_{\x,\y}N(\x,\y,n)$}, {recall the definitions in Section~\ref{Ndef}.} As in \eqref{eq:recdecomp} and \eqref{eq:recdecomp2}, we have
\begin{equation}
N(\x,\y,n) \leq K(\x,\y,n) + \sum_{k=1}^{n-1} \sum_{\substack{\z = (z,u)\\ t>u> \ell_{n-k}\vee s}} N(\x,\z,n-k)K(\z,\y,k) \,\, .\label{eq:recdecomp_rv_N}
\end{equation}
Here $\z$ is the most powerful vertex of the path disregarding $\y$ and connects to $\y$ via less powerful vertices.
As done for $K(\x,\y,k)$ in the previous section we compare $\E_{\x,\y}N(\x,\y,n)$ with a deterministic mapping  
\[
{e_N \colon(\mathbb{R}^d\times (0,1])^2 \times \mathbb{N} \to [0,\infty),}
\]
defined as
\[
e_N(\x,\y,1) = \rho(\kappa^{-1/\delta}(t\wedge s)^{\gamma}(t\vee s)^{1-\gamma}\abs{x-y}^d),\quad \text{for}\ \x,\y \in \mathbb{R}^d\times (0,1],
\]
and for $n\geq 2$
\begin{equation}
e_N(\x,\y,n) = e_K(\x,\y,n) + \sum_{i=1}^{n-1}\int\limits_{\mathbb{R}^d\times (\ell_{n-k}\vee s,1]} \mathd \z e_N(\x,\z,n-k)e_K(\z,\y,k), \label{eq:e_def_rec_N}
\end{equation}
for $\x,\y \in \mathbb{R}^d\times (0,1],$
if $\smash{\abs{x-y}^d > \kappa^{1/\delta}(t\wedge s)^{-\gamma} (t\vee s)^{-\gamma/\delta}}$, and otherwise $e_N(\x,\y,2) = 1$ and $e_N(\x,\y,n) = 0$ for $n\geq 3$. 
\begin{lemma}\label{lem:path_number_determ_bound}
Let $\x, \y \in \mathbb{R}^d\times (0,1]$ be two given vertices. Then, for all $n\in \mathbb{N}$, we have
\begin{equation}
\E_{\x,\y} N(\x,\y,n) \leq e_N(\x,\y,n).
\end{equation}
\end{lemma}
\begin{proof}
First recall that for $\smash{\abs{x-y}^d \leq \kappa^{1/\delta}(t\wedge s)^{-\gamma} (t\vee s)^{-\gamma/\delta}}$ we have $N(\x,\y,n) = 0$ for $n\geq 3$ and $N(\x,\y,2) = 1$. Thus in this case $N(\x,\y,n)$ is equal to $e_N(\x,\y,n)$ and consequently {their expectations are equal}. Otherwise, the proof follows the same argument as in Lemma \ref{lem:kconn_number_determ_bound}, where we again classify the possible connection strategies between $\x$ 
and~$\y$ through coloured binary trees. We therefore only briefly present the required class of trees, explain the association of a path to the corresponding tree and the restrictions on {marks} and space which a step minimizing path that associates to $T$ has to satisfy.\pagebreak[3]\medskip 

{Let $\mathcal{T}_{n}^{cb}$ be a class of coloured rooted binary trees with $n$ vertices which are constructed as follows. For $k\leq n$, we have a \emph{backbone} 
consisting of $k$ vertices, starting with the root followed by $k-1$ vertices, each a left child of the previous one. The last vertex in this line is coloured red, the others blue.
Let $i_1,\ldots,i_k\in \mathbb{N}$ with $i_1+\ldots + i_k = n-k$. A  tree $T\in \mathcal{T}_{n}^{cb}$ is formed by attaching to the $j$-th vertex (as seen by a left-to-right exploration of the backbone) a coloured subtree $T_j\in \mathcal{T}_{i_j}^{c}$ rooted in its right child, for $j=1,\ldots,k$.}\medskip 

{To any path $P = (\x_0,\x_1,\ldots,\x_{m+1})$ with {$\x_0 = \x$ and $\x_{m+1} = \y$} where the connecting vertices have larger marks than $\y$ we associate a tree $T \in \mathcal{T}_n^{cb}$ as follows. We say $\x_i$ is a \emph{powerful} vertex of $P$ if $t_i\leq t_j$ for all $j=0,\ldots,i-1$. By definition, the vertices $\x_0$ and $\x_m$ are always powerful vertices. We denote by $\set{\x_{i_1},\ldots,\x_{i_{k+1}}}$ the set of powerful vertices keeping the order in the path. Then two consecutive powerful vertices $\x_{i_j}$ and $\x_{i_{j+1}}$ are, by definition, connected  via a path of connectors $\x_{i_j+1},\ldots \x_{i_{j+1}-1}$ of conductance $w_j:=w_P(i_{j+1})-w_P(i_j)$. If $w_j\geq 2$, associate the connectors of the path connecting $\x_{i_j}$ and $\x_{i_{j+1}}$ to a non-empty coloured tree \smash{$T_j \in \mathcal{T}_{w_j-1}^c$} as in the proof of 
Lemma~\ref{lem:kconn_number_determ_bound}.
Let $T \in \mathcal{T}_n^{cb}$ be the coloured tree which has a backbone of length $k$ and where $T_{j}$ is attached to the $j$-th vertex (as seen by a left-to-right exploration of the backbone) {such} that its right child is the root of $T_j$, see Figure~5 for an example.} 
\bigskip\\

\begin{figure}[h]
\begin{center}
\begin{tikzpicture}[scale=0.30, every node/.style={scale=0.7}]
\node (A) at (0,6.5)[circle, draw, fill= black ,label={left:$\x$}] {};
\node (B) at (2,9.9)[circle, draw, fill=gray, label={$\x_1$}] {};
\node (C) at (4,5.5)[circle, draw, fill= black, label={left:$\x_2$}] {};
\node (D) at (6,10)[circle,  draw, fill=gray, label = {$\x_3$}] {};
\node (E) at (8,10.5)[circle,  draw, fill=gray, label = {$\x_4$}] {};
\node (F) at (10, 4)[circle, draw,fill = black, label={left:$\x_5$}] {};
\node (G) at (12,9.7) [circle, draw, fill=gray, label={$\x_6$}] {};
\node (H) at (14,10.3)[circle, draw,fill=gray, label={$\x_7$}] {};
\node (I) at (16,8.8) [circle, draw,fill= gray, label={$\x_8$}] {};
\node (J) at (18,9.8) [circle, draw,fill= gray, label={$\x_9$}] {};
\node (K) at (20,0)[circle, fill = black, label={left:$\y$}] {};
\draw (A) to (B) to (C);
\draw[very thick](C) to (D);
\draw (D) to (E) to (F) to (G) to (H) to (I) to (J) to (K);
\draw[loosely dotted] (1,9.4) rectangle (3,11.9);
\draw[dashed] (5,9.5) rectangle (9,12.5);
\draw[dotted] (11,8.3) rectangle (19,12.3);
\end{tikzpicture}
\qquad
\begin{tikzpicture}[scale=0.40, every node/.style={scale=0.7}]
\node (F) at (8, 1)[circle, draw,fill = cyan, label={left:$\x_5$}] {};
\node (C) at (5,3)[circle, draw, fill= cyan, label={left:$\x_2$}] {};
\node (D2) at (6,7)[circle,  draw, fill= red, label = {}] {};
\node (A) at (2,5)[circle, draw, fill= red, label={left:$\x$}] {};
\node (B) at (4,7)[circle, draw, fill= cyan, label={$\x_1$}] {};
\node (D) at (7,5)[circle,  draw, fill= cyan, label = {$\x_3$}] {};
\node (E) at (8,7)[circle,  draw, fill= cyan, label = {$\x_4$}] {};
\node (J) at (10,3) [circle, draw,fill= cyan, label={$\x_8$}] {};
\node (G) at (11,5) [circle, draw, fill= cyan, label={$\x_{9}$}] {};
\node (H) at (9,5)[circle, draw,fill= cyan, label={$\x_6$}] {};
\node (I) at (10,7) [circle, draw,fill= cyan, label={$\x_7$}] {};
\draw (A) to (C) to (F);
\draw[dashed] (A) to (B);
\draw[dashed] (C) to (D) to (E);
\draw[dashed] (F) to (J) to (G);
\draw[dashed] (J) to (H) to (I);
\draw[dashed] (D) to (D2);
\draw[loosely dotted] (3,6.5) rectangle (5,8.5);
\draw[dashed] (5.6,4.5) rectangle (8.4,8.5);
\draw[dotted] (8.6,2.5) rectangle (11.4,8.5);
\end{tikzpicture}
\caption{Associating a coloured binary tree to a path. The powerful vertices of the path are indicated in black. We have $k=3$ 
vertices on the backbone. The three trees attached to the backbone are constructed as in Figure~4, where the vertices with the smallest mark on the connecting paths are the roots which are attached as right children to the backbone.}\label{fig:dec_order-tree2}
\end{center}
\end{figure}
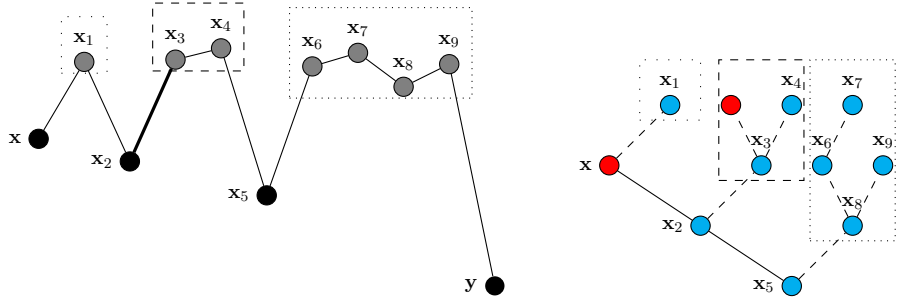

Given a tree $T\in \mathcal{T}_{n}^{cb}$, let $m$ be the number of blue vertices of the tree and $k$ the number of vertices of the backbone. As in the proof of Lemma~\ref{lem:kconn_number_determ_bound}, we define a labelling 
\[
\sigma_T:\{0,\ldots,m\} \to T, i\mapsto\sigma_T(i),
\]
by letting $\sigma_T(0)$ be the red vertex on the backbone and $\sigma_T(i)$ be the $i$th vertex seen by a left-to-right exploration of the blue vertices of $T$. Define the bijection
\[
\tau_T:\{0,\ldots,n-1\} \to T, i \mapsto \tau_T(i),
\]
by letting $\tau_T(0)$ be the red vertex on the backbone and $\tau_T(i)$ be the $(i+1)$st vertex seen by a left-to-right exploration of all other vertices of the tree. Denote by $v_1,\ldots,v_k$ the vertices of the backbone of $T$ and $T_1,\ldots,T_k$ the subtrees rooted in their right child. Set $i_j := \sigma^{-1}(v_{j})$, for $i=1,\ldots,k$, and $i_{k+1} := m+1$. 
Then, the following restrictions on {marks} and space are satisfied by the vertices $\x_1,\ldots,\x_{m}$ of any path connecting $\x_0 = \x$ and $\x_{m+1} = \y$ to which $T$ is associated:
\medskip
\begin{enumerate}[(i)]
\item $t_{i_{j}}>t_{i_{j+1}}$, for $j=1,\ldots,k$, \medskip
\item if there exists a vertex $v_j$ of the backbone with $\tau_T^{-1}(v_j) \geq 2$, then
\[
\abs{x_0 - x_{i_j}}^d > \kappa^{1/\delta}(t_0\wedge t_{i_j})^{-\gamma} (t_0\vee t_{i_j})^{-\gamma/\delta},
\]
\item for $j=1,\ldots,k$, the vertices $\x_{i_j+1},\ldots,\x_{i_{j+1}-1}$ satisfy the four restrictions on {marks} and space given by the coloured tree $T_i$ and $\x_{i_j},\x_{i_{j+1}}$ as described prior to the proof of Lemma~\ref{lem:kconn_number_determ_bound}.\medskip
\end{enumerate}
For $T\in \mathcal{T}_{n}^{cb}$, we define $N_T(\x,\y)$ as the number of step minimizing paths to which $T$ is associated. Denote again by $v_1,\ldots,v_k$ the vertices of the backbone of $T$ and set $i_j := \sigma^{-1}(v_j)$, $i_{k+1} := m+1$. Splitting the tree at each blue vertex of the backbone leads to
\begin{align*}
N_T(\x,\y) \leq \sum_{\x_{i_2},\ldots,\x_{i_{k}}}&1\{t_{i_1}>\ldots>t_{i_{k+1}} \} \\
&\!\!\!\! \prod_{\substack{2\leq j\leq k\\ \tau_T^{-1}(v_j)\geq 2}} 1\{\abs{x_0 - x_{i_j}}^d > \kappa^{1/\delta}(t_0\wedge t_{i_j})^{-\gamma} (t_0\vee t_{i_j})^{-\gamma/\delta}\}\\
&\prod_{1\leq j\leq k} K_{T_j}(\x_{i_j},\x_{i_{j+1}}),
\end{align*}
where $T_j$ is the subtree attached to the right child of $v_j$. Proceeding for each $K_{T_j}$ and using the iterative structure of $e_N$ as in the proof of Lemma~\ref{lem:kconn_number_determ_bound} yields the result.
\end{proof}

As a path described by the event $\tilde{A}_n^{_{(\x)}}$ (recall the definition from Section~\ref{lem:kconn_number_determ_bound}) has a restriction on the mark but not on the location of its last vertex, we can use the integral 
\begin{equation}\label{hia}
\int\limits_{\mathbb{R}^d} \mathd y \,\, \E_{\x,\y}N(\x,\y,n),
\end{equation} 
with $\y = (y,s)$ and $s$ smaller than some yet to be determined value to bound $\P_{\x}(\tilde{A}_n^{_{(\x)}})$. Thus, we define for given $\x = (x,t)$ and $n\in \mathbb{N}$ the mapping $\mu_n^\x:(0,t] \to [0,\infty)$ by 
\begin{equation}
\mu_n^\x(s) := \int_{\mathbb{R}^d}\mathd y \, e_N(\x,\y,n), \quad \text{for}\ s\in (0,t], \y = (y,s).\label{eq:mu_def}
\end{equation} 
%This mapping $\mu_n^\x(s)$ is used to bound \eqref{hia} in terms of the mark of the final vertex~$\y$. 
Recall that we write $k_*:= k\pmod 2$ and $I_\rho := \int\mathd x \, \rho(\kappa^{-1/\delta}\abs{x}^d)$. By the definition of $e_N(\x,\y,1)$ we have
$\mu_1^\x(s) \leq I_\rho s^{-\gamma} t^{\gamma-1}$, for
$s\in (0,t]$,
and, for $n\geq 2$, with a short calculation using Lemma \ref{lem:kconn} we get the recursive property
\begin{align}
\mu_n^\x(s) \leq  &\,\, I_\rho C^{n-1}\ell_0^{(\gaus{\frac{n}{2}}-1)(1-\gamma-\gamma/\delta)}\log(\tfrac{1}{\ell_0})^{n_{*}} s^{-\gamma}t^{-\gamma/\delta}\label{eq:mu_ineq_rec_1}\\
&+ \sum_{k=2}^{n-1} I_\rho C^{k-1}\ell_{n-k}^{(\gaus{\frac{k}{2}}-1)(1-\gamma-\gamma/\delta)}\log(\tfrac{1}{\ell_{n-k}})^{k_{*}} s^{-\gamma} \int\limits_{\ell_{n-k}}^t \mathd u \mu_{n-k}^{\x}(u)u^{-\gamma/\delta}\label{eq:mu_ineq_rec_2}\\
&+ I_\rho s^{-\gamma} \int\limits_{\ell_{n-1}}^t \mathd u \, \mu_{n-1}^\x(u) u^{\gamma-1},\quad \text{for}\ s\in (0,t],\label{eq:mu_ineq_rec_3}
\end{align}
where $C>0$ is the constant {from} Lemma~\ref{lem:kconn}. {Here, the first summand \eqref{eq:mu_ineq_rec_1} corresponds to the first summand of \eqref{eq:recdecomp_rv_N}, {i.e.} the number of paths with conductance~$n$ where the first vertex $\x$ and the last vertex with mark $s$ are the two most powerful vertices of the path. The summands \eqref{eq:mu_ineq_rec_2} and \eqref{eq:mu_ineq_rec_3} describe the second summand of \eqref{eq:recdecomp_rv_N}, where \eqref{eq:mu_ineq_rec_3} covers the case that the last vertex of a path is directly connected to the preceding most powerful vertex.}\bigskip \\
{Using the recursive inequality in \eqref{eq:mu_ineq_rec_1} - \eqref{eq:mu_ineq_rec_3} we now establish bounds for $\mu_n^{\x}$.} {To make the proof more transparent we continue working with a general sequence $(\ell_n)_{n\in\mathbb{N}_0}$ assuming only that it is at least exponentially decaying, i.e. {for any $b>0$ it holds} that $\ell_{n+2} < b \ell_n$. We choose $b>0$ small enough such that $\sum_{j=2}^\infty b^{(\gamma+\gamma/\delta-1)\frac{(j-3)(j-1)}{8}}$ converges. This choice is possible because in our regime  $\gamma+\gamma/\delta$ is larger than one.  
We denote the limit of the series by $c_{b}>1$. As we have already seen for the optimal path structure in Section \ref{subsec:outline}, the chosen sequence $(\ell_n)_{n\in\mathbb{N}_0}$ {decays much faster than any exponential rate} so that this assumption will not have any effect on the result.} 
Without loss of generality we may additionally assume $\ell_0<\frac{1}{e}$.\medskip

\begin{lemma}\label{lem:mu_bound}
Let $\x = (x,t)$ be a given vertex and let the sequence $(\ell_n)_{n\in\mathbb{N}_0}$ be at least exponentially decaying with $\ell_0<t\wedge \frac{1}{e}$. Then, there exists a constant $c$ such that, for $n\in \mathbb{N}$, we have
\begin{equation}
\mu_n^\x(s) \leq C_n s^{-\gamma},\quad \text{for}\ s\in (0,t], \label{eq:mu_bound}
\end{equation}
where 
\begin{equation}
C_{n+2} = c^2 \ell_{n}^{1-\gamma-\gamma/\delta}C_n + c \log\big(\tfrac{1}{\ell_{n+1}}\big)C_{n+1} \label{eq:c_n_def}
\end{equation}
and 
\[
C_1 = c\ell_0^{\gamma-1},\quad C_2 = c^2\ell_0^{-\gamma/\delta} + c\log(\tfrac{1}{\ell_1})C_1.\] 
\end{lemma}
\begin{proof}
We choose the constant $c>0$ such that it is larger than $\frac{I_{\rho}c_b}{(\gamma + \gamma/\delta-1)\wedge 1}$ and larger than the constant $C$ {from} Lemma \ref{lem:kconn}. Since this also implies that $c>I_{\rho}$, by the definition of $\mu_1^\x$ we have
\[
\mu_1^\x(s) = I_{\rho} s^{-\gamma}t^{\gamma-1} \leq c\ell_0^{\gamma-1}s^{-\gamma} = C_1s^{-\gamma}\quad \text{for}\ s\in (0,t].
\]
For $n=2$, the recursive inequality for $\mu_2^\x$ yields
\[
\mu_2^\x(s) \leq I_\rho C s^{-\gamma}t^{-\gamma/\delta} + I_\rho s^{-\gamma} \int\limits_{\ell_1}^t \mathd u \, \mu_1^\x(u)u^{\gamma-1}\quad \text{for}\ s\in (0,t].
\]
Using the already established bound for $n=1$ we have
\begin{align*}
\mu_2^\x(s) &\leq c^2\ell_0^{-\gamma/\delta} s^{-\gamma} + I_\rho s^{-\gamma} \int\limits_{\ell_1}^t \mathd u C_1 u^{-1}\\
&\leq c^2\ell_0^{-\gamma/\delta} s^{-\gamma} + c\log(\tfrac{1}{\ell_1})C_1 s^{-\gamma}
 =: C_2 s^{-\gamma}\quad \text{for}\ s\in(0,t].
\end{align*}
Now let $n\geq 3$ and we assume that \eqref{eq:mu_bound} holds for all $\tilde{n}\leq n-1$. Then, using the already established bounds and the recursive inequality property we have
\begin{align*}
\mu_n^\x(s) & \leq I_\rho C^{n-1}\ell_0^{(\gaus{\frac{n}{2}}-1)(1-\gamma-\gamma/\delta)}\log(\tfrac{1}{\ell_0})^{n_{*}} s^{-\gamma}t^{-\gamma/\delta}\\
& \phantom{xxxx} + \sum_{k=2}^{n-1} I_\rho C^{k-1}\ell_{n-k}^{(\gaus{\frac{k}{2}}-1)(1-\gamma-\gamma/\delta)}\log(\tfrac{1}{\ell_{n-k}})^{k_{*}} s^{-\gamma} \int\limits_{\ell_{n-k}}^t \mathd u C_{n-k}u^{-\gamma-\gamma/\delta}\\
& \phantom{xxxx} + I_\rho s^{-\gamma} \int\limits_{\ell_n-1}^t \mathd u \, C_{n-1} u^{-1}\\
&\leq I_\rho C^{n-1}\ell_0^{(\gaus{\frac{n}{2}}-1)(1-\gamma-\gamma/\delta)}\log(\tfrac{1}{\ell_0})^{n_{*}} s^{-\gamma}t^{-\gamma/\delta}\\
&\phantom{xxxx} + \sum_{k=2}^{n-1} \frac{I_\rho}{\gamma+\gamma/\delta-1} C^{k-1}\ell_{n-k}^{(\gaus{\frac{k}{2}})(1-\gamma-\gamma/\delta)}\log(\tfrac{1}{\ell_{n-k}})^{k_{*}}C_{n-k} s^{-\gamma}\\
& \phantom{xxxx} + I_\rho C_{n-1} \log(\tfrac{1}{\ell_{n-1}}) s^{-\gamma}\quad \text{for}\ s\in (0,1].
\end{align*}
Assume for the moment that
\begin{align}
\begin{split}
\sum_{k=2}^{n-1} & c^{k-1} C_{n-k}\ell_{n-k}^{(\gaus{\frac{k}{2}})(1-\gamma-\gamma/\delta)}\log(\tfrac{1}{\ell_{n-k}})^{k_{*}} + c^{n-1}\ell_{0}^{(\gaus{\frac{n}{2}}-1)(1-\gamma-\gamma/\delta)}\ell_0^{-\gamma/\delta}\log(\tfrac{1}{\ell_{0}})^{n_{*}} \\ 
&\leq c_b cC_{n-2}\ell_{n-2}^{1-\gamma-\gamma/\delta} \label{eq:exponential_c_n_bound} 
\end{split}
\end{align}
holds. Then, as $c>C$, the term  $\mu_n^\x(s)$ can be further bounded by
\[
\tfrac{I_\rho c_b}{(\gamma+\gamma/\delta-1)\wedge 1} c C_{n-2}\ell_{n-2}^{1-\gamma-\gamma/\delta} s^{-\gamma} + I_\rho C_{n-1} \log(\tfrac{1}{\ell_{n-1}}) s^{-\gamma},
\]
which by \eqref{eq:c_n_def} is smaller than $C_n s^{-\gamma}$ for $s\in (0,t]$. Hence, by induction the stated inequality holds for all $n\in \mathbb{N}$.\smallskip

It remains to show that \eqref{eq:exponential_c_n_bound} holds. If $k$ is even, a {repeated} application of~\eqref{eq:c_n_def} and $\ell_{n+2} <b\ell_n$ yields
\[
c^{k-1} C_{n-k}\ell_{n-k}^{(\gaus{\frac{k}{2}})(1-\gamma-\gamma/\delta)}\log(\tfrac{1}{\ell_{n-k}})^{k_{*}} \leq cC_{n-2} \ell_{n-2}^{1-\gamma-\gamma/\delta} b^{(\gamma+\gamma/\delta-1)(\sum_{j=0}^{\frac{k-2}{2}}j)}.
\]
If $k$ is odd a similar calculation leads to 
\[
c^{k-1} C_{n-k}\ell_{n-k}^{(\gaus{\frac{k}{2}})(1-\gamma-\gamma/\delta)}\log(\tfrac{1}{\ell_{n-k}})^{k_{*}} \leq cC_{n-2} \ell_{n-2}^{1-\gamma-\gamma/\delta} b^{(\gamma+\gamma/\delta-1)(\sum_{j=0}^{\frac{k-3}{2}}j)}.
\]
{Distinguishing whether n is even or odd, the second term of \eqref{eq:exponential_c_n_bound} can be bounded in a similar way and so the whole expression} can be bounded by
\[
cC_{n-2}\ell_{n-2}^{1-\gamma-\gamma/\delta} \left(\sum_{\substack{k=2\\\text{k even}}}^n b^{(\gamma+\gamma/\delta-1)\frac{(k-2)k}{8}} + \sum_{\substack{k=3\\\text{k odd}}}^n b^{(\gamma+\gamma/\delta-1)\frac{(k-3)(k-1)}{8}}\right),
\]
where the two sums can be bounded by $c_b$ which implies that \eqref{eq:exponential_c_n_bound} holds.
\end{proof}

Notice that, as stated in Section \ref{subsec:outline}, the inequality \eqref{eq:exponential_c_n_bound} shows us that the major contribution to the expected value of $N(\x,\y,n)$ comes from the paths where the two most powerful vertices are connected via a single connector. To see why, notice that the right-hand side of \eqref{eq:exponential_c_n_bound} is, up to a constant, the same as the $k=2$ term of the left-hand side. In fact, Lemma \ref{lem:mu_bound} shows that the dominant class of possible paths is the one described in Section \ref{subsec:outline}.\medskip

We are now ready to bound the probability of the event $\tilde{A}^{(\x)}_n$, i.e.~the event that there exists a path of conductance $n$ where the final vertex is the first and only one which has a mark smaller than the corresponding $\ell_n$. In particular the final vertex is the most powerful vertex of the path. %As in Section \ref{subsec:outline}, b
By  Mecke's equation, we have
\[
\P_{\x}(\tilde{A}_n^{(\x)}) \leq \int\limits_{\mathbb{R}^d \times (0,\ell_n]} \mathd \y \, \E_{\x,\y}N(\x,\y,n).
\]
Hence, Fubini's theorem and Lemma~\ref{lem:mu_bound} yield
\[
\P_{\x}(\tilde{A}_n^{(\x)}) \leq \int\limits_0^{\ell_n} \mathd s \mu_n^\x(s) \leq \frac{1}{1-\gamma} \ell_n^{1-\gamma}C_n.
\]
As in Section \ref{subsec:outline}, with $\ell_0 < t\wedge \frac{1}{e}$ we choose the sequence $(\ell_n)_{n\in \mathbb{N}_0}$ for $\varepsilon>0$, such that 
\begin{equation}
\frac{1}{1-\gamma}C_n\ell_n^{1-\gamma} = \frac{\varepsilon}{\pi^2 n^2}, \label{eq:truncdef_gen}
\end{equation}
and we have
\[
\sum_{n=1}^\Delta \P_{\x}(A^{(\x)}_n) \leq \sum_{n=1}^\Delta \P_{\x}(\tilde{A}^{(\x)}_n) \leq \sum_{n=1}^\Delta \frac{1}{1-\gamma}C_n\ell_n^{1-\gamma} \leq \sum_{n=1}^\infty \frac{\varepsilon}{\pi^2 n^2} = \frac{\varepsilon}{6}.
\]
Since $C_n$ is defined recursively, we can obtain a recursive representation of the sequence $(\ell_n)_{n\in \mathbb{N}_0}$. Let $\eta_n := \ell_n^{-1}$ for $n\in \mathbb{N}_0$. Then, we have
\begin{align}
\begin{split}\eta_{n+2}^{1-\gamma}  &= \frac{\pi^2 (n+2)^2}{3\varepsilon}\frac{1}{1-\gamma}C_{n+2}\\
&= \frac{\pi^2 (n+2)^2}{3\varepsilon}\frac{1}{1-\gamma}\left[c^2 \ell_{n}^{1-\gamma-\gamma/\delta}C_n + c\log\left(\frac{1}{\ell_{n+1}}\right)C_{n+1}\right]\\
& = \frac{(n+2)^2}{n^2} c^2\eta_n^{\gamma/\delta} + \frac{(n+1)^2}{n^2} c\log(\eta_{n+1})\eta_{n+1}^{1-\gamma}. \label{eq:truncrec_gen}
\end{split}
\end{align}
Hence, there exists a different constant $c>0$ such that $\eta_{n+2}^{1-\gamma} \leq c\eta_n^{\gamma/\delta} + c\log(\eta_{n+1})\eta_{n+1}^{1-\gamma}$. By induction, we conclude that there exist $b>0$ and $B>0$ such that
\begin{equation}
\eta_n \leq b\exp\left(B\left(\frac{\gamma}{\delta(1-\gamma)}\right)^{n/2}\right),\label{eq:truncorder_gen}
\end{equation}
and thus the rate of decay of $(\ell_n)_{n\in \mathbb{N}_0}$ is faster than exponential.

\paragraph{{Probability bounds for good paths.}}
We now proceed to establish a bound on the last summand \smash{$\sum_{n=1}^{2\Delta} \P_{\x,\y}(B^{_{(\x,\y)}}_n)$} of \eqref{eq:truncmombound}. To do so we consider the original graph~$\mathscr{G}$. Recall that $B^{_{(\x,\y)}}_n$ is the event that there exists a good path of length~$n$ between $\x$ and $\y$. This can be bounded by the union of all possible good paths given by the vertices of the Poisson point process, i.e. % Hence, we have
\[
\P_{\x,\y}(B^{(\x,\y)}_n) = \P_{\x,\y}\left(\bigcup_{\substack{\x_1,\ldots,\x_{n-1}\in \mathscr{G}\\ (\x_0,\ldots,\x_{n})\ \text{good}}}^{\neq} \set{\x_0 \sim \x_1 \sim \ldots \sim \x_{n-1} \sim \x_n}\right),
\]
where $\x = \x_0$, $\y = \x_n$, $\bigcup^{\neq}$ denotes the union across all possible sets of pairwise distinct vertices $\x_0,\ldots,\x_n$ of the Poisson process. By Mecke's equation the right-hand side can be bounded from above by 
\begin{align*}
\int\limits_{\mathclap{\mathbb{R}^d\times (\ell_1,1]}} \mathd \x_1 \cdots \int\limits_{\mathclap{\mathbb{R}^d\times(\ell_{\gaus{\frac{n}{2}}},1]}} \mathd \x_{\gaus{\frac{n}{2}}}\cdots \int\limits_{\mathclap{\mathbb{R}^d\times (\ell_1,1]}} \mathd \x_{n-1} \,
 \mathbb{P}_{\mathbf{x},\x_1, \ldots, \x_{n-1}, \mathbf{y}}\set{\x \sim \x_1\sim \ldots\sim \x_{n-1} \sim \y}.
\end{align*}
The following lemma reduces this bound to a non-spatial problem for paths of ``reasonable'' length which only depends on the marks of $\x_1,\ldots,\x_{n-1}$ but not on their location. This allows us to use a similar strategy as the one used by Dereich et al.\ in~\cite{DerMM2012}, where lower bounds for the typical distance of non-spatial preferential attachment models are established.\smallskip
\begin{lemma}\label{lem:spatial_integration}
For given vertices $\x=(x,t)$ and $\y=(y,s)$, let $\Delta \leq c_\varepsilon\abs{x-y}^\epsilon$ for some $1>\varepsilon>0$ and $c_{\varepsilon}>0$. Then, there exist constants $a>0$ and $\tilde{\kappa} > 0$ such that, for $n\leq \Delta$, we have 
\begin{align*}
\int\limits_{(\mathbb{R}^d)^{n-1}}\bigotimes_{i=1}^{n-1} \mathd x_i\mathbb{P}_{\mathbf{x},\x_1, \ldots, \x_{n-1}, \mathbf{y}} & \set{\x \sim \x_1\sim \ldots\sim \x_{n-1} \sim \y}\\
& \leq \abs{x-y}^{-a} \prod_{k=1}^{n} \tilde{\kappa} (t_{k-1}\wedge t_k)^{-\gamma}(t_{k-1}\vee t_k)^{\gamma-1},
\end{align*}
where $t_0=t$ resp. $t_n=s$ are the marks of $\x$ resp. $\y$ and $\x_i = (x_i, t_i)$ for $i=1,\ldots,n-1$.
\end{lemma}
\begin{remark}\label{rem:spatial_integration}
The constants $a$ and $\tilde{\kappa}$ of Lemma \ref{lem:spatial_integration} depend on the choice of $\varepsilon$ and $c_\varepsilon$. But for $\Delta = O(\log\abs{x-y})$, for any $\epsilon>0$ there exists a $c_\epsilon>0$, such that, for $\abs{x-y}$ large enough, we have $\Delta \leq c_\varepsilon\abs{x-y}^\epsilon$. Thus, if $\abs{x-y}$ is large enough, the choice of $a$ and $\tilde{\kappa}$ does not depend on $\abs{x-y}$.
\end{remark}
\begin{proof}Let $\set{\x, \x_1, \ldots, \x_{n-1}, \y}$ be a set of given vertices. By Assumption \ref{ass:main1} we have 
\begin{align*}
\int\limits_{(\mathbb{R}^d)^{n-1}}\bigotimes_{i=1}^{n-1} \mathd x_i & \mathbb{P}_{\mathbf{x},\x_1, \ldots, \x_{n-1}, \mathbf{y}} \set{\x \sim \x_1\sim \ldots\sim \x_{n-1} \sim \y}\\
&\leq \int\limits_{(\mathbb{R}^d)^{n-1}}\bigotimes_{i=1}^{n-1} \mathd x_i \prod_{i=1}^{n}\rho\big(\kappa^{-1/\delta} (t_{i-1}\wedge t_i)^{\gamma}(t_{i-1}\vee t_i)^{1-\gamma} \abs{x_{i-1}-x_i}^{d}\big).
\end{align*}
As $n\leq c_\varepsilon\abs{x-y}^\epsilon$, no matter the choice of vertices, there must exist at least one edge between two vertices $\x_{k-1} = (x_{k-1},t_{k-1})$ and $\x_k = (x_k,t_k)$ with $\abs{x_{k-1}-x_k} \geq c_\varepsilon^{-1} \abs{x-y}^{1-\varepsilon}$. Hence, the expression above can be further bounded by
\begin{align*}
&\sum_{k=1}^{n-1}\int\limits_{(\mathbb{R}^d)^{n-1}}\bigotimes_{i=1}^{n-1} \mathd x_i \rho\big(c_\varepsilon^{-d}\kappa^{-1/\delta} (t_{k-1}\wedge t_k)^{\gamma}(t_{k-1}\vee t_k)^{1-\gamma} \abs{x-y}^{d(1-\varepsilon)}\big)\\
&\phantom{\sum_{k=1}^{n-1}\int\limits_{(\mathbb{R}^d)^{n-1}}\bigotimes_{i=1}^{n-1} \mathd \x_i}\times \prod_{\substack{i=1\\ i\neq k}}^{n}\rho\big(\kappa^{-1/\delta} (t_{i-1}\wedge t_i)^{\gamma}(t_{i-1}\vee t_i)^{1-\gamma} \abs{x_{i-1}-x_i}^{d}\big)\\
&\leq \sum_{k=1}^{n-1}\rho\big(c_\varepsilon^{-d}\kappa^{-1/\delta} (t_{k-1}\wedge t_k)^{\gamma}(t_{k-1}\vee t_k)^{1-\gamma} \abs{x-y}^{d(1-\varepsilon)}\big)\\
& \phantom{guggemolunnebbes} \times\prod_{\substack{i=1\\ i\neq k}}^{n}I_\rho (t_{i-1}\wedge t_i)^{-\gamma}(t_{i-1}\vee t_i)^{\gamma-1},
\end{align*}
where the last inequality is achieved by integration over the location of the vertices. 
We choose $\tilde{\kappa}>2c_\varepsilon^{d}\kappa^{1/\delta}\vee 2I_\rho$. Since $\delta>1$, the term
$$\rho\big(c_\varepsilon^{-d}\kappa^{-1/\delta} (t_{k-1}\wedge t_k)^{\gamma}(t_{k-1}\vee t_k)^{1-\gamma} \abs{x-y}^{d(1-\varepsilon)}\big)$$ can be bounded by $c_\varepsilon^{d}\kappa^{1/\delta} (t_{k-1}\wedge t_k)^{-\gamma}(t_{k-1}\vee t_k)^{\gamma-1} \abs{x-y}^{-d(1-\varepsilon)}$ and therefore there exists a constant $a>0$ such that we have
\begin{align*}
&\int\limits_{(\mathbb{R}^d)^{n-1}}\bigotimes_{i=1}^{n-1} \mathd x_i\mathbb{P}_{\mathbf{x},\x_1, \ldots, \x_{n-1}, \mathbf{y}}\set{\x \sim \x_1\sim \ldots\sim \x_{n-1} \sim \y}\\
& \leq \abs{x-y}^{-a} \prod_{k=1}^{n} \tilde{\kappa} (t_{k-1}\wedge t_k)^{-\gamma}(t_{k-1}\vee t_k)^{\gamma-1}.
\end{align*}
\\[-10mm]
\end{proof}

\noindent
By Remark~\ref{rem:spatial_integration}, with Lemma \ref{lem:spatial_integration} and Fubini's theorem we obtain
\[
\P_{\x,\y}(B^{(\x,\y)}_n) \leq \abs{x-y}^{-a} \int\limits_{\ell_1}^1 \mathd t_1 \cdots \!\!\!\int\limits_{\ell_{\gaus{\frac{n}{2}}}}^1 \mathd t_{\gaus{\frac{n}{2}}}\cdots \!\!\int\limits_{\ell_1}^1 \mathd t_{n-1}\prod_{k=1}^{n} \tilde{\kappa} (t_{k-1}\wedge t_k)^{-\gamma}(t_{k-1}\vee t_k)^{\gamma-1},
\]
where $\x = (x,t_0)$ and $\y = (y,t_n)$. 
We define, %for $\x = (x,t_0)$ and $n\in \mathbb{N}$,
\begin{align}
\nu_n^\x(s) =  \int\limits_{\ell_1}^1 \mathd t_1 \cdots \!\!\! & \int\limits_{\ell_{n-1}}^1 \mathd t_{n-1}  \notag\\ & \tilde{\kappa} (t_{n-1}\wedge s)^{-\gamma}(t_{n-1}\vee s)^{\gamma-1}\prod_{k=1}^{n-1} \tilde{\kappa} (t_{k-1}\wedge t_k)^{-\gamma}(t_{k-1}\vee t_k)^{\gamma-1} \label{eq:non-spatial_mu_def}
\end{align}
and set $\nu_0^\x(s) = \delta_0(t-s)$. Then, the inequality above can be rewritten as
\[
\P_{\x,\y}(B^{(\x,\y)}_n) \leq \abs{x-y}^{-a} \int\limits_{\ell_{\gaus{\frac{n}{2}}}}^1 \mathd s \nu_{\gaus{\frac{n}{2}}}^\x(s)\nu_{n-\gaus{\frac{n}{2}}}^\y(s).
\]
Note that as defined, $\nu_n^\x(s)$ can be written recursively as
\begin{equation}
\nu_n^\x(s) = \int\limits_{\ell_{n-1}}^1 \mathd u \, \nu_{n-1}^\x(u)\tilde{\kappa} (u\wedge s)^{-\gamma}(u\vee s)^{\gamma-1}. \label{eq:non-spatial_mu_rec}
\end{equation}
This allows us to establish an upper bound for $\nu_n^\x(s)$ analogous to the {non-spatial} case in \cite{DerMM2012}. 
The following lemma is a corollary of \cite[Lemma 1]{DerMM2012}.

\begin{lemma}\label{lem:non-spatial_mu_bound}
Let $(\ell_n)_{n\in \mathbb{N}}$ be a given non-increasing sequence and $\nu_n^\x(s)$ be as defined in \eqref{eq:non-spatial_mu_def}, where $\x = (x,t)$ and $s\in (0,1)$. Then, there exists a constant $c>0$ such that, for all $n\in \mathbb{N}$,  
\begin{equation}
\nu_n^\x(s) \leq \alpha_n s^{-\gamma} + \beta_n s^{\gamma-1}, \label{eq:non-spatial_mu_bound}
\end{equation}
where 
\begin{align}
\begin{split}\alpha_{n+1} &= c \big(\alpha_n \log\big(\tfrac{1}{\ell_n}\big) + \beta_n\big)\\
\beta_{n+1} &= c \big(\alpha_n \ell^{1-2\gamma} + \beta_n\log\big(\tfrac{1}{\ell_n}\big)\big) \label{eq:alp_bet_def}
\end{split}
\end{align}
and $\alpha_1 = \tilde{\kappa} t^{\gamma-1}$, $\beta_1 = \tilde{\kappa} t^{-\gamma}$.
\end{lemma}
\begin{proof}
For $n=1$, we have by \eqref{eq:non-spatial_mu_def} that
\[
\nu_1^\x(s) = \tilde{\kappa} (t \wedge s)^{-\gamma}(t \vee s)^{\gamma-1} = \mathbbm{1}_{\set{s\leq t}} \tilde{\kappa} s^{-\gamma}t^{\gamma-1} + \mathbbm{1}_{\set{s> t}} \tilde{\kappa} t^{-\gamma}s^{\gamma-1} \leq \alpha_1 s^{-\gamma} + \beta_1 s^{\gamma-1}.
\]
Assume \eqref{eq:non-spatial_mu_bound} holds for $n\in \mathbb{N}$. Then, by \eqref{eq:non-spatial_mu_rec}, we have that
\begin{align*}
\nu_{n+1}^\x(s) &= \int\limits_{\ell_{n}}^1 \mathd u \nu_{n}^\x(u)\tilde{\kappa} (u\wedge s)^{-\gamma}(u\vee s)^{\gamma-1}
\leq  \int\limits_{\ell_{n}}^1 \mathd u \nu_{n}^\x(u)\tilde{\kappa} \left(s^{-\gamma}u^{\gamma-1} + u^{-\gamma}s^{\gamma-1}\right)\\
\phantom{\nu_{n+1}^\x(s)}&\leq  \tilde{\kappa} s^{-\gamma}\int\limits_{\ell_{n}}^1 \mathd u \left(\alpha_n u^{-1} + \beta_n u^{2\gamma-2}\right) + \tilde{\kappa} s^{\gamma-1}\int\limits_{\ell_{n}}^1 \mathd u \left(\alpha_n u^{-2\gamma} + \beta_n u^{-1}\right)\\
&\leq  \tilde{\kappa} s^{-\gamma} \left(\alpha_n \log\left(\tfrac{1}{\ell_n}\right) + \beta_n \frac{1}{2\gamma-1}\right) + \tilde{\kappa} s^{\gamma-1}\left(\alpha_n \frac{1}{2\gamma-1}\ell_n^{1-2\gamma} + \beta_n\log\left(\tfrac{1}{\ell_n}\right)\right)\\
&\leq \alpha_{n+1}s^{-\gamma} + \beta_{n+1}s^{\gamma-1}.
\end{align*}
Hence, by induction \eqref{eq:non-spatial_mu_bound} holds for all $n\in \mathbb{N}$.
\end{proof}
Although Lemma \ref{lem:non-spatial_mu_bound} holds for an arbitrary sequence $(\ell_n)_{n\in \mathbb{N}}$, recall that we have chosen $(\ell_n)_{n\in \mathbb{N}}$ such that \eqref{eq:truncdef_gen} holds. This implies by \eqref{eq:truncrec_gen} that there exists a constant $c_1>0$ such that 
\begin{equation}
\eta_{n+2} \geq c_1 \eta_n^{\gamma/\delta(1-\gamma)}\ \text{for all}\ n\in \mathbb{N}_0, \label{eq:eta_recbound}
\end{equation}
where $\eta_n := \ell_n^{-1}$ as before. Additionally, notice that $(\alpha_n)_{n\in \mathbb{N}}$ and $(\beta_n)_{n\in \mathbb{N}}$ are non-decreasing sequences. By Lemma \ref{lem:non-spatial_mu_bound}, we have that
\begin{align*}
\sum_{n=1}^{2\Delta} \P_{\x,\y}(B^{(\x,\y)}_n) &\leq \abs{x-y}^{-a}\sum_{n=1}^{2\Delta} \int\limits_{\ell_{\gaus{\frac{n}{2}}}}^1 \mathd s \nu_{{\gaus{\frac{n}{2}}}}^\x(s)\nu_{{n-\gaus{\frac{n}{2}}}}^\y(s)\\
&\leq 2 \abs{x-y}^{-a}\sum_{n=1}^{\Delta} \int\limits_{\ell_{n}}^1 \mathd s \left(\alpha_n s^{-\gamma} + \beta_n s^{\gamma-1}\right)^2\\
&\leq \frac{4}{2\gamma-1} \abs{x-y}^{-a}\sum_{n=1}^{\Delta} \alpha_n^2 \ell_n^{1-2\gamma} + \beta_n^2.
\end{align*}
It follows from the definition of $(\alpha_n)_{n\in \mathbb{N}}$ and $(\beta_n)_{n\in \mathbb{N}}$ that $\beta_n \leq c^{-1}\alpha_{n+1}$ and $$\beta_n \leq c\big(\alpha_n \ell_n^{1-2\gamma} + \alpha_n\log\big(\tfrac{1}{\ell_n}\big)\big),$$ where the second summand on the right-hand side is bounded by a multiple of the first. Therefore, there exists a constant $c_2>0$ such that $\beta_n^2 \leq c_2 \alpha_{n+1}^2 \ell_{n+1}^{1-2\gamma}$. This and the monotonicity of the sequences $(\alpha_n)_{n\in \mathbb{N}}$ and $(\ell_n)_{n\in \mathbb{N}}$ gives that
\begin{equation}
\sum_{n=1}^{2\Delta} \P_{\x,\y}(B^{(\x,\y)}_n) \leq \frac{4(1+c_2)}{2\gamma-1} \abs{x-y}^{-a}\sum_{n=1}^{\Delta} \alpha_{n+1}^2 \ell_{n+1}^{1-2\gamma}. \label{eq:final_good_path_bound}
\end{equation}
Recall that the sequence $(C_n)_{n\in \mathbb{N}}$ from Lemma \ref{lem:mu_bound} is defined as 
\[
C_{n+2} = c^2 \ell_{n}^{1-\gamma-\gamma/\delta}C_n + c \log\left(\tfrac{1}{\ell_{n+1}}\right)C_{n+1}
\]
with $C_1 = c\ell_0^{\gamma-1}$ and $C_2 = c^2\ell_0^{-\gamma/\delta} + c\log(\tfrac{1}{\ell_1})C_1$. We  compare this sequence to $(\alpha_n)_{n\in \mathbb{N}}$ in order to bound~\eqref{eq:final_good_path_bound} further. By writing $\alpha_{n+2}$ in terms of $\alpha_n$ and $\beta_n$ we have that
\[
\alpha_{n+2} = c^2\left(\alpha_n\big(\ell_n^{1-2\gamma} + \log(\tfrac{1}{\ell_{n+1}})\log(\tfrac{1}{\ell_n})\big) + \beta_n\big(\log(\tfrac{1}{\ell_{n+1}}) + \log(\tfrac{1}{\ell_n})\big)\right).
\]
As all summands on the right-hand side are bounded by a multiple of $\alpha_n \ell_n^{1-2\gamma}\log({1/\ell_{n+1}})$ and $\log({1/\ell_{n+1}})$ is smaller than a multiple of $\log(1/\ell_{n})$, there exists a constant $c_3$ such that $\alpha_{n+2} \leq c_3 \alpha_n \ell_n^{1-2\gamma}\log({1/\ell_n})$. To compare $(\alpha_n)_{n\in \mathbb{N}}$ and $(C_n)_{n\in \mathbb{N}}$, {notice that, up to a constant, $\alpha_1$} and $\alpha_2$ are equal to $C_1$ and $C_2$. Moreover
\[
\frac{\alpha_{n+2}}{C_{n+2}} \leq \frac{c_3\ell_n^{1-2\gamma}\log\left(\tfrac{1}{\ell_n}\right) \alpha_n }{c^2 \ell_n^{1-\gamma-\gamma/\delta} C_n} = \frac{c_3}{c^2} \ell_n^{\gamma/\delta - \gamma} \log\left(\tfrac{1}{\ell_n}\right) \frac{\alpha_n}{C_n}.
\]
Applying this inequality recursively and expressing $\alpha_{n+2}$ we obtain that for some $c_4>0$ 
\[
\alpha_n \leq C_n\prod_{i=1}^{\suag{\frac{n}{2}}-1}c_4 \ell_{n-2i}^{\gamma/\delta-\gamma}\log\left(\tfrac{1}{\ell_{n-2i}}\right)\quad \text{for all}\ n\in \mathbb{N}.
\]
Hence, we have
\begin{align*}
&\alpha_{n+1}^2 \ell_{n+1}^{1-2\gamma} \leq C_{n+1}^2\ell_{n+1}^{1-2\gamma} \prod_{i=1}^{\suag{\frac{n+1}{2}}-1}c_4^2 \ell_{n+1-2i}^{2(\gamma/\delta-\gamma)}\log\left(\tfrac{1}{\ell_{n+1-2i}}\right)^2\\
&\leq \left(\tfrac{3\varepsilon(1-\gamma)}{\pi^2 (n+1)^2}\right)^2 \ell_{n+1}^{-1} \ell_{n+1}^{2(\gamma/\delta-\gamma)\sum_{i=1}^{\suag{\frac{n+1}{2}}-1} \left(\delta(1-\gamma)/\gamma\right)^i}\prod_{i=1}^{\suag{\frac{n+1}{2}}-1}\!\!\!\!\!\!c_1^{\sum_{j=1}^i(\delta(1-\gamma)/\gamma)^i} c_4^2 \log\left(\tfrac{1}{\ell_{n+1-2i}}\right)^2,
\end{align*}
where the second inequality follows by \eqref{eq:truncdef_gen} and \eqref{eq:eta_recbound}. Observe that, as \smash{$\frac{\delta(1-\gamma)}{\gamma}<1$}, the series \smash{$\sum_{i=1}^{\infty} \left(\delta(1-\gamma)/\gamma\right)^i$} converges. Hence, there exists a constant which is larger than~$c_1$ to the power 
\smash{${\sum_{j=1}^i(\delta(1-\gamma)/\gamma)^i}$} for any $i\in \mathbb{N}$ and 
 a constant $c_5>0$ such that
\[
\ell_{n+1}^{2(\gamma/\delta-\gamma)\sum_{i=1}^{\suag{\frac{n+1}{2}}-1} \left(\delta(1-\gamma)/\gamma\right)^i} \leq \ell_{n+1}^{-c_5}.
\]
Furthermore since we have established that $\eta_n$ is of the order displayed in \eqref{eq:truncorder_gen} it follows directly that the left-hand side multiplied with the product above can also be bounded by $\ell_{n+1}^{-c_5}$ for any sufficiently large constant $c_5>0$. Hence, there exists a further constant $c_6 >0$ such that \smash{$\frac{4(1+c_2)}{2\gamma-1}\alpha_{n+1}^2 \ell_{n+1}^{1-2\gamma} \leq \frac{c_6}{(n+1)^4}\ell_{n+1}^{-(1+c_5)}$}. Therefore, we have by using \eqref{eq:truncorder_gen} once more that
\begin{align*}
\sum_{n=1}^{2\Delta} \P_{\x,\y}(B^{(\x,\y)}_n) &\leq \frac{c_6}{(\Delta+1)^3}\abs{x-y}^{-a}\ell_{\Delta+1}^{-(1+c_5)}\\
&\leq \frac{c_6b}{(\Delta+1)^3}\abs{x-y}^{-a}\exp\left(B(1+c_5)\left(\frac{\gamma}{\delta(1-\gamma)}\right)^{\frac{\Delta+1}{2}}\right).
\end{align*}
Let $D>0$ such that $B(1+c_5)(\gamma/(\delta(1-\gamma)))^{\frac{1-D}{2}}<a$ and choose $\Delta \leq \frac{2\log\log\abs{x-y}}{\log\left(\gamma/\delta(1-\gamma))\right)} - D$. Then the above expression is of order \smash{$\mathcal{O}(\log\log\abs{x-y}^{-2})$}. Hence, for our choice of $\Delta$, we have
\[
\P_{\x,\y}\set{\mathd (\x, \y) \leq 2\Delta} \leq \varepsilon + \mathcal{O}\left(\log\log\abs{x-y}^{-2}\right),
\]
which implies the stated lower bound of Theorem \ref{mainthm}(b).

\subsection{The non-ultrasmall regime}
In this section we consider the case $\gamma<\frac{\delta}{\delta+1}$ and show that the graph is not ultrasmall, i.e. the chemical distance in the graph is not of double logarithmic order of the Euclidean distance. In particular, we show the following.

\begin{prop}\label{prop:smlthm}
Let $\mathscr{G}$ be a geometric random graph which satisfies Assumption \ref{ass:main1} {for {some} $\delta>1$ and  \smash{$0<\gamma<\frac{\delta}{\delta + 1}$}. Then, for any $p>1$, there exists $c>0$ such that, for $\x,\y \in \mathbb{R}^d\times (0,1)$,  we have 
\[
\mathd(\mathbf{x},\mathbf{y}) \geq c\frac{\log \abs{x-y}}{\left(\log\log\abs{x-y}\right)^p}
\]
under {$\P_{\x,\y}$} with high probability as $\abs{x-y} \to \infty$.}\end{prop}

The proof is structurally analogous to the ultrasmall case, but significantly easier due to the simpler nature of the dominating strategy.
As in Section~\ref{subsec:ultsml_phase}, we bound the probabilities in \eqref{eq:truncmombound} using a suitable truncation sequence $(\ell_n)_{n\in \mathbb{N}_0}$ such that the probability that bad paths starting in a vertex $\x$ exist can be made arbitrarily small. {
In this case, however, the truncation sequence decreases only exponentially. {Similarly to the ultrasmall case}, we construct a graph~$\tilde{\mathscr{G}}$ which contains a copy of $\mathscr{G}$ and additionally an edge is added between two vertices $\x=(x,t)$ and $\y = (y,s)$ of $\tilde{\mathscr{G}}$  whenever
\[
\abs{x-y}^d\leq \kappa^{1/\delta}(t\wedge s)^{-\gamma}(t\vee s)^{\gamma - 1}.
\]
{Unlike done previously in Section \ref{subsec:ultsml_phase}, we assign no conductance to any edge in $\tilde{\mathcal{G}}$ and therefore only consider the lengths of paths.}
We declare a self-avoiding path $P = (\x_0,\ldots,\x_n)$ in $\tilde{\mathscr{G}}$
\emph{step minimizing} if there exists no edge between $\x_i$ and $\x_j$ for all $i,j$ with $\abs{i-j}\geq 2$ and denote by \smash{$\tilde{A}^\x_n$} the event that there exists a step minimizing path starting in $\x$ of length $n$ in $\tilde{\mathscr{G}}$, where the final vertex is the first vertex which has a mark smaller than the corresponding $\ell_n$.} Then the first two summands of the right-hand side of \eqref{eq:truncmombound} can be bounded from above by \smash{$\sum_{n=1}^\Delta \P_{\x}(\tilde{A}^{(\x)}_n)$ and $\sum_{n=1}^\Delta \P_{\y}(\tilde{A}^{(\y)}_n)$}, since for any path implying the event \smash{$A^{(\x)}_n$} there exists a step minimizing path in $\tilde{\mathscr{G}}$ of smaller or equal length which also fails to be good on its last vertex.\pagebreak[3]
\bigskip 

To bound these probabilities, we define the random variable $N(\x,\y,n)$ as the number of distinct step minimizing paths between $\x$ and $\y$ of length $n$, whose vertices $(x_1,t_1),\ldots (x_{n-1},t_{n-1})$ fulfill $t\geq \ell_0, t_1\geq \ell_1,\ldots, t_{n-1} \geq \ell_{n-1}$ and which all have a larger mark than $\y$. By Mecke's equation we have that 
$$\P_{\y}(\tilde{A}^{(\y)}_n) \leq \int_{\mathbb{R}^d\times (0,\ell_n]} dy \, \E_{\x,\y}N(\x,\y,n), \qquad \mbox{ for $n\in \mathbb{N}$.}$$ 
As before, the paths counted in $N(\x,\y,n)$ 
can be decomposed such that \eqref{eq:recdecomp_rv_N} holds, where $K(\x,\y,k)$~is the number of step minimizing paths between $\x$ and $\y$ of length $k$ such that the vertices $\x_1,\ldots,\x_{{k}-1}$ between them have marks larger than $\x$ and $\y$. We again refer to such vertices as connectors. Note that if \smash{$\abs{x-y}^d\leq \kappa^{1/\delta}(t\wedge s)^{-\gamma}(t\vee s)^{\gamma - 1}$}, there exists no step minimizing paths of length larger or equal two between  $\x$ and $\y$. Hence, we have $N(\x,\y,n) = K(\x,\y,n) = 0$ for $n\geq 2$ {under this assumption}.
\bigskip 

We now bound the expectation of $K(\x,\y,k)$.  
As in Section \ref{subsec:ultsml_phase}, we define a mapping
\[
{e_K:(\mathbb{R}^d\times (0,1])^2\times \mathbb{N} \to [0,\infty),}
\]
where
$e_K(\x,\y,1) = \rho(\kappa^{-1/\delta}(t\wedge s)^{\gamma}(t\vee s)^{1-\gamma}\abs{x-y}^d),$ for $\x,\y \in \mathbb{R}^d\times (0,1]$ and\begin{equation*}
e_K(\x,\y,k) = \sum_{i=1}^{k-1}\,\, \int\limits_{\mathbb{R}^d\times (t\vee s,1]} \mathd \z \,  e_K(\x,\z,i)e_K(\z,\y,k-i), \quad \text{for}\ k\ge2, \x,\y \in \mathbb{R}^d\times (0,1], \label{eq:e_def_rec_non}
\end{equation*}
if $\smash{\abs{x-y}^d > \kappa^{1/\delta}(t\wedge s)^{-\gamma}(t\vee s)^{\gamma - 1}}$ and otherwise $e_K(\x,\y,k) = 0$. 
As before we use a binary tree to classify the connection strategies and use this together with Assumption~\ref{ass:main1} to obtain
$\E_{\x,\y} K(\x,\y,k) \leq e_K(\x,\y,k)$, for $k\in \mathbb{N}$. 
\medskip

\begin{lemma}\label{lem:kconn_non}
Let $\x = (x,t), \y = (y,s)$ be vertices with 
$\smash{\abs{x-y}^d > \kappa^{1/\delta}(t\wedge s)^{-\gamma}(t\vee s)^{\gamma - 1}}$. Then there exists $C>1$ such that, for $k\geq 2$, we have
\[
e_K(\x,\y,k) \leq C^{k-1} \kappa(t\wedge s)^{-\gamma \delta}(t\vee s)^{(\gamma-1)\delta}\abs{x-y}^{-d\delta}.
\]
\end{lemma}
\begin{proof}
By \cite[Lemma 2.2]{GraLM2021} there exists a constant $C>1$ such that if $\abs{x-y}^d > \kappa^{1/\delta}(t\wedge s)^{-\gamma}(t\vee s)^{\gamma - 1}$ we have
\begin{align}
\begin{split}
e_K(\x,\y,2) \leq& \int\limits_{\mathbb{R}^d\times (t\vee s,1]} \mathd \z \rho(\kappa^{-1/\delta}t^{\gamma}u^{1-\gamma}\abs{x-z}^d)\rho(\kappa^{-1/\delta}s^{\gamma} u^{1-\gamma}\abs{y-z}^d)\\
 \leq& C \kappa(t\wedge s)^{-\gamma \delta}(t\vee s)^{(\gamma-1)\delta}\abs{x-y}^{-d\delta}.\label{eq:twoconn_non}
\end{split}
\end{align}
We now show by induction that
\begin{equation}
e_K(\x,\y,k) \leq \Cat(k-1)C^{k-1} \kappa(t\wedge s)^{-\gamma \delta}(t\vee s)^{(\gamma-1)\delta}\abs{x-y}^{-d\delta}\label{eq:kconn_bound_non}
\end{equation}
holds for all $k\geq 2$. This is sufficient, since $\Cat(k)\leq 4^{k}$. For $k=2$ this follows from~\eqref{eq:twoconn_non}. Let $k\geq 3$ and assume \eqref{eq:kconn_bound_non} holds for all $j= 2,\ldots, k-1$. For $\smash{\abs{x-y}^d > \kappa^{1/\delta}(t\wedge s)^{-\gamma}(t\vee s)^{\gamma - 1}}$ this, together with the definition of $e_K(\x,\y,k)$, yields
\begin{align*}
e_K(\x,\y,k) \leq \sum_{i=1}^{k-1}&\Cat(i-1)\Cat(k-i-1)C^{k-2} \times\\&\int\limits_{\mathbb{R}^d\times (t\vee s,1]} \mathd \z \,\rho(\kappa^{-1/\delta}t^{\gamma}u^{1-\gamma}\abs{x-z}^d)\rho(\kappa^{-1/\delta}s^{\gamma} u^{1-\gamma}\abs{y-z}^d).
\end{align*}
With \eqref{eq:twoconn_non} the right-hand side can be further bounded by
\begin{align*}
\sum_{i=1}^{k-1} \Cat(i-1)\Cat(k-i-1)C^{k-1}\kappa(t\wedge s)^{-\gamma \delta}(t\vee s)^{(\gamma-1)\delta}\abs{x-y}^{-d\delta}.
\end{align*}
As $\sum_{i=1}^{k-1} \Cat(i-1)\Cat(k-i-1) = \Cat(k-1)$ we get that \eqref{eq:kconn_bound_non} holds for $k$.
\end{proof}
\paragraph{Probability bounds for bad paths}
Using Lemma \ref{lem:kconn_non} and \eqref{eq:recdecomp_rv_N} we find a suitable upper bound for $\int_{\mathbb{R}^d}\E_{\x,\y}N(\x,\y,n)  \mathd y$, with $\y = (y,s)$, which leads to a bound for $\P_{\x}(\tilde{A}_n^{\x})$. Recall that by \eqref{eq:recdecomp_rv_N} we have, for $n\in \mathbb{N}$,
\[
N(\x,\y,n) \leq K(\x,\y,n) + \sum_{k=1}^{n-1} \sum_{\substack{\z = (z,u)\\ t>u> \ell_{n-k}\vee s}} N(\x,\z,n-k)K(\z,\y,k).
\]
As in Section \ref{subsec:ultsml_phase}, to establish an upper bound on $\E_{\x,\y}N(\x,\y,n)$, we define a mapping
\[
{e_N\colon (\mathbb{R}^d\times (0,1])^2 \times \mathbb{N} \to [0,\infty),}
\]
by setting
$e_N(\x,\y,1) = \rho(\kappa^{-1/\delta}(t\wedge s)^{\gamma}(t\vee s)^{1-\gamma}\abs{x-y}^d)$, for $\x,\y \in \mathbb{R}^d\times (0,1]$,
and for $n\geq 2$ if $\smash{\abs{x-y}^d > \kappa^{1/\delta}(t\wedge s)^{-\gamma}(t\vee s)^{\gamma - 1}}$ we set $e_N(\x,\y,n)$ to be
\begin{equation}
e_K(\x,\y,n) + \sum_{i=1}^{n-1}\int\limits_{\mathbb{R}^d\times (\ell_{n-k}\vee s,1]} \mathd \z e_N(\x,\z,n-k)e_K(\z,\y,k), \quad \text{for}\ \x,\y \in \mathbb{R}^d\times (0,1],
\end{equation}
 and otherwise $e_N(\x,\y,n) = 0$. As in Section~\ref{subsec:ultsml_phase} we have $\E_{\x,\y} N(\x,\y,n) \leq e_N(\x,\y,n)$, for $n\in \mathbb{N}$.
Thus, for a given vertex $\x=(x,t)$ and $n\in \mathbb{N}$, an upper bound of $\int_{\mathbb{R}^d} \mathd y \E_{\x,\y}N(\x,\y,n)$ is given by the mapping $\mu_n^\x:(0,t] \to [0,\infty)$ defined by
\begin{equation}
\mu_n^\x(s) := \int_{\mathbb{R}^d}\mathd y \, e_N(\x,\y,n),\quad \text{for}\ s\in (0,t],
\end{equation}
where $\y = (y,s)$. We interpret $s$ as the mark of the last vertex of a path counted by the random variable $N(\x,\y,n)$.
With \smash{$I_\rho = \int\mathd x \rho(\kappa^{-1/\delta}\abs{x}^d)$} we can see by the definition of $e_N(\x,\y,1)$ that
$\mu_1^\x(s) \leq I_\rho s^{-\gamma}t^{\gamma-1}$ for $s\in (0,t]$
and for $n\geq 2$ it follows by  a short calculation and Lemma \ref{lem:kconn_non} that
\[
\mu_n^\x(s) \leq I_\rho C^{n-1}s^{-\gamma}t^{\gamma-1}+\sum_{k=1}^{n-1}I_\rho C^{k-1}s^{-\gamma} \int\limits_{\ell_{n-k}}^t\mathd u \mu_{n-k}^\x(u) u^{\gamma-1}\quad \text{for}\ s\in (0,t],
\]
where $C>1$ is the constant from Lemma \ref{lem:kconn_non}.
To establish a bound for $\mu_n^\x$ no further assumptions on the truncation sequence $(\ell_n)_{n\in\mathbb{N}_0}$ are necessary. As discussed in Section \ref{subsec:outline} we will see that the major contribution to the mass of $\mu_n^\x(s)$ comes from the paths where the two most powerful vertices are connected directly and not via one or more connectors. This is indicated by the definition of the sequence $(C_n)_{n\in \mathbb{N}_0}$ and the inequality \eqref{eq:exponential_cn_bound_non} in the proof of the following lemma.
\begin{lemma} \label{lem:mu_bound_non}
Let $\x = (x,t)$ be a given vertex and let the sequence $(\ell_n)_{n\in \mathbb{N}_0}$ be monotonically decreasing with $\ell_0 <t\wedge \frac{1}{e}$. Then, there exists $c>0$ such that, for $n\in \mathbb{N}$, \begin{equation}
\mu_n^\x(s) \leq C_n s^{-\gamma}\quad \text{for}\ s\in (0,t], \label{eq:mu_bound_non}
\end{equation}
where $C_1 = c\ell_0^{\gamma-1}$ and $C_{n+1} = c\log(\frac{1}{\ell_n})C_n$.
\end{lemma}

\begin{proof}
We choose the constant $c>2(C\vee I_\rho)$, where $C$ is as in
Lemma~\ref{lem:kconn_non}. Then by definition of $\mu_1^\x$ we have 
\smash{$\mu_1^\x(s) = I_\rho s^{-\gamma}t^{\gamma-1} \leq c s^{-\gamma}\ell_0^{\gamma-1} = C_1s^{-\gamma}$} for $s\in (0,t]$.
Let $n\geq 2$ and assume that \eqref{eq:mu_bound_non} holds for all $\tilde{n} \leq n-1$. Then, by \eqref{eq:mu_bound_non},
\begin{align*}
\mu_n^\x(s) &= I_\rho C^{n-1} s^{-\gamma}t^{\gamma-1} + \sum_{k=1}^{n-1} I_\rho C^{k-1} s^{-\gamma} \int\limits_{\ell_{n-k}}^t \mathd u \mu_{n-k}^\x (u) u^{\gamma-1}\\
&= I_\rho s^{-\gamma} \Big(C^{n-1}\ell_0^{\gamma-1} + \sum_{k=1}^{n-1} C_{n-k} C^{k-1}\log\big(\tfrac{1}{\ell_{n-k}}\big)\Big).
\end{align*}
We now want to show that 
\begin{equation}
C^{n-1}\ell_0^{\gamma-1} + \sum_{k=1}^{n-1} C_{n-k} C^{k-1}\log\left(\tfrac{1}{\ell_{n-k}}\right) \leq 2\log\left(\tfrac{1}{\ell_{n-1}}\right) C_{n-1}, \label{eq:exponential_cn_bound_non}
\end{equation}
since assuming this leads to
\smash{$\mu_n^\x(s) \leq 2I_\rho \log(\tfrac{1}{\ell_{n-1}}) C_{n-1}s^{-\gamma} \leq c\log(\tfrac{1}{\ell_{n-1}}) C_{n-1}s^{-\gamma}$}
\smash{$ = C_n s^{-\gamma}$},
which completes the proof. By definition of the constants $C_n$ we have that
\[
C^{n-1}\ell_0^{\gamma-1} + \sum_{k=1}^{n-1} C_{n-k} C^{k-1}\log(\tfrac{1}{\ell_{n-k}}) \leq \log(\tfrac{1}{\ell_{n-1}})C_{n-1} + \frac{1}{2}\sum_{k=2}^n C^{k-2}C_{n-k+1}.
\]
As $\log(\frac{1}{\ell_n})>1$, for all $n\in \mathbb{N}_0$, we have $C_{n+1}\geq cC_n$ by definition of $(C_n)_{n\in \mathbb{N}_0}$, and using that $c>2C$, the right-hand side can be further bounded by
\begin{align*}
\log\big(\tfrac{1}{\ell_{n-1}}\big)C_{n-1} + \sum_{k=2}^n \big(\tfrac{1}{2}\big)^{k-1} c^{k-2} C_{n-k+1}
& \leq 2\log\big(\tfrac{1}{\ell_{n-1}}\big) C_{n-1},
\end{align*}
which shows \eqref{eq:exponential_cn_bound_non}.
\end{proof}
Now we bound the probability of the event $\tilde{A}_n^{(\x)}$, i.e. the event that there exists a path of length $n$, where the last vertex is the only vertex which has a mark smaller than its truncation bound $\ell_n$. As in Section~\ref{subsec:ultsml_phase}, Mecke's equation yields
\[
\P(\tilde{A}_n^{(\x)}) \leq \int\limits_{\mathbb{R}^d \times (0,\ell_n]} \mathd \y \E_{\x,\y}N(\x,\y,n) \leq \int\limits_0^{\ell_n} \mathd s \mu_n^{\x}(s) \leq \frac{1}{1-\gamma}\ell_n^{1-\gamma} C_n,
\]
where we have used Fubini's theorem in the second inequality and Lemma \ref{lem:mu_bound_non} in the third one. With $\ell_0 < t\wedge \frac{1}{e}$ we choose the sequence $(\ell_n)_{n\in \mathbb{N}}$ for $\epsilon>0$, such that 
\[
\frac{1}{1-\gamma} C_n \ell_n^{1-\gamma} = \frac{\varepsilon}{\pi^2n^2},
\]
and get $P(\tilde{A}_n^{(\x)})\leq \frac{\varepsilon}{6}$. From the recursive definition of  the sequence
$(C_n)$ we obtain a recursive representation of $(\ell_n)_{n\in \mathbb{N}_0}$. Let $\eta_n := \ell_n^{-1}$ for $n\in \mathbb{N}_0$, then
\begin{align*}
\eta_{n+1}^{1-\gamma} &= \tfrac{\pi^2(n+1)^2}{3\varepsilon}\tfrac{1}{1-\gamma}C_{n+1}
&= \tfrac{\pi^2(n+1)^2}{3\varepsilon}\tfrac{1}{1-\gamma}\left(c\log(\eta_{n})C_n\right) = \tfrac{(n+1)^2}{n^2}c\log(\eta_{n+1})\eta_{n+1}^{1-\gamma}.
\end{align*}
Hence, there exists a new constant $c>0$ such that $\eta_{n+1}^{1-\gamma} \leq c\log(\eta_{n+1})\eta_{n+1}^{1-\gamma}$ and by induction we get that for any $p>1$ there exists $B>1$ large enough such that
\begin{equation}
\eta_n \leq B^{n\log^p(n+1)}. \label{eq:eta_bound_non}
\end{equation}

\paragraph{Probability bounds for good paths}
We now consider the existence of good paths between two given vertices $\x$ and $\y$.
We focus on the case 
\smash{$\gamma \in (\frac{1}{2}, \frac{\delta}{\delta+1})$}, as the cases $\gamma = \frac{1}{2}$ and $\gamma<\frac{1}{2}$ follow with analogous or simpler arguments. {As before we restrict the event $B_n^{\x,\y}$ to the existence of a step minimizing good path of length~$n$ connecting $\x$ and $\y$ in $\tilde{\mathcal{G}}$. Deviating a bit from the method of Section~\ref{subsec:ultsml_phase} we relax the definition of $B_n^{\x,\y}$ {by} defining $\tilde{B}_n^{\x,\y}$ as the event that there exists a step minimizing path between $\x$ and $\y$ in $\tilde{\mathscr{G}}$ where the most powerful vertex of the path has a mark larger than~\smash{$\ell_{\gaus{\frac{n}{2}}}$.} Then the term $\sum_{n=1}^{2\Delta} \P_{\x,\y}(B_n^{\x,\y})$ in \eqref{eq:truncmombound} can be replaced by 
\smash{$\sum_{n=1}^{2\Delta} \P_{\x,\y}(\tilde{B}_n^{\x,\y})$}.}
\bigskip 

We characterize the paths used in $\tilde{B}_n^{\x,\y}$ by their powerful vertices, {as done for regular paths} in \cite{GraLM2021}. {A vertex $\x_k$ of a path $(\x_0,\ldots,\x_n)$ is \emph{powerful} if $t_i\geq t_k$ for all $i=0,\ldots, k-1$ or if $t_i\geq t_k$ for all $i=k+1,\ldots, n$. Note that by definition the vertices $\x = \x_0$ and $\y = \x_n$ are always powerful. The indices of the powerful vertices are a subset of $\set{0,\ldots,n}$ which we denote by $\set{i_0,i_1,\ldots,i_{m-1},i_m}$, where $m+1$ is the number of powerful vertices in a path and $i_0=0$, $i_m=n$. As the most powerful vertex of a good path fulfils the assumption above, there exists a $k\in\{0,\ldots,m\}$ such that $\x_{i_k}$ is the most powerful vertex of the path. We decompose the good paths at the powerful vertices first and then proceed to decompose the path between powerful vertices $\x_{i_j}$ and $\x_{i_{j-1}}$ in the same way as done for the random variable $K(\x_{i_{j-1}},\x_{i_j}, i_{j}-i_{j-1})$ in Section \ref{subsec:ultsml_phase}.} Using Mecke's equation, we get
\begin{align*}
&\P_{\x,\y}(\tilde{B}_n^{(\x,\y)})\\
&\leq \sum_{m=1}^n \sum_{k=0}^m \int\limits_{\substack{(\mathbb{R}^d\times (0,1])^{m-1}\\ t_0>\ldots>t_k> \ell_{\gaus{\frac{n}{2}}}\\t_k<\ldots<t_m}} \bigotimes\limits_{j=1}^{m-1}\mathd \x_j \sum_{\substack{\set{i_1,\ldots,i_{m-1}}\\ \subset \set{1,\ldots,n-1}}} \prod\limits_{j=1}^m e_K(\x_{i_{j-1}},\x_{i_j},i_j-i_{j-1})\\
&\leq \sum_{m=1}^n \binom{n-1}{m-1}C^{n-m}\sum_{k=0}^m \int\limits_{\substack{(\mathbb{R}^d\times (0,1])^{m-1}\\ t_0>\ldots>t_k> \ell_{\gaus{\frac{n}{2}}}\\t_k<\ldots<t_m}} \bigotimes\limits_{j=1}^{m-1}\mathd \x_j \\
& \phantom{wiggledeewoggledee}\prod\limits_{j=1}^m \rho\left(\kappa^{-1/\delta}(t_{j-1}\wedge t_j)^\gamma (t_{j-1}\vee t_j)^{1-\gamma} \abs{x_j-x_{j-1}}^d \right).
\end{align*}
Then, following the same arguments as in the proof of Lemma \ref{lem:spatial_integration}, there exists $a>0$ and $\tilde{\kappa}>0$ such that \smash{$\P_{\x,\y}(\tilde{B}_n^{(\x,\y)})$} is bounded by
\begin{align*}
\abs{x-y}^{-a}\sum_{m=1}^n \binom{n-1}{m-1}C^{n-m}\tilde{\kappa}^{m-1}\sum_{k=0}^m \!\!\!\int\limits_{\substack{(0,1]^{m-1}\\ t_0>\ldots>t_k> \ell_{\gaus{\frac{n}{2}}}\\t_k<\ldots<t_m}} \!\!\!\!\!\!\bigotimes\limits_{j=1}^{m-1}\mathd t_j \prod\limits_{j=1}^m (t_{j-1}\wedge t_j)^{-\gamma} (t_{j-1}\vee t_j)^{\gamma-1}.
\end{align*}
By a simple calculation\footnote{For details see the proof of \cite[Lemma 2.5]{GraLM2021}, which differs from this calculation only in the fact that the mark of the first and last vertex of a path is fixed in our setting.} the sum over $k$ on the right-hand side can be bounded by {a constant multiple of}
\[
\sum_{k=1}^{m-1} \binom{m-2}{k-1} \frac{\ell_{\gaus{\frac{n}{2}}}^{1-2\gamma} \log^{m-2}\big({\ell^{-1}_{\gaus{\frac{n}{2}}}}\big)}{(m-2)!} + \frac{2\ell_0^{-1}\log^{m-1}\big({\ell^{-1}_0}\big)}{(m-1)!}.
\]
Since $\sum_{k=1}^{m-1} \binom{m-2}{k-1} \leq 2^{m-2}$ and the second summand can be bounded by a multiple of the first, there exists a constant $c_1>0$ such that \smash{$\P_{\x,\y}(\tilde{B}_n^{(\x,\y)})$} is bounded by 
\begin{align*}
&c_1\abs{x-y}^{-a}\sum_{m=1}^n \binom{n-1}{m-1}C^{n-m}\tilde{\kappa}^{m-1}\frac{\ell_{\gaus{\frac{n}{2}}}^{1-2\gamma} \log^{m-2}(\ell_{\gaus{\frac{n}{2}}}^{-1})2^{m-2}}{(m-2)!}\\
&\leq c_1\abs{x-y}^{-a}\sum_{m=1}^n \binom{n-1}{m-1}C^{n-m}\tilde{\kappa}^{m-1}B^{(2\gamma-1)\frac{n}{2}\log^p\left(\frac{n+2}{2}\right)}\frac{ n^{m-2}\log^{p(m-2)}(\frac{n+2}{2})}{(m-2)!},
\end{align*}
where we have used \eqref{eq:eta_bound_non} for the second inequality and denoted  $(-1)!=1$ and {$\tilde{\kappa}$ might have changed between the steps}. Since $$\frac{n^{m-2}\log^{p(m-2)}(\frac{n+2}{2})}{(m-2)!}\leq \frac{n^{n-2}\log^{p(n-2)}(\frac{n+2}{2})}{(n-2)!}$$ for all $m=1,\ldots,n$ and $\sum_{m=1}^n \binom{n-1}{m-1} \leq 2^n$, there exists a constant $c_2\geq 2(C\vee \tilde{\kappa})$ such that the right-hand side above can be further bounded by
\[
c_1\abs{x-y}^{-a}c_2^n B^{(2\gamma-1)\frac{n}{2}\log^p\left(\frac{n+2}{2}\right)}\frac{ n^{n-2}\log^{p(n-2)}(\frac{n+2}{2})}{(n-2)!}.
\]
By Stirling's formula we have that $\frac{n^{n-2}}{(n-2)!}\leq e^n$. Hence, there exists $c_3>0$ such that
\begin{align*}
\sum_{n=1}^{2\Delta} \P_{\x,\y}(\tilde{B}_n^{(\x,\y)}) &\leq c_1\abs{x-y}^{-a} \sum_{n=1}^{2\Delta} c_3^n e^{p(n-2)\log\log\left(\frac{n+2}{2}\right)}B^{(2\gamma-1)\frac{n}{2}\log^p\left(\frac{n+2}{2}\right)}\\
&\leq c_1\abs{x-y}^{-a} 2\Delta c_3^{2\Delta} e^{p(2\Delta-2)\log\log\left(\Delta+1\right)}B^{(2\gamma-1)\Delta\log^p\left(\Delta+1\right)}.
\end{align*}
We can see that $B^{(2\gamma-1)\Delta\log^p(\Delta+1)}$ dominates the right-hand side in the sense that there exist constants $c_4,c_5>0$ such that
\begin{align*}
\sum_{n=1}^{2\Delta} \P_{\x,\y}(\tilde{B}_n^{(\x,\y)}) &\leq c_4\abs{x-y}^{-a}B^{c_5\Delta\log^p(\Delta+1)}\\
&= c_4 \exp\big(c_5\log(B)\Delta\log^p(\Delta+1) - \log(\abs{x-y}^a)\big).
\end{align*}
We now set
\[
\Delta\leq \frac{\log(\abs{x-y}^a)}{c_5\log(B)(\log\log(\abs{x-y}^a))^p}-1.
\]
Then, we have that 
\begin{align*}
c_5\log(B)\Delta\log^p(\Delta+1)& - \log(\abs{x-y}^a)\\
& \leq \log(\abs{x-y}^a) \big(1-\tfrac{p\log\log\log(\abs{x-y}^a)}{\log\log(\abs{x-y}^a)}\big)^p - \log(\abs{x-y}^a).
\end{align*}
A second order Taylor expansion shows that the right-hand side converges to $-\infty$ as $\abs{x-y} \to \infty$. Hence, for such a choice of $\Delta$, we have 
$\P_{\x,\y}\set{\mathd(\x,\y) \leq 2\Delta} \leq \varepsilon + o(1)$
%as $\abs{x-y}\to \infty$ 
which implies the statement of Proposition \ref{prop:smlthm}. 

\section{Proof of the upper bound for the chemical distance}

To prove the upper bound for the chemical distance, we show the following proposition.\medskip

\begin{prop}\label{prop:uppthm}
Suppose Assumption 1.2 holds for  $\gamma>\frac{\delta}{\delta + 1}$. Then for any  vertex~$\x$ there exists a path with no more than 
$$(4 + o(1)) \frac{\log\log \abs{x}}{\log\big(\frac{\gamma}{\delta(1-\gamma)}\big)}$$ 
vertices connecting $\0$ and $\x$, with high probability under $\P_{\0,\x}(\phantom{i}\cdot\phantom{i} \mid \0 \leftrightarrow \x )$  as $\abs{x} \to \infty$.
\end{prop}
\medskip

To prove this result, we rely on a strategy introduced in \cite{HirM2020}. Since the vertices of~$\mathscr{G}$ are given by the points of a Poisson process, the most powerful vertex inside a box with volume of order \smash{$\abs{x}^d$} around the midpoint between $\0$ and $\x$  typically has a mark smaller than $\abs{x}^{-d}\log\abs{x}$. Hence, it is sufficient to construct a short enough path from $\0$ resp.\ $\x$ to this most powerful vertex inside the box. Here, as in Section \ref{subsec:outline}, the typical connection type between two powerful vertices is crucial. For $\gamma>\frac{\delta}{\delta+1}$ we expect two powerful vertices to be connected via a vertex with larger mark, which we again call a \emph{connector}. In fact, the following lemma shows that for a powerful vertex with mark $t$ and a suitable vertex with a sufficiently smaller mark, the probability that there exist no connector which neighbours each of the two vertices is decaying exponentially fast as the mark $t$ gets small. 
This is a corollary of \cite{HirM2020} and follows with the same calculations as in \cite[Lemma 3.1]{GraLM2021}. We now fix for the rest of the section
\[
\alpha_1 \in \big(1,\tfrac{\gamma}{\delta(1-\gamma)}\big)\quad \text{and}\quad \alpha_2 \in \left(\alpha_1,\tfrac{\gamma}{\delta}(1+\alpha_1\delta)\right),
\]
noting that our assumptions ensure that the intervals are nonempty.
\begin{lemma}\label{lem:connector_existence}
There exists $c>0$ such that for two given vertices $\x = (x,t), \y=(y,s) \in \X$ with $t,s\leq \frac{1}{4}$, $s\leq t^{\alpha_1}$ and $\abs{x-y}^d\leq t^{-\alpha_2}$ we have
\[
\P_{\x,\y}\set{\x\overset{2}{\conn} \y} \geq 1 - \exp\big(-ct^{(\alpha_2-\alpha_1\gamma)\delta-\gamma}\big).
\]
\end{lemma}
\begin{proof}
We only consider connectors $\z = (z,u)$ with $u\geq \frac{1}{2}$ and {$\abs{x-z}<t^{-\frac\gamma{d}}$}. Then, by the thinning theorem \cite[Theorem 5.2]{LasP2017} and Assumption \ref{ass:main2} the number of such connectors is Poisson distributed with its mean bounded from below by
\begin{align*}
\int_{\frac{1}{2}}^1 {\int_{B_{t^{-\gamma/d}}(x)} }&{\alpha^2}\rho(\kappa^{-1/\delta})\rho(\kappa^{-1/\delta}s^\gamma\abs{y-z}^d)  \, dz \, du\\
&\geq \frac{{\alpha^2}\rho(\kappa^{-1/\delta})}{2}t^{-\gamma}\rho({\kappa^{-1/\delta}s^\gamma(t^{-\gamma/d}+\abs{x-y})^d}),
\end{align*} 
where $\rho(x) := 1\wedge x^{-\delta}$ as in the previous section.
As $\abs{x-y}^d < t^{-\alpha_2}$ and $s<t^{\alpha_1}$ this can be bounded from below by
${c} t^{(\alpha_2-\alpha_1\gamma)\delta-\gamma} $,
where \smash{$c = ({{\alpha^2}\rho(\kappa^{-1/\delta})\kappa}{2^{-(d\delta+1)}})\wedge$} \smash{$({\alpha^2}\rho(\kappa^{-1/\delta})/2)$}.
%which implies the result of the lemma.
\end{proof}

We now look into a box
$H(x) = \frac{x}{2}+[-2\abs{x},2\abs{x}]^d$
and
introduce a hierarchy of layers $L_1\subset L_2\subset \ldots \subset \mathcal{X} \cap H(x)\times (0,1)$ of vertices inside the box containing $\0$ and~$\x$. While the layer $L_1$ only contains vertices with very small mark, vertices with larger and larger  marks are included in layers with larger index. More precisely, as in~\cite{HirM2020} we set 
\[
L_k = \mathcal{X} \cap H(x)\times \big(0,(4\abs{x})^{-d\alpha_1^{-k}}\big]
\]
and
\[
K = \min\big\{k\geq 1 \colon (4\abs{x})^{-d\alpha_1^{-k}} \geq (\log (4\abs{x})^d)^{-\eta^{-1}}\big\}-1,
\]
where {$\eta = (\gamma - (\alpha_2-\alpha_1\gamma)\delta)\wedge (\alpha_2-\alpha_1) >0$}. {As the vertex set $\mathcal{X}$ is a Poisson process, by Lemma \ref{lem:connector_existence} for a given vertex in layer $L_{k+1}$ there exists with high probability a suitable vertex in layer $L_k$ such that  both vertices are connected via a connector with high probability. As in~\cite{HirM2020} and~\cite{JacM2017} we can use an estimate as  in Lemma~\ref{lem:connector_existence} to see that a vertex in $L_1$ is either the most powerful vertex in the box or connected to it via a connector, with high probability as $\abs{x} \to \infty$. Hence we get that $\diam(L_K) \leq 4K$.}\bigskip
\pagebreak[3]

Since $K$ is of order \smash{$(1+o(1))\frac{\log\log\abs{x}}{\log{\alpha_1}}$}, to finish the proof it suffices to show that the vertices~$\0$ and $\x$ are connected to the layer $L_K$ in fewer than $o(\log\log\abs{x})$ steps. To do so, we first show that~$\0$ (resp. $\x$) is connected to a vertex with sufficiently small mark and within distance smaller than $\abs{x}$ in finitely many steps. Then, we show that this vertex is connected to a vertex of $L_K$ in $o(\log\log\abs{x})$ steps. To keep the existence of these two paths sufficiently independent we rely on a sprinkling argument. For $b<1$ we assign independently to each vertex in $\X$ the color \emph{black} with probability $b$ and \emph{red} with probability $r=1-b$. Then, we denote by $\mathscr{G}^b$ the graph induced by restricting $\mathscr{G}$ to the black vertices and the edges between them. In the same way we define $\mathscr{G}^r$ for the black vertices. Note that $\mathscr{G}^r \cup \mathscr{G}^b$ is a subgraph of $\mathscr{G}$. \pagebreak[3]\bigskip

We use the black vertices to ensure the existence of the first part of the path in~$\mathscr{G}^b$. Thus, we define for $\0$ (and similarly for $\x$) the event \smash{$E^b(D,s,v)$} that there exists a black vertex $\z$ with mark smaller than $s$ and within distance shorter than $v$ such that there exists a path in $\mathscr{G}^b$ of length smaller $D$ between $\0$ and $\z$. Then, given $\z$, we use the red vertices to show that $\z$ is connected to the layer $L_K$ in sufficiently few steps. We denote by $L_k^r$ the restriction of $L_k$ to its red vertices. {Observe that we still have $\diam(L_K^r) \leq 4K$ in $\mathscr{G}^r$, as Lemma \ref{lem:connector_existence} restricted to $\mathscr{G}^r$ also holds if the constant $c$ is multiplied by $r$.} We define $F$ to be the event that $\z$ is connected by a path of length smaller than $o(\log\log\abs{x})$ to $L_K^r$ in $\mathscr{G}^r$.
{Note that the event $\0 \conn \x$ implies that with high probability  $\0$ and $\x$ are part of the unique infinite component $K_\infty$ of $\mathscr{G}$, since $\P_{\0,\x}(\set{\0 \conn \x}\backslash \set{\0,\x \in K_\infty})$ converges to zero as $\abs{x}\to \infty$. This is a consequence of the uniqueness of the infinite component $K_\infty$ as $\set{\0 \conn \x}\backslash \set{\0,\x \in K_\infty}$ implies that $\0 $ and $\x$ are part of the same finite component whose asymptotic proportion of vertices is zero.} {Thus, to prove Proposition \ref{prop:uppthm} %for the phase $\gamma>\frac{\delta}{\delta+1}$, 
it is sufficient to show that for any $s>0$ there exists a almost surely finite random variable $D(s)$ such that
\[
\lim_{b\nearrow 1} \liminf_{s\searrow 0} \liminf_{\abs{x}\to \infty} \P_\0\left(\set{\0 \in K_\infty^b} \cap E^b(D(s),s,\abs{x}) \cap F\right) \geq \theta
\]
where $\theta$ is the asymptotic proportion of vertices in the infinite component of $\mathscr{G}$ and $K_\infty^b$ is the infinite component of $\mathscr{G}^b$.} Note that, as \smash{$\gamma>\frac{\delta}{\delta + 1}$}, the critical percolation parameter of the graph $\mathscr{G}$ is $0$ by~\cite{GraLM2021}, and therefore $K_\infty^b$ exists and is unique. Now the probability above can be bounded from below by
\begin{align*}
&\P_\0\left(\set{\0 \in K_\infty^b} \cap E^b(D(s),s,\abs{x}) \cap F\right)\\
&\geq \P_\0\set{\0 \in K_\infty^b} - \P_\0\left(\set{\0 \in K_\infty^b}\backslash E^b(D(s),s,\abs{x})\right) - \P_\0\left( E^b(D(s),s,\abs{x})\backslash F\right)\\
&= \P_\0\set{\0 \in K_\infty^b} - \P_\0\left(\set{\0 \in K_\infty^b}\backslash E^b(D(s),s,\abs{x})\right) - \E_\0\left[(1-\P_\z(F\vert \mathscr{G}^b))\mathbf{1}_{E^b(D(s),s,\abs{x})}\right].
\end{align*}
We show in the following two lemmas that the last two terms converge to $0$ as $s\to 0$ and $\abs{x}\to \infty$ as in \cite{HirM2020}, which yields
\[
\liminf_{s\searrow 0} \liminf_{\abs{x}\to \infty} \P_\0\left(\set{\0 \in K_\infty^b} \cap E^b(D(s),s,\abs{x}) \cap F\right) \geq \theta_b,
\]
where $\theta_b$ is the asymptotic proportion of vertices in the infinite component of $\mathscr{G}^b$. As in \cite[Proposition 7]{JacM2017} it can be shown that the percolation probability $\theta_b$ is continuous in $b$ such that $\theta_b$ converges to $\theta$ as $b\nearrow 1$, which completes the proof.

\begin{lemma}\label{lem:upper_bound_firstpath}
Let $b,s>0$. Then, there exists an almost surely finite random variable $D(s)$ such that
\[
\lim_{\abs{x}\to \infty} \P_\0\left(\set{\0 \in K_\infty^b}\backslash E^b(D(s),s,\abs{x})\right) = 0.
\]
\end{lemma}
\begin{proof}
Let $E^b(D,s)$ be the event that there exists a black vertex $\z$ with mark smaller than $s$ which is connected to $\0$ in less than $D$ steps. If $\0 \in K_\infty^b$ there exists a path connecting $\0$ to at least one black vertex with mark smaller than $s$. This follows from the results in \cite{GraLM2021} where it is shown that vertices with arbitrarily small mark are contained in the infinite component $K_\infty^b$. In fact, the random variable
$D_\infty = \min\set{D:\ \text{the event}\ E^b(D,s)\ \text{occurs}}$ is finite. Hence, if $\abs{x}$ is large enough, $E^b(D_\infty,s,\abs{x})$ occurs if $\0 \in K_\infty^b$ and thus 
$
\lim_{\abs{x}\to \infty} \P_\0\left(\set{\0 \in K_\infty^b}\backslash E^b(D(s),s,\abs{x})\right) = 0
$.
\end{proof}

\begin{lemma}\label{lem:upper_bound_secondpath}
{Let $b>0$ and, on $E^b(D(s),s,\abs{x})$, 
denote by $\z$ the black vertex $(x_0,t_0)$ with $t_0<s$
within graph distance $D(s)$ from $\0$ in $\mathscr{G}^b$, which 
minimizes $|x_0|$. Then,
\[
\lim_{s\searrow 0}\limsup_{\abs{x}\to \infty} \E_\0\left[(1-\P_\z(F\vert {\mathscr G}^b))\mathbf{1}_{E^b(D(s),s,\abs{x})}\right] = 0.
\]}
\end{lemma}
\begin{proof}
Starting in $\z=(x_0,t_0)$ we want to find a red vertex $\x_1 = (x_1,t_1) \in \mathcal{X} \cap H(x)\times(0,1)$ with \smash{$\abs{x_0-x_1}^d\leq t_0^{-\alpha_2}$} and \smash{$t_1\leq t_0^{\alpha_1}$} which is connected to $\z$ via one connector. Since $\abs{x_0}\leq \abs{x}$, we have that $x_0 \in H(x)$. Note that the volume of the intersection of $H(x)$ and the ball \smash{$B_{t_0^{-\alpha_2/d}}(x_0)$} is a positive proportion of the ball volume. Hence, there exists $c>0$ such that the number of red vertices inside the box $H(x)$ with \smash{$\abs{x_0-x_1}^d\leq t_0^{-\alpha_2}$ and $t_1\leq t_0^{\alpha_1}$} is Poisson-distributed with parameter larger than $crt_0^{\alpha_1-\alpha_2}$ and thus the probability that such a vertex does not exist is bounded by 
\[
p_1 = \exp(-crt_0^{\alpha_1-\alpha_2}) + \exp\big(-crt_0^{(\alpha_2-\alpha_1\gamma)\delta-\gamma}\big),
\]
where the second summand is a consequence of Lemma \ref{lem:connector_existence} restricted to $\mathscr{G}^r$. Repeating this strategy, the same arguments yield that for a vertex $\x_{j-1} = (x_{j-1},t_{j-1})$ the probability that 
{there does not exist a connection to} a red vertex $\x_j = (x_j,t_j)$ inside $H(x)$ with 
\smash{$\abs{x_{j-1}-x_j}^d\leq t_{j-1}^{-\alpha_2}$} and $t_j\leq t_{j-1}^{\alpha_1}$ is bounded by
\[
p_j = \exp(-crt_{j-1}^{\alpha_1-\alpha_2}) + \exp\big(-crt_{j-1}^{(\alpha_2-\alpha_1\gamma)\delta-\gamma}\big).
\]
As $\eta = (\gamma - (\alpha_2-\alpha_1\gamma)\delta)\wedge (\alpha_2-\alpha_1)$ and $t_j\leq t_{j-1}^{\alpha_1}$, the right-hand side can be further bounded such that 
\smash{$p_j\leq 2\exp(-ct_0^{_{{-\eta\alpha_1^{j-1}}}}).$} Applying a union bound, the probability of failing to reach $L_K^r$ from $\z$ is bounded by $$2\sum_{j=1}^\infty \exp\Big(-cs^{-\eta\alpha_1^{j-1}}\Big),$$ which converges to $0$ as $s\searrow 0$, as shown in \cite[Lemma A.4]{HirM2020}. As  \smash{$t_j\leq t_0^{\alpha^j}$}, it takes at most $O(\log\log\log\abs{x})$ iterations of this strategy to arrive to a red vertex inside $H(x)$ with mark smaller than $(\log \abs{x})^{-\eta^{-1}}$. This completes the proof.
\end{proof}

\appendix

\section{Further calculations for the ultrasmall regime}

\begin{lemma}\label{cal_lem_1}
Let $x,y\in \mathbb{R}^d$, $t,s\in (0,1]$ and $\ell>0$ with $\ell<t\vee s$. For $\gamma > \frac{\delta}{\delta+1}$, 
\begin{align*}
\int\limits_{t\vee s}^1  \mathd u \int\limits_{\mathbb{R}^d} \mathd z & \rho\big(\kappa^{-1/\delta} t^{\gamma}u^{\gamma/\delta}\abs{x-z}^{d}\big)\rho\big(\kappa^{-1/\delta} s^{\gamma}u^{\gamma/\delta}\abs{y-z}^{d}\big)\\
\leq& \, \, \tilde{c} \ell^{1-\gamma-\gamma/\delta} \kappa(t\wedge s)^{-\gamma \delta}(t\vee s)^{-\gamma}\abs{x-y}^{-d\delta},
\end{align*}
where $\tilde{c}=\frac{2^{d \delta +1}I_\rho}{(\gamma+\gamma/\delta-1)} \vee 1$.\pagebreak[3]
\end{lemma}
\pagebreak[3]

\begin{proof}
Assume $t<s$, then we have
\begin{align*}
\int\limits_s^1 \mathd u \, &\int\limits_{\mathbb{R}^d} \mathd z \rho\big(\kappa^{-1/\delta} t^{\gamma}u^{\gamma/\delta}\abs{x-z}^{d}\big)\rho\big(\kappa^{-1/\delta} s^{\gamma}u^{\gamma/\delta}\abs{y-z}^{d}\big)\\
\leq& \int\limits_\ell^1 \mathd u \, \rho\big(2^{-d}\kappa^{-1/\delta} t^{\gamma}u^{\gamma/\delta}\abs{x-y}^{d}\big)\int\limits_{\mathbb{R}^d} \mathd z \,\rho\big(\kappa^{-1/\delta} s^{\gamma}u^{\gamma/\delta}\abs{y-z}^{d}\big)\\
&\phantom{do}+ \int\limits_\ell^1 \mathd u \, \rho\big(2^{-d}\kappa^{-1/\delta} s^{\gamma}u^{\gamma/\delta}\abs{x-y}^{d}\big)\int\limits_{\mathbb{R}^d} \mathd z \,\rho\big(\kappa^{-1/\delta} t^{\gamma}u^{\gamma/\delta}\abs{x-z}^{d}\big)\\
\leq& I_\rho 2^{d \delta}\kappa \Big[t^{-\gamma\delta}s^{-\gamma}\abs{x-y}^{-d \delta}\int\limits_\ell^1 \mathd u \, u^{-\gamma-\gamma/\delta} + \int\limits_\ell^1 \mathd u \, t^{-\gamma}s^{-\gamma\delta} \abs{x-y}^{-d \delta} u^{-\gamma-\gamma/\delta}\Big]
\\
\leq& \tfrac{I_\rho 2^{d \delta}}{\gamma+\gamma/\delta-1}\kappa \ell^{1-\gamma-\gamma/\delta} \Big[t^{-\gamma\delta}s^{-\gamma}\abs{x-y}^{-d \delta} + s^{-\gamma\delta}t^{-\gamma}\abs{x-y}^{-d \delta}\Big]\\
\leq& \tfrac{I_\rho 2^{d \delta+1}}{\gamma+\gamma/\delta-1}\kappa \ell^{1-\gamma-\gamma/\delta} t^{-\gamma\delta}s^{-\gamma}\abs{x-y}^{-d \delta},
\end{align*}
where we used for the first inequality that, for $z\in \mathbb{R}^d$, either $\abs{x-z}$ or $\abs{y-z}$ is larger than $\frac{\abs{x-y}}{2}$ and for the third inequality that $\gamma > \frac{\delta}{\delta+1}$ implies $\gamma+\gamma/\delta-1 > 0$.
\end{proof}

\begin{lemma}\label{cal_lem_2}
Let $x,y\in \mathbb{R}^d$, $t,s\in (0,1]$ and $\frac{1}{e}>\ell>0$ with $\ell<t\vee s$. For $\gamma>\frac{\delta}{\delta+1}$, 
\begin{align*}
\int\limits_{t\vee s}^1 \mathd u \int\limits_{\mathbb{R}^d} \mathd z \,  & \rho(\kappa^{-1/\delta} t^{\gamma}u^{\gamma/\delta}\abs{x-z}^{d})\rho(\kappa^{-1/\delta}s^{\gamma}u^{1-\gamma}\abs{y-z}^d)\\
\leq& \, \tilde{c} \log(\ell^{-1}) \kappa(t\wedge s)^{-\gamma \delta}(t\vee s)^{-\gamma}\abs{x-y}^{-d\delta},
\end{align*}
where $\tilde{c}=\frac{I_\rho 2^{d \delta + 1}}{(\delta - 1)(\gamma+\gamma/\delta-1)\wedge 1}$.
\end{lemma}
\begin{proof}
Since, for $z\in \mathbb{R}^d$, either $\abs{x-z}$ or $\abs{y-z}$ is larger than \smash{$\frac{\abs{x-y}}{2}$}, we have
\begin{align*}
\int\limits_{t\vee s}^1 &  \mathd u \int\limits_{\mathbb{R}^d} \mathd z \,  \rho(\kappa^{-1/\delta} t^{\gamma}u^{\gamma/\delta}\abs{x-z}^{d})\rho(\kappa^{-1/\delta}s^{\gamma}u^{1-\gamma}\abs{y-z}^d)\\
\leq& \int\limits_{\ell}^1 \mathd u \,  \rho(2^{-d} \kappa^{-1/\delta} t^{\gamma}u^{\gamma/\delta}\abs{x-y}^{d}) \int\limits_{\mathbb{R}^d} \mathd z \, \rho(\kappa^{-1/\delta}s^{\gamma}u^{1-\gamma}\abs{y-z}^d)\\
&+ \int\limits_{\ell}^1 \mathd u  \, \rho(2^{-d}\kappa^{-1/\delta}s^{\gamma}u^{1-\gamma}\abs{x-y}^d)\int\limits_{\mathbb{R}^d} \mathd z \, \rho(\kappa^{-1/\delta} t^{\gamma}u^{\gamma/\delta}\abs{x-z}^{d})\\
\leq& \, I_\rho 2^{d \delta}\kappa \bigg[t^{-\gamma \delta} s^{-\gamma}\abs{x-y}^{-d \delta} \int\limits_{\ell}^1 \mathd u \,  u^{-1} +  s^{-\gamma\delta} t^{-\gamma}\abs{x-y}^{-d \delta} \int\limits_{\ell}^1 \mathd u  \, u^{-\gamma/\delta + (\gamma-1)\delta}\bigg].
\end{align*}
As \smash{$\gamma > \frac{\delta}{\delta + 1}$} and $\delta>1$, we have $-\gamma/\delta + (\gamma-1)\delta>-1$. Hence, the last expression can be further bounded by 
\begin{align*}
I_\rho 2^{d \delta} & \kappa \big[\log(\ell^{-1})t^{-\gamma \delta}s^{-\gamma} \abs{x-y}^{-d \delta} + \tfrac{1}{(\delta-1)(\gamma+\gamma/\delta-1)} s^{-\gamma\delta} t^{-\gamma}\abs{x-y}^{-d \delta}\big]\\
\leq& \tfrac{I_\rho 2^{d \delta+1}}{(\delta-1)(\gamma+\gamma/\delta-1)\wedge 1} \log(\ell^{-1}) \kappa (t\wedge s)^{-\gamma\delta} (t\vee s)^{-\gamma} \abs{x-y}^{-d \delta},
\end{align*}
since $\log(\ell^{-1})>1$.
\end{proof}
\bigskip

{\bf Acknowledgment:} This research was supported by Deutsche Forschungsgemeinschaft (DFG) as Project Number~425842117. No data or code was created as part of this research. We would also like to thank
the anonymous referees for valuable comments which led to significant improvements in the paper.


\begin{thebibliography}{10}

\bibitem{AntP1996}
{\sc Antal, P., and Pisztora, A.}
\newblock On the chemical distance for supercritical {B}ernoulli percolation.
\newblock {\em Ann. Probab. 24}, 2 (1996), 1036--1048.

\bibitem{BarA99}
{\sc Barab\'asi, A.-L., and Albert, R.}
\newblock Emergence of scaling in random networks.
\newblock {\em Science 286}, 5439 (1999), 509--512.

\bibitem{BenB2001}
{\sc Benjamini, I., and Berger, N.}
\newblock The diameter of long-range percolation clusters on finite cycles.
\newblock {\em Random Structures Algorithms 19}, 2 (2001), 102--111.

\bibitem{Bis2004}
{\sc Biskup, M.}
\newblock On the scaling of the chemical distance in long-range percolation
  models.
\newblock {\em Ann. Probab. 32}, 4 (2004), 2938--2977.

\bibitem{BisL2019}
{\sc Biskup, M., and Lin, J.}
\newblock Sharp asymptotic for the chemical distance in long-range percolation.
\newblock {\em Random Structures Algorithms 55}, 3 (2019), 560--583.

\bibitem{BolJR2007}
{\sc Bollob\'{a}s, B., Janson, S., and Riordan, O.}
\newblock The phase transition in inhomogeneous random graphs.
\newblock {\em Random Structures Algorithms 31}, 1 (2007), 3--122.

\bibitem{BriKL2018}
{\sc Bringmann, K., Keusch, R., and Lengler, J.}
\newblock Average distance in a general class of scale-free networks with
  underlying geometry.
\newblock arXiv:1602.05712, 2018.

\bibitem{CerP2012}
{\sc Cern\'{y}, J., and Popov, S.}
\newblock On the internal distance in the interlacement set.
\newblock {\em Electron. J. Probab. 17\/} (2012), no. 29, 25.

\bibitem{DeivHH2013}
{\sc Deijfen, M., van~der Hofstad, R., and Hooghiemstra, G.}
\newblock Scale-free percolation.
\newblock {\em Ann. Inst. Henri Poincar\'{e} Probab. Stat. 49}, 3 (2013),
  817--838.

\bibitem{DepW2019}
{\sc Deprez, P., and W\"{u}thrich, M.~V.}
\newblock Scale-free percolation in continuum space.
\newblock {\em Commun. Math. Stat. 7}, 3 (2019), 269--308.

\bibitem{DerMM2012}
{\sc Dereich, S., M\"{o}nch, C., and M\"{o}rters, P.}
\newblock Typical distances in ultrasmall random networks.
\newblock {\em Adv. in Appl. Probab. 44}, 2 (2012), 583--601.

\bibitem{DinS2015}
{\sc Ding, J., and Sly, A.}
\newblock Distances in critical long range percolation.
\newblock arXiv:1303.3995, 2015.

\bibitem{DomHH2010}
{\sc Dommers, S., van~der Hofstad, R., and Hooghiemstra, G.}
\newblock Diameters in preferential attachment models.
\newblock {\em J. Stat. Phys. 139}, 1 (2010), 72--107.

\bibitem{DreRB2014}
{\sc Drewitz, A., R\'{a}th, B., and Sapozhnikov, A.}
\newblock On chemical distances and shape theorems in percolation models with
  long-range correlations.
\newblock {\em J. Math. Phys. 55}, 8 (2014), 083307, 30.

\bibitem{GarM2004}
{\sc Garet, O., and Marchand, R.}
\newblock Asymptotic shape for the chemical distance and first-passage
  percolation on the infinite {B}ernoulli cluster.
\newblock {\em ESAIM Probab. Stat. 8\/} (2004), 169--199.

\bibitem{GarM2007}
{\sc Garet, O., and Marchand, R.}
\newblock Large deviations for the chemical distance in supercritical
  {B}ernoulli percolation.
\newblock {\em Ann. Probab. 35}, 3 (2007), 833--866.

\bibitem{GraGLM2019}
{\sc Gracar, P., Grauer, A., L\"{u}chtrath, L., and M\"{o}rters, P.}
\newblock The age-dependent random connection model.
\newblock {\em Queueing Syst. 93}, 3-4 (2019), 309--331.

\bibitem{GraHMM2019}
{\sc Gracar, P., Heydenreich, M., M\"onch, C., and M\"orters, P.}
\newblock Recurrence versus transience for weight-dependent random connection
  models.
\newblock arXiv:1911.04350, 2019.

\bibitem{GraLM2021}
{\sc Gracar, P., L\"{u}chtrath, L., and M\"{o}rters, P.}
\newblock Percolation phase transition in weight-dependent random connection
  models.
\newblock {\em Adv. Appl. Prob.\/} (2021), (to appear).

\bibitem{GriM1990}
{\sc Grimmett, G.~R., and Marstrand, J.~M.}
\newblock The supercritical phase of percolation is well behaved.
\newblock {\em Proc. Roy. Soc. London Ser. A 430}, 1879 (1990), 439--457.

\bibitem{HaoH2021}
{\sc Hao, N., and Heydenreich, M.}
\newblock Graph distances in scale-free percolation: the logarithmic case.
\newblock arXiv:2105.05709, 2021.

\bibitem{HilU2021}
{\sc Hil\'ario, M., and Ungaretti, D.}
\newblock Euclidean and chemical distances in ellipses percolation.
\newblock arXiv:2103.09786, 2021.

\bibitem{Hir2017}
{\sc Hirsch, C.}
\newblock From heavy-tailed {B}oolean models to scale-free {G}ilbert graphs.
\newblock {\em Braz. J. Probab. Stat. 31}, 1 (2017), 111--143.

\bibitem{HirM2020}
{\sc Hirsch, C., and M\"{o}nch, C.}
\newblock Distances and large deviations in the spatial preferential attachment
  model.
\newblock {\em Bernoulli 26}, 2 (2020), 927--947.

\bibitem{JacM2015}
{\sc Jacob, E., and M\"{o}rters, P.}
\newblock Spatial preferential attachment networks: power laws and clustering
  coefficients.
\newblock {\em Ann. Appl. Probab. 25}, 2 (2015), 632--662.

\bibitem{JacM2017}
{\sc Jacob, E., and M\"{o}rters, P.}
\newblock Robustness of scale-free spatial networks.
\newblock {\em Ann. Probab. 45}, 3 (2017), 1680--1722.

\bibitem{LasP2017}
{\sc Last, G., and Penrose, M.}
\newblock {\em Lectures on the Poisson Process}.
\newblock Cambridge University Press, 2017.

\bibitem{NorR2006}
{\sc Norros, I., and Reittu, H.}
\newblock On a conditionally {P}oissonian graph process.
\newblock {\em Adv. in Appl. Probab. 38}, 1 (2006), 59--75.

\bibitem{TeiU2017}
{\sc Teixeira, A., and Ungaretti, D.}
\newblock Ellipses percolation.
\newblock {\em J. Stat. Phys. 168}, 2 (2017), 369--393.

\bibitem{vdH}
{\sc van~der Hofstad, R.}
\newblock {\em Random Graphs and Complex Networks II}.
\newblock Cambridge University Press, to appear in 2022.

\bibitem{vdHHZ2007}
{\sc van~der {H}ofstad, R., Hooghiemstra, G., and Znamenski, D.}
\newblock Distances in random graphs with finite mean and infinite variance
  degrees.
\newblock {\em Electron. J. Probab. 12\/} (2007), no. 25, 703--766.

\end{thebibliography}
\end{document}